\documentclass[11pt,english]{article}
\usepackage[T1]{fontenc}
\usepackage[utf8]{inputenc}
\usepackage{csquotes}
\usepackage{geometry}
\PassOptionsToPackage{obeyFinal}{todonotes}
\usepackage{enumitem}
\usepackage{color}
\usepackage[english]{babel}

\usepackage{mathrsfs}
\usepackage{mathtools}
\usepackage{amsmath}
\usepackage{amsthm}
\usepackage{amssymb,pdfsync}
\usepackage{esint}
\usepackage{microtype}
\usepackage[colorlinks=true,pdfpagemode=UseNone,urlcolor=blue,linkcolor=blue,citecolor=blue,breaklinks=true]{hyperref}
\usepackage{lmodern}
\usepackage{textcomp}
\newcommand{\email}[1]{{\href{mailto:#1}{\nolinkurl{#1}}}}
\usepackage{zref-clever}
\zcsetup{
    abbrev=true, %
    cap=true,
    nameinlink=false,
	rangetopair=false,
	endrange=stripprefix,
	sort=false,
	reffont=\upshape,
	lastsep={, and },
	comp=false
    }

\zcRefTypeSetup{equation}{
Name-sg = ,
name-sg = ,
Name-pl = ,
name-pl = ,
refbounds = {\textup{(},,,\textup{)}},
+refbounds-rb = {\textup{(},,,},
+refbounds-re = {,,,\textup{)}},
rangesep = {--}
}

\numberwithin{equation}{section}
\numberwithin{figure}{section}
\numberwithin{table}{section}
\theoremstyle{plain}
\newtheorem{thm}{\protect\theoremname}[section]
\theoremstyle{definition}
\newtheorem{defn}[thm]{\protect\definitionname}
\theoremstyle{plain}
\newtheorem{prop}[thm]{\protect\propositionname}
\theoremstyle{remark}
\newtheorem{rem}[thm]{\protect\remarkname}
\theoremstyle{plain}
\newtheorem{lem}[thm]{\protect\lemmaname}
\theoremstyle{plain}
\newtheorem{cor}[thm]{\protect\corollaryname}

\providecommand{\corollaryname}{Corollary}
\providecommand{\definitionname}{Definition}
\providecommand{\lemmaname}{Lemma}
\providecommand{\propositionname}{Proposition}
\providecommand{\remarkname}{Remark}
\providecommand{\theoremname}{Theorem}
\zcsetup{countertype={thm=theorem}}
\AddToHook{env/prop/begin}{\zcsetup{countertype={thm=proposition}}}
\AddToHook{env/lem/begin}{\zcsetup{countertype={thm=lemma}}}
\AddToHook{env/rem/begin}{\zcsetup{countertype={thm=remark}}}
\AddToHook{env/cor/begin}{\zcsetup{countertype={thm=corollary}}}
\AddToHook{env/defn/begin}{\zcsetup{countertype={thm=definition}}}
\let\cref\zcref
\begin{document}
\global\long\def\supp{\operatorname{supp}}%

\global\long\def\Uniform{\operatorname{Uniform}}%

\global\long\def\Lip{\operatorname{Lip}}%

\global\long\def\Bernoulli{\operatorname{Bernoulli}}%

\global\long\def\dif{\mathrm{d}}%

\global\long\def\e{\operatorname{e}}%

\global\long\def\ii{\operatorname{i}}%

\global\long\def\Cov{\operatorname{Cov}}%

\global\long\def\esssup{\operatorname*{ess\,sup}}%

\global\long\def\Var{\operatorname{Var}}%

\global\long\def\pv{\operatorname{p.v.}}%

\global\long\def\e{\mathrm{e}}%

\global\long\def\p{\mathrm{p}}%

\global\long\def\st{\text{\ :\ }}%

\global\long\def\q{\mathrm{q}}%

\global\long\def\R{\mathbb{R}}%
\global\long\def\Z{\mathbb{Z}}%
\global\long\def\E{\mathbb{E}}%
\global\long\def\EE{\mathrm{E}}%

\global\long\def\Law{\operatorname{Law}}%

\global\long\def\supp{\operatorname{supp}}%

\global\long\def\image{\operatorname{image}}%

\global\long\def\dif{{d}}%

\global\long\def\eps{\varepsilon}%

\global\long\def\sgn{\operatorname{sgn}}%

\global\long\def\tr{\operatorname{tr}}%

\global\long\def\Hess{\operatorname{Hess}}%

\global\long\def\Re{\operatorname{Re}}%

\global\long\def\Im{\operatorname{Im}}%

\global\long\def\Dif{\operatorname{D}}%

\global\long\def\divg{\operatorname{div}}%

\newcommand{\red}{\textcolor{red}}
\newcommand{\disp}{\displaystyle}

\newcommand{\bv}{{\bf v}}
\newcommand{\cZ}{{\mathcal Z}}
\newcommand{\cX}{{\mathcal X}}
\newcommand{\farc}{\frac}
\newcommand{\Rm}{{\mathbb R}}
\newcommand{\Zm}{{\mathbb Z}}
\newcommand{\Pm}{{\mathbb P}}
\newcommand{\Nm}{{\mathbb N}}
\makeatletter
\def\blfootnote{\gdef\@thefnmark{}\@footnotetext}
\makeatother

\title{Stationary solutions to the stochastic Burgers equation on the line}
\author{Alexander Dunlap\thanks{Department of Mathematics, Stanford University, 450 Jane Stanford Way, Building 380, Stanford, CA, 94305, USA}\protect\phantom{\footnotesize *}\textsuperscript{,}\thanks{Current address: Department of Mathematics, Duke University, 120 Science Dr, Durham, NC, 27708, USA} \and Cole Graham\footnotemark[1]\protect\phantom{\footnotesize *}\textsuperscript{,}\thanks{Current address: Department of Mathematics, University of Wisconsin,  480 Lincoln Drive, Madison, WI, 53706, USA} \and Lenya Ryzhik\footnotemark[1]}
\maketitle
\begin{abstract}
We consider invariant measures for the stochastic Burgers equation on $\mathbb{R}$, forced
by the derivative of a spacetime-homogeneous Gaussian noise that is white in time and smooth in
space. An invariant measure is indecomposable, or extremal, 
if it cannot be represented as a convex combination 
of other invariant measures. We show that for each~$a\in\R$, there is a unique indecomposable law of a spacetime-stationary solution with mean~$a$, in a suitable function space. We also show that solutions starting from spatially-decaying perturbations of mean-$a$ periodic functions converge in law to the extremal space-time
stationary solution with mean~$a$ as time goes to infinity.
\end{abstract}
\blfootnote{{\it Email addresses:} \email{alexander.dunlap@duke.edu}, \email{graham@math.wisc.edu}, \email{ryzhik@stanford.edu}}

\section{Introduction}

\subsection*{The stochastic Burgers equation on the line}

We consider strong solutions $u(t,x)$ to the stochastic Burgers
equation written formally as
\begin{equation}\label{eq:uPDE}
  \partial_{t}u+\frac{1}{2}\partial_{x}(u^{2})=\frac{1}{2}\partial_{x}^{2}u+\partial_{x}\dot{V},\qquad t,x\in\mathbb{R}. 
\end{equation}
Here, the potential $\dot{V}$ is a spatial smoothing, by a symmetric
mollifier $\rho\in\mathcal{C}^{\infty}(\mathbb{R})\cap H^{1}(\mathbb{R})$,
of a space-time Gaussian white noise $\dot{W}$:
\begin{equation}\label{eq:Vdot}
\dot{V}(t,x)=(\rho*\dot{W})(t,x), 
\end{equation}
where
\[
  \mathbb{E}[\dot{W}(t,x)\dot{W}(t',x')]=\delta(t-t')\delta(x-x').
\]
In \cref{eq:Vdot} and throughout the paper, $*$ denotes spatial convolution. We will often use the notation $\rho^{*2}=\rho*\rho$.

To be more precise, let $(\Omega,\mathcal{F},\mathbb{P})$ be a standard
probability space and let $W=W(t,x)$ be a cylindrical Wiener process
on
$L^{2}(\mathbb{R})$ whose covariance operator is the identity. This is 
discussed, for example, in \cite[Section 4.3.1]{DPZ14}. 
Let $\{\mathcal{F}_t\}_{t\ge0}$ be the usual filtration corresponding 
to $W$, so that~$\mathcal{F}_t\subset\mathcal{F}$ is the $\sigma$-algebra 
generated by $W|_{[0,t]\times\R}$. We do not assume that $\mathcal{F}=\bigcup_{t\ge0}\mathcal{F}_t$: 
we will freely define additional random variables throughout the paper which are independent 
of the noise $W$, and will always assume that $\Omega$ is large enough to include such random 
variables. The
Itô time differential~$\dif W$ is thus a white noise on $\mathbb{R}\times\mathbb{R}$.
The random field 
$
V(t,x)=(\rho*W)(t,x)
$ is a Gaussian process on $\R\times\R$ with a continuous modification, which is in fact spatially smooth since $\rho$ is smooth. We
we will always assume that we are working with this modification, and refer the reader to \cref{subsec:sPDEtoPDE} below for more details.
We interpret the equation \cref{eq:uPDE} as
\begin{equation}\label{eq:uPDE-1}
  \dif u=\frac{1}{2}\big[\partial_{x}^{2}u-\partial_{x}(u^{2})\big]\dif t+\dif(\partial_{x}V),\qquad t,x\in\mathbb{R}. 
\end{equation}

The random Gaussian forcing
$V$ is not uniformly bounded in space, and so neither will be 
solutions to~\cref{eq:uPDE-1}. Indeed, this would be the case even without the nonlinear term 
in \cref{eq:uPDE-1}. Thus, one needs to work with \cref{eq:uPDE-1} in weighted function
spaces that permit spatial growth. To the best of our knowledge,
previous work has not considered the well-posedness
of \cref{eq:uPDE-1} in spaces allowing as much growth as 
we require, so let us first state the  existence and uniqueness result we will need. 
We denote by $\mathcal{X}_{m}$ 
the space of continuous functions on $\mathbb{R}$ growing at infinity
slower   than $|x|^{\ell}$ for all $\ell>m$, equipped
with an appropriately weighted topology described in \cref{subsec:functionspaces}.
Our first result is that the equation \cref{eq:uPDE-1} is well-posed
in $\mathcal{X}_{m}$ as long as $m<1$.
The restriction
$m<1$ is necessary: if solutions are allowed to grow
linearly, then characteristics starting at infinity may reach the origin in a finite
time, even in the absence of noise. This is  related (by the Cole--Hopf transform, discussed below) to the familiar restriction that the initial condition
for the standard heat equation should grow more slowly than~$\exp(|x|^2)$ at infinity. 

In the following theorem, we assume that the initial condition for \cref{eq:uPDE-1}
is continuous.
We will also need to consider discontinuous
initial data, but as the spaces involved become more complicated,
we defer the more general statement to \cref{prop:thetawellposed} 
in \cref{subsec:functionspaces}. %
\begin{thm}\label{thm:existence}
Let $m\in(0,1)$. With probability $1$, there is a map
\[
\Psi:\mathcal{X}_{m}\to\mathcal{C}_{\mathrm{loc}}([0,\infty);\mathcal{X}_{m})
\]
so that for each $v\in\mathcal{X}_m$, $u=\Psi(v)$
is the unique strong solution in $\mathcal{C}_{\mathrm{loc}}([0,\infty);\mathcal{X}_{m})$
to \cref{eq:uPDE-1} with~$u(0,x)=v(x)$.
The map $\Psi$ is continuous almost surely. Finally,
the semigroup
$P_{t}f(v)=\mathbb{E}[f(u(t,\cdot))]$
has the Feller property: if $f$ is a bounded continuous function
on $\mathcal{X}_{m}$, then so is $P_{t}f$ for any $t>0$.
\end{thm}

A proof of \cref{thm:existence}, as well as \cref{prop:thetawellposed}
handling discontinuous initial data, occupies \cref{sec:solntheory}.

\subsection*{Space-time stationary solutions: existence and stability}
 
Our main interest is in solutions to \cref{eq:uPDE-1}
that are statistically stationary  under both  the time evolution and
translations in space. 
We will need to consider invariant measures for  ensembles of solutions 
$\mathbf{u}=(u_{1},\ldots,u_{N})=\mathbf{u}(t,x)$ 
to~\cref{eq:uPDE-1}, satisfying %
\begin{equation}\label{eq:uPDE-many}
  \dif u_{i}=\frac{1}{2}[\partial_{x}^{2}u_{i}-\partial_{x}(u_{i}^{2})]\dif t+\dif(\partial_{x}V), \qquad t,x\in\mathbb{R},\qquad i=1,\ldots,N. 
\end{equation}
These equations are decoupled. The solutions $u_1,\ldots,u_N$ have different initial conditions but are all 
subject to the same noise. It may often be convenient for the reader
to think of the case $N=1$, which corresponds to a single initial condition. 
We will use the $N>1$ case to prove some statements for families of coupled solutions
that we will need, in particular, for the ordering results below.

Let us first define precisely what we mean by invariant
measures. Let $\mathscr{P}(\mathcal{X}_{m}^{N})$ be the space of
probability measures on $\mathcal{X}_{m}^{N}$, and for each
$\nu\in\mathscr{P}(\mathcal{X}_{m}^{N})$ and $t\ge 0$, let 
$P_{t}^{*}\nu=\Law(\mathbf{u}(t,\cdot))$. Here, $\mathbf{u}$
is a solution to \cref{eq:uPDE-many} with initial condition $\mathbf{u}(0,\cdot)\sim\nu$. 
When we consider such solutions, we always assume that the noise $V$ is independent of the random initial condition $\mathbf{u}(t,\cdot)$.
The set of invariant measures under \cref{eq:uPDE-many} is 
\[
\overline{\mathscr{P}}(\mathcal{X}_{m}^{N})=
\{\nu\in\mathscr{P}(\mathcal{X}_{m}^{N})\st P_{t}^{*}\nu=\nu\}.
\]
To formulate the spatial translation invariance, we first define, for
$x\in\mathbb{R}$ and $\mathbf{v}=(v_{1},\ldots,v_{N}):\mathbb{R}\to\mathbb{R}^{N}$,
 the translation operator on $\mathcal{X}_{m}^{N}$ as
\begin{equation}
  \tau_{x}\mathbf{v}(y)=(v_{1}(y-x),\ldots,v_{N}(y-x)),\label{eq:translationoperator}
\end{equation}
and the corresponding operator $(\tau_{x})_*$ on $\mathscr{P}(\mathcal{X}_{m}^{N})$.
For
$G=\mathbb{R}$
or $G=L\mathbb{Z}$ for some $L > 0$, we set %
\begin{equation}\label{jun2502}
\mathscr{P}_{G}(\mathcal{X}_{m}^{N})=
\{\nu\in\mathscr{P}(\mathcal{X}_{m}^{N})\st(\tau_{x})_{*}\nu=\nu\text{ for all }x\in G\},
\end{equation}
the space of probability measures on $\mathcal{X}^N_m$ that are invariant under the action of $G$. We also define the corresponding invariant measures under \cref{eq:uPDE-many}:
\[
  \overline{\mathscr{P}}_{G}(\mathcal{X}_{m}^{N})=\overline{\mathscr{P}}(\mathcal{X}_{m}^{N})\cap\mathscr{P}_{G}(\mathcal{X}_{m}^{N}).
\]
The space $\overline{\mathscr{P}}_{G}(\mathcal{X}_{m}^{N})$ is the space of 
invariant measures corresponding to ``space-time stationary solutions.'' Here, spatial
stationarity is understood as either with respect to all spatial translations if $G=\R$,
or as $L$-periodicity if $G=L\Z$. 

The space $\overline{\mathscr{P}}_{G}(\mathcal{X}_{m}^{N})$ is convex, and 
we denote by  $\overline{\mathscr{P}}_{G}^{\mathrm{e}}(\mathcal{X}_{m}^{N})$ 
the set of  its extremal elements. 
We note
that extremality  is equivalent to the ergodicity property that if
$A\subset\mathcal{X}_{m}^{N}$ is a Borel subset such that $\tau_{x}A=A$
for all $x\in G$ and $P_{t}\mathbf{1}_{A}=\mathbf{1}_{A}$ $\mu$-a.s.
for all $t\ge0$, then either $\mu(A)=0$ or $\mu(A)=1$. The corresponding 
equivalence for measures that are invariant under a Markov semigroup can be 
found in Theorem 5.1 of \cite{H08}. Our measures are invariant both under
a Markov semigroup and a group of translations. In that case, the
equivalence of extremality and ergodicity
can be proved using the corresponding statement for invariant measures under a set of maps 
in Proposition~12.4 of~\cite{Phe01}, 
along with the equivalence proved in Corollary 5.3 of~\cite{H08}.
On an intuitive level, both extremality  and ergodicity are ways to formalize  
indecomposability -- such invariant measures 
represent the ``building blocks'' of the possible long time behaviors.

To formulate our result on the existence, uniqueness and properties
of the extremal invariant measures, 
we will use an auxiliary random variable $X$, depending on $G$. 
If~$G=\R$, we let~$X=0$ a.s.,  and if $G=L\Z$ with some $L>0$, then  $X \sim \mathrm{Uniform}([0,L])$.
\begin{thm}\label{thm:maintheorem-classification}Fix $m\in[1/2,1)$, $N\in\mathbb{N}$, and $G=\R$ or $G=L\Z$ for some $L>0$. The random variable $X$ is independent of all other random variables.
For each $\mathbf{a}\in\mathbb{R}^{N}$,
there exists a unique $\nu_{\mathbf{a}}\in\overline{\mathscr{P}}_{G}^{\mathrm{e}}(\mathcal{X}_{m}^{N})$
such that if $\mathbf{v}=(v_{1},\ldots,v_{N})\sim\nu_{\mathbf{a}}$,
then $\mathbb{E}\mathbf{v}(X) = \mathbf{a}$ and $\mathbb{E} |\mathbf{v}(X)|^2 <\infty$.
Moreover, the measures $\nu_{\mathbf{a}}$ satisfy the following properties.
\begin{enumerate}[leftmargin = 1.5cm, label = \textup{(P\arabic*)}, ref = (\textup{P\arabic*})]
 \item\label{enu:ordered} Order: with probability one, we have 
 \[
\hbox{ $\sgn((v_{j}-v_{k})(x))=\sgn(a_{j}-a_{k})$,
 for each $j,k\in\{1,\ldots,N\}$ and all $x$
in $\mathbb{R}$.}
\]
\item\label{enu:shearinvariant} 
Shear invariance: if $c\in\mathbb{R}$, then $\nu_{\mathbf{a}+(c,\ldots,c)}=\Law(\mathbf{v}+(c,\ldots,c))$,
where $\mathbf{v}\sim\nu_{\mathbf{a}}$.
\item\label{enu:alsoinvariant} $G$-independence:
the measure $\nu_{\mathbf{a}}$ is also an element of $\overline{\mathscr{P}}_{G'}^{\mathrm{e}}(\mathcal{X}_{1/2}^{N})$
for $G'=\R$ and $G'=L'\Z$, for all $L'>0$.
\item\label{enu:smoothtempereddistn} If $\mathbf{v}\sim\nu_{\mathbf{a}}$ then with probability $1$, $\mathbf{v}$ is (spatially) smooth, and $\mathbf{v}$ and all of its derivatives grow at most polynomially at infinity.
  \item\label{enu:decomposesamemean} If $\nu_{\mathbf{a}}=(1-q)\mu_{0}+q\mu_{1}$ for some $q\in (0,1)$ and measures $\mu_{0},\mu_{1}\in\mathscr{P}_{G}(\mathcal{X}_{m}^{N})$,
    and $\mathbf{v}\sim\mu_{0}$, then $\mathbb{E}\mathbf{v}(x)=\mathbf{a}$.
  \end{enumerate}
\end{thm}
Properties \cref{enu:ordered} and \cref{enu:shearinvariant}
in the above theorem have a clear intuitive meaning:
space-time stationary solutions are ordered and have shear invariance.
Property \cref{enu:smoothtempereddistn} is simply a reflection of 
the spatial parabolic smoothing of the viscous Burgers equation. It does not, 
of course, extend to smoothness in time if $\mathbf{v}$ is allowed to evolve under the dynamics: 
this is precluded by the presence of the temporally-rough forcing $\dif V$.
As for  \cref{enu:alsoinvariant}, let us note that, a priori, the extremal property of an element 
of~$\overline{\mathscr{P}}_{G}(\mathcal{X}_{m}^{N})$ depends on the group~$G$ of translations. Even if
 $G' \subset G$ so that~$\overline{\mathscr{P}}_{G}(\mathcal{X}_{m}^{N}) \subset \overline{\mathscr{P}}_{G'}(\mathcal{X}_{m}^{N})$,  one might worry that
an extremal element of~$\overline{\mathscr{P}}_{G}(\mathcal{X}_{m}^{N})$ is not  necessarily extremal in $\overline{\mathscr{P}}_{G'}(\mathcal{X}_{m}^{N})$, since the
latter is potentially a larger set. Property \cref{enu:alsoinvariant} rules out this situation.
Finally, property \cref{enu:decomposesamemean} is a weak version of ergodicity under just $G$ (not under the dynamics): if $\mu\in\overline{\mathcal{P}}_G(\mathcal{X}^N_m)$ can be decomposed into multiple measures that are $G$-invariant, even if not $P_t$-invariant, then those measures must have the same mean.

We emphasize that the properties \cref{enu:ordered}--\cref{enu:decomposesamemean} 
in \cref{thm:maintheorem-classification} are not part of the uniqueness statement. That is, 
to know that an extremal invariant measure $\nu\in\overline{\mathscr{P}}^{\mathrm{e}}_G(\mathcal{X}^N_m)$ is equal to $\nu_\mathbf{a}$, 
we need only know that $\E \mathbf{v}(X)=\mathbf{a}$ and $\mathbb{E} |\mathbf{v}(X)|^2 <\infty$ for $\mathbf{v}\sim\nu$.
In particular, for any $\nu\in\overline{\mathscr{P}}_G^{\mathrm{e}}(\mathcal{X}^N_m)$ such that $\mathbb{E} |\mathbf{v}(X)|^2 <\infty$ for $\mathbf{v}\sim\nu$, there is an $\mathbf{a}\in\R^N$ so that $\nu=\nu_\mathbf{a}$.

Let us next discuss the long-time behavior of the dynamics \cref{eq:uPDE-many}.
We denote by 
$L^{\infty}(\mathbb{R}/L\mathbb{Z})$  the space of $L$-periodic  functions in $L^\infty(\mathbb{R})$. We will prove a stability result for initial conditions lying in spaces of functions that can be bounded above and below by periodic functions whose averages can be made arbitrarily close to each other. The following definition states this precisely. We will use the partial order
$\preceq$ on $\mathbb{R}^{N}$ defined by
\begin{equation}
  (x_{1},\ldots,x_{N})\preceq(y_{1},\ldots,y_{N})\iff\text{\ensuremath{x_{i}\le y_{i}} for each \ensuremath{i=1,\ldots,N}.}\label{eq:partialorder}
\end{equation}
Once again, it may be convenient for the reader to think of the case $N=1$. 
\begin{defn}\label{def:basinsofattraction}
For $\mathbf{a}\in\mathbb{R}^{N}$,
we denote by $\mathscr{B}_{\mathbf{a}}$ the set of all $\mathbf{v}\in L^\infty(\mathbb{R})^{N}$
such that for every $\eps>0$, there exists an $L_\eps\in(0,\infty)$ and
$\mathbf{v}_{-}^\eps,\mathbf{v}_{+}^\eps\in L^\infty(\mathbb{R}/L\mathbb{Z})^{N}$
so that $\mathbf{v}_{-}^\eps\preceq\mathbf{v}\preceq\mathbf{v}_{+}^\eps$ and
\begin{equation}\label{eq:basinofattractioncond}
\frac{1}{L_\eps}\int_0^{L_\eps}\mathbf{v}_+^\eps(x)\,\dif x
-(\eps,\ldots,\eps)\preceq \mathbf{a}\preceq
\frac{1}{L_\eps}\int_0^{L_\eps}\mathbf{v}_-^\eps(x)\,\dif x + (\eps,\ldots,\eps).
\end{equation}
\end{defn}
The following proposition gives a reasonably
general sufficient condition for a function to be in~$\mathscr{B}_{\mathbf{a}}$.
\begin{prop}\label{prop:stability-criterion}
Suppose that a function $\mathbf{v}\in L^\infty(\mathbb{R})^{N}$
can be written as $\mathbf{v}=\mathbf{v}_{\mathrm{per}}+\mathbf{v}_{\mathrm{int}}+\mathbf{v}_{\mathrm{z}}$,
where $\mathbf{v}_{\mathrm{per}}\in L^\infty(\mathbb{R}/L\mathbb{Z})^{N}$
for some $L\in(0,\infty)$, $\mathbf{v}_{\mathrm{int}}\in(L^{1}(\mathbb{R})\cap L^\infty (\mathbb{R}))^{N}$,
and $\mathbf{v}_{\mathrm{z}}\in L^\infty(\mathbb{R})^{N}$
is such that
\[
    \lim\limits _{|x|\to\infty}|\mathbf{v}_{\mathrm{z}}(x)|=0.
\]
 Then $\mathbf{v}\in\mathscr{B}_{\mathbf{a}}$, where $\mathbf{a}=\disp\frac{1}{L}\int_{0}^{L}\mathbf{v}_{\mathrm{per}}(x)\,\dif x$.
\end{prop}
We can now state our stability result.
\begin{thm}\label{thm:maintheorem-stability}
Let $m\in[1/2,1)$, 
$\mathbf{a}\in\mathbb{R}^{N}$, and $\mathbf{u}$ be a solution
  to \cref{eq:uPDE-many} with initial condition~$\mathbf{v}\in\mathscr{B}_{\mathbf{a}}$.
  Then we have
  \begin{equation}
    \lim_{t\to\infty}\Law(\mathbf{u}(t,\cdot))=\nu_{\mathbf{a}}\label{eq:convinlaw-1}
  \end{equation}
  in the sense of weak convergence of probability measures
  on $\mathcal{X}_{m}^{N}$.
\end{thm}

We
note that \cref{eq:convinlaw-1} can be upgraded to convergence as probability measures on spaces
of higher regularity using parabolic regularity estimates. Because
the norms involved become rather complicated, we direct the reader
to \cref{lem:mildregularitypoly} below.  

Results similar to ours were obtained by Bakhtin and Li in \cite{BL19},
using completely different methods. In that paper, the authors considered
\cref{eq:uPDE-1} with driving noise $V$ that is not a Wiener process
but rather a step process that jumps at integer times. This means
that the solution only feels ``kicks'' at integer times, rather
than white-in-time forcing. Their approach considers the question
from the point of view of directed polymers. In addition to what we
prove, they show that if the solution is started at a \emph{negative}
time $-T$, then as $T\to\infty$ the solution at time $0$ converges
 almost surely to a stationary initial condition (the
 one-force-one-solution principle). They also prove somewhat
larger basins of attraction than those described in \cref{def:basinsofattraction} (including, in particular, rarefaction waves).
However, their proof uses the properties of the kick forcing in a
serious way, and an adaptation to the white-in-time case is not clear.
Our work extends many of the results of \cite{BL19} to the white-in-time
setting, and provides a completely different, PDE-based perspective
on the problem.

The work \cite{BL19} is part of a two-decade-long program to understand
the attractors of the stochastic Burgers equation using the Lax--Oleinik
formula in the inviscid case or directed polymers in the viscous
case; see e.g. \cite{Bak16,BCK14,BL18,Bor18,EKMS00,IK03} and the
reviews \cite{BK18,BK07}. Part of the motivation for this program
is the goal of understanding the KPZ universality phenomenon, as the
equation is conjectured to lie in the KPZ universality class. We refer
to \cite{BK18} for more details and a fascinating discussion.

Our setting and PDE-based approach are closely related to those considered
by Boritchev in~\cite{Bor13sharp} on the one-dimensional torus, with
a multi-dimensional extension in \cite{Bor16multi}, and  a general review 
given in  \cite{Bor14Review}.
In particular, \cite{Bor13sharp} establishes the existence and uniqueness
of invariant measures for~\cref{eq:uPDE-1} on the torus. Existence
of such measures was previously shown in \cite{DPG94}; see also \cite{DPDT94} for the case when the noise is also white in space. As in the
present paper, the classification of stationary solutions in \cite{Bor13sharp}
is based on contractive properties of the Burgers equation. However,
\cite{Bor13sharp} uses $L^{1}$-contraction and a maximum principle
for $\partial_{x}u$ to establish a Doeblin-type condition and show
that all mean-$0$ solutions must converge to the unique mean-$0$
stationary solution. This relies on the compactness of the domain
in important ways. In the whole space, instead of using 
Poincar\'e inequality-type ideas, 
we show that any invariant measures~$\nu_{1},\nu_{2}\in\overline{\mathscr{P}}(\mathcal{X}_{m})$
have a coupling~$\nu\in\overline{\mathscr{P}}(\mathcal{X}_{m}^{2})$,
and, moreover, that if~$\mathbf{v}\sim\nu$, then the components of
$\mathbf{v}$ are ordered almost surely. This ordering allows us to
classify the laws of extremal stationary solutions. It is here that 
using $N>1$ in \cref{eq:uPDE-many} becomes crucial.

The stochastic Burgers equation with  unmollified spacetime
white noise, or the spatial gradient of spacetime white noise, has
also been the subject of significant interest in the literature. Much
of this work, such as~\cite{BCJL94,DPDT94,Gyo98,GN99,H11,HW13}, principally
concerns well-posedness for the equation, which is of course a more
difficult problem when the noise is spatially rough than when it is
smooth.
Well-posedness of the equation driven by the spatial gradient
of spacetime white noise is essentially the same problem as the
well-posedness of the KPZ equation driven by spacetime white noise,
as considered in, for example, \cite{BG97,GP17,Hai13,PR19}.
Ergodicity properties
for the stochastic Burgers equation with singular forcing on a compact domain are considered in
\cite{GP18,Ros19}.

In a different direction,
the papers \cite{BCJL94,BGN14,Gyo98,GN99,LN18,Unt13,Unt15,Unt17,PR19,ZZZ20} consider the
stochastic Burgers or KPZ equations on the whole space, but with initial conditions
and/or noise that are constrained to be growing more slowly than we need
to treat forcing by space-time stationary Gaussian fields. Most of these works 
assume that the initial condition and/or noise are in some $L^p(\R)$ space, which does not apply
to the space-stationary setting. The works \cite{BCJL94,PR19,ZZZ20} assume that the integral of 
the Burgers solution grows at most linearly at infinity; this will be true for our stationary 
solutions by Birkhoff's ergodic theorem, but we do not obtain the quantitative control on such 
growth that would be required to use results of this type. The work \cite{Unt13} considers the 
KPZ equation in spaces with ``locally bounded averages,'' a condition which again does not 
readily correspond to the estimates we obtain. The work \cite{Kim06} proves the existence
of invariant measures for the stochastic Burgers equation with non-gradient-type
noise but with a zero-order dissipation term to provide compactness
and remove the potentially growing low frequencies.

\subsection*{The Cole--Hopf transform, connection to the KPZ equation, and compactness}

In addition to the PDE arguments, the proof of the uniform bounds for the solutions
of the stochastic Burgers equation requires
one crucial application of the Feynman-Kac formula.
By the Cole--Hopf transform \cite{BCJL94,Col51,Hop50}
\begin{equation}
  h=-\log\phi,\qquad u=\partial_{x}h=-\partial_{x}\phi/\phi,\label{eq:colehopf}
\end{equation}
the stochastic Burgers equation \cref{eq:uPDE-1} is closely related
to the KPZ equation \cite{KPZ86}
\begin{equation}
  \dif h=\frac{1}{2}\Big[\partial_{x}^{2}h-(\partial_{x}h)^{2}+\|\rho\|_{L^{2}(\mathbb{R})}^{2}
  \Big]\dif t+\dif V\label{eq:KPZ}
\end{equation}
and the multiplicative stochastic heat equation
\begin{equation}
  \dif\phi=\frac{1}{2}\partial_{x}^{2}\phi-\phi \, \dif V,\label{eq:SHE}
\end{equation}
in which the last product is interpreted in an Itô sense. Here, $\rho(x)$ is the mollifier
in \cref{eq:Vdot}.
The fact that the results of the
transformation \cref{eq:colehopf} indeed satisfy the claimed PDEs is a computation
using It\^{o}'s formula (see e.g. \cite[Theorem~4.17]{DPZ14}). Note that, because
we work with noise that is spatially smooth, the Cole--Hopf transform
requires no infinite renormalization as in the white in time and space 
case~\cite{BG97,Hai13}, but simply the finite Itô correction given by
the term $\tfrac{1}{2}\|\rho\|_{L^{2}(\mathbb{R})}^{2}$ in~\cref{eq:KPZ},
which is half the derivative of the quadratic variation of
the process $t\mapsto V(t,x)$ for fixed $x$. The Cole--Hopf transform is a common tool in the study of the stochastic heat, KPZ, and Burgers equations.
In particular, the Cole--Hopf transform explains why it is natural to take the forcing in
\cref{eq:uPDE-1} to be the gradient of a random field, which is crucial
for the existence of space-time stationary solutions.

Because of the close relationship between \cref{eq:uPDE-1}, \cref{eq:KPZ},
and \cref{eq:SHE}, one might naïvely expect that stationary solutions
for one of the equations induce stationary solutions for the others.
However, this works only in one direction, because information is
lost when taking the spatial derivative to pass from $h$ to $u$.
That is, stationarity of $u$ does not imply stationarity of its antiderivative
$h$.
In one and two spatial dimensions, neither \cref{eq:KPZ} nor \cref{eq:SHE} is expected to admit stationary solutions, as the pointwise statistics of solutions started from constant initial conditions diverge.
The situation is different in three or more spatial dimensions, in which,
if the noise $V$ is sufficiently small, the multiplicative stochastic
heat equation admits nonzero stationary solutions \cite{DS80,DGRZ18,MSZ16,TZ98}.
In the low-dimensional case (or in higher dimensions with strong noise),
the KPZ evolution started at $0$ has a ``zero-frequency component''
whose variance diverges as $t\to\infty$. The goal of the present work
can thus be interpreted as showing that this component, which is eliminated
when we take a derivative and pass to the Burgers equation, is the
only obstruction to the existence of stationary solutions. One may
also see this as a Harnack-type property for $\phi$: although $\phi$
may not be stationary, the ratio $\partial_{x}\phi/\phi$ is stationary.

The relationship with the KPZ equation is important in our proof
strategy, as we now describe. To prove the existence of stationary solutions for the Burgers equation,
we first establish a form of compactness. The proof of this starts by taking expectations in
the KPZ equation \cref{eq:KPZ}. Since $u=\partial_{x}h$, the nonlinear
term in \cref{eq:KPZ} is~$u^{2}$, so second moments of solutions to
\cref{eq:uPDE-1} are related to the growth of~$\mathbb{E}h$. 
Asymptotically,~$\E h$ is $t$ times the Lyapunov exponent for the stochastic
heat equation \cref{eq:SHE}, corresponding to the linear-in-time drift in the
solution to the KPZ equation with white-noise forcing;
see for instance~\cite{ACQ11,BG99,FV05,SS10,Tsa18}.
More importantly for our purposes, $\E h$
can be shown to be increasing, as a function of time, using the Feynman--Kac
formula. We prove this in \cref{prop:gamma}, the only part of our work
that relies on the Feynman--Kac formula. At the moment, we do not know how to replace this use
of the Feynman-Kac formula by a purely PDE argument. With the second moment bound in hand,
we obtain a tightness statement that implies that $u$ converges along subsequences
of time-averaged laws of solutions to \cref{eq:uPDE-many}. Limits of such subsequences
can be shown by the Krylov--Bogoliubov theorem (see \cref{prop:krylovbogoliubov})
to be stationary in time.

While bounding the one-point variance of solutions
to \cref{eq:uPDE-1} is crucial to our proof, we do not say anything about the multipoint
correlations of solutions. It is expected that stationary solutions
to \cref{eq:uPDE-1} should have correlation functions that are integrable
in space, so that, when rescaled appropriately, the solutions approach
a white noise process. To our knowledge, this question, which is related
to KPZ universality, has not been resolved for the stochastic Burgers
equation with any kind of spatially smooth noise. In  \cite{FQ15,HQ18}, a different regularization of the spacetime-white-noise-forced Burgers equation is considered for
which this statement is clear.

\subsection*{Shear-invariance, ordering and \texorpdfstring{$L^1$}{L¹}-contraction
for the Burgers equation}

The three key ingredients to the classification of stationary solutions
are the shear invariance, ordering, and $L^1$-contraction properties of the Burgers equation.
All three  are analogues of well-known properties of the 
deterministic Burgers equation in the absence of random forcing.

The deterministic shear invariance simply says that if $u(t,x)$ is a solution
to \cref{eq:uPDE}
with $V=0$ then~$u_c(t,x)=u(t,x+ct)-c$ is also a solution, for any $c\in\Rm$. 
The shear invariance in law of~\cref{eq:uPDE-many} with a random forcing is the following property.
Suppose that $\mathbf{u}=(u_{1},\ldots,u_{N})$ solves \cref{eq:uPDE-many}
and define $\tilde{\mathbf{u}}(t,x)=\mathbf{u}(t,x+ct)-(c,\ldots,c)$.
Then it is easy to see that
\[
  \partial_{t}\tilde{u}_{i}=\frac{1}{2}\partial_{x}^{2}u_{i}-\frac{1}{2}\partial_{x}(u_{i}^{2})+\dif(\partial_{x}\tilde{V}),
\]
where $\tilde{V}(t,x)=V(t,x+ct)$. Since $\dif(\partial_{x}V)$ is
white in time, informally speaking $\dif(\partial_{x}\tilde{V})$
and $\dif(\partial_{x}V)$ have the same law. Therefore, $\tilde{\mathbf{u}}$
agrees in law with a solution to \cref{eq:uPDE-many}. This is made 
precise in \cref{subsec:shearinvariance}. On the other hand, if $\mathbf{u}(t,\cdot)$
is space-stationary, then $\tilde{\mathbf{u}}(t,\cdot)$ has the same
law as~$\mathbf{u}(t,\cdot)-(c,\ldots,c)$. This directly leads 
to statement~\cref{enu:shearinvariant} in \cref{thm:maintheorem-classification}. 
It also allows us, once we have constructed a single invariant measure, to construct many 
by vertical translation. 

The ordering and $L^1$-contraction properties for the random Burgers equation
are closely related, as in the deterministic case.
Informally speaking, in \cref{thm:ordering}, we show that space-time stationary solutions to~\cref{eq:uPDE-1} are ordered.
This is not an immediate consequence of the standard comparison principle because
we can not a priori pin down
any fixed time  when we would easily compare the two solutions and claim that this order propagates.
A precise formulation of the ordering of the solutions is that 
the components of a space-time stationary solution
to \cref{eq:uPDE-many} must be ordered almost surely.  
In addition, we show in \cref{prop:couple}
that any two laws of space-time stationary solutions to~\cref{eq:uPDE-1} 
or~\cref{eq:uPDE-many} can be
coupled to obtain another space-time stationary solution to \cref{eq:uPDE-many}, with more
components. This implies that there cannot be two distinct elements of $\overline{\mathscr{P}}_{\mathbb{R}}^{\mathrm{e}}(\mathcal{X}_{m})$
with the same mean. The ordering is a consequence of the comparison
principle and $L^{1}$-contraction for the Burgers equation, which we discuss in \cref{sec:comparison-L1}. To prove the ordering statement,  we show that two components of a spacetime-stationary solution
to \cref{eq:uPDE-many} cannot intersect transversely, as that would reduce an~$L^1$-norm. Then we use the strong
maximum principle to rule out degenerate intersections.

To prove the convergence of the solutions to an invariant measure, 
we again use the $L^1$-contraction property of the Burgers equation.
Under appropriate conditions, two solutions evolving according to the same noise must get close to one another at many times.
Then, intuitively, the $L^1$-contraction forces them to stay close to each other for all times.
Of course, the difference of two space-stationary solutions is generally not in $L^1(\R)$, so here the $L^1$-contraction is used on the probability space. The $L^1$ contraction property on the probability space is analogous to but different
from the standard spatial $L^1$-contraction and holds for spatially invariant solutions.

\subsection*{Organization of the paper}
The paper is organized as follows. In \cref{sec:solntheory}, we 
show that the equation \cref{eq:uPDE} is well-posed in certain weighted spaces spaces,
as long as the growth at infinity is sublinear. The main result of that section
is \cref{prop:thetawellposed}, from which \cref{thm:existence}
follows immediately.
In \cref{sec:comparison-L1} we prove the comparison principle and $L^1$-contraction both
in space and in probability.
In \cref{sec:basicproperties} we prove some other useful basic properties of the solutions.
In \cref{sec:tightness}, we establish the tightness in $\mathcal{X}_m$ of the solution to \cref{eq:uPDE-1} started from a constant, for $m\in[1/2,1)$. This shows the existence of stationary solutions to \cref{eq:uPDE-1}.
In \cref{sec:classification}, we complete the proof of \cref{thm:maintheorem-classification} by classifying all extremal elements of $\overline{\mathscr{P}}_G(\mathcal{X}_m)$.
In \cref{sec:stability}, we prove the stability result \cref{thm:maintheorem-stability}. 
The appendices contain the proofs of several auxiliary results.
In \cref{appendix:stability-criterion}, we prove \cref{prop:stability-criterion},
as its proof is elementary and unrelated to the rest of the paper.
\cref{appendix:app-weight} includes some background on weighted spaces and estimates on the solutions
to the heat equation in 
weighted spaces; these results are used extensively in \cref{sec:solntheory}. In \cref{appendix:classicalismild} we show that classical solutions  
to the Burgers equation are mild (tying up a loose end from \cref{sec:solntheory} that is not used in the rest of the paper),
 and in \cref{appendix:symmetrylemma}, we prove some other technical lemmas that are used at various points throughout the paper.

\subsection*{Acknowledgments}

We are happy to thank Yuri Bakhtin and Konstantin Khanin for generous
explanations of their work, and Kevin Yang for productive discussions.
We are also especially grateful to Yu Gu for pointing out a strengthening
of \cref{prop:gamma} that led to an improvement of our main result
over the initial version of the paper, and for pointing out an error in an earlier version.
AD was partially supported
by an NSF Graduate Research Fellowship under grant DGE-1147470, CG
by the Fannie and John Hertz Foundation and NSF grant DGE-1656518,
and LR by NSF grants DMS-1613603 and DMS-1910023, and ONR grant N00014-17-1-2145.

\section{Solutions to the Burgers equation in weighted spaces\label{sec:solntheory}}

In this section, we construct solutions to the stochastic Burgers
equation in weighted spaces that permit growth at infinity. A key preliminary step is a standard trick going back to \cite{DPDT94}: by subtracting off a solution to the linearized version of \cref{eq:uPDE-many}, 
we reduce a stochastic partial differential equation \cref{eq:uPDE-many} to 
a partial differential equation   \cref{eq:thetajPDE} with random coefficients
coming from the solution to the linearized problem. This does not circumvent the need to
work with solutions that may grow at infinity but it does allow us to work
with classical solutions. 
The goal of this section is to show that \cref{eq:thetajPDE} 
has classical solutions pathwise,
in appropriate weighted spaces, so that 
we can treat the noise as a fixed object rather than a random one. 
Thus, the only genuine stochastic analysis required  is to understand
the Gaussian process solving the linearized problem. This step also allows
us to avoid some of the minor additional technicalities involved with working
directly with the strong solutions in the sense of~\cite{DPZ14}.  

The two main results of this section are, first, \cref{prop:thetawellposed},
which is a version of \cref{thm:existence} that allows for discontinuous initial data, and
stated in terms of the solutions to  \cref{eq:thetajPDE}, and, second, the Feller property stated as
\cref{prop:fellerproperty}.

To prove \cref{prop:thetawellposed}, we first consider the periodized
version of the problem, with both the initial conditions and the noise periodized, and then 
pass to the limit as the periodization length is taken to infinity. The periodized 
problem is set up in \cref{subsec:periodized}. To solve it, we use the mild formulation of the problem, 
which we relate to the classical formulation in \cref{subsec:mildsolutions}. We 
then solve the periodized problem, using a fixed-point argument similar to that of \cite{DPDT94}, 
in \cref{subsec:localintimexistence}. To extend the solution theory to the whole space, we   
control the growth of the solutions in sublinearly weighted spaces in \cref{subsec:globalintime}. It is here that the proof diverges significantly from the situation for the linear problem, as the sublinear weights are necessary for well-posedness. Finally, we pass to the limit of the periodization scales in \cref{subsec:solutionsonthewholespace} 
to prove \cref{prop:thetawellposed,thm:existence}.

\subsection{From an SPDE to a PDE\label{subsec:sPDEtoPDE}}

We avoid working directly with the SPDE \cref{eq:uPDE-1}
by making use of the following trick introduced in~\cite{DPDT94}.
Solving a linearization of \cref{eq:uPDE-1}, namely
\begin{equation}\label{jun2206}
  \dif\psi=\frac{1}{2}\partial_{x}^{2}\psi\dif t+\dif(\partial_{x}V),
\end{equation}
with initial condition $\psi(0,\cdot)\equiv0$, is simple: the solution
is given by the stochastic integral
\begin{equation}
  \psi(t,x)=\int_{0}^{t}[\partial_{x}G_{t-s}*\dif V(s,\cdot)](x)=\int_{0}^{t}\int_{\mathbb{R}}(\partial_{x}G_{t-s}*\rho)(x-y)\,\dif W(s,y),\label{eq:psidef}
\end{equation}
where
\begin{equation}\label{jun2202}
  G_{t}(x)=(2\pi t)^{-1/2}\exp\{-x^{2}/(2t)\}
\end{equation}
is the heat kernel. See \cite[Chapter 5 and Theorem 5.2]{DPZ14} for
a detailed discussion of such stochastic integrals, but note also that
$\psi$ is simply a mean-zero Gaussian process on $\mathbb{R}\times\mathbb{R}$
with covariance function
\begin{align}
  \mathbb{E}\psi(t,x)\psi(t',x') & =\int_{0}^{t\wedge t'}\int_{\mathbb{R}}(\partial_{x}G_{t-s}*\rho)(x-y)(\partial_{x}G_{t'-s}*\rho)(x'-y)\,\dif y\,\dif s\nonumber                                 \\
                                 & =-\int_{0}^{t\wedge t'}\partial_{xx}(G_{t+t'-2s}*\rho^{*2})(x-x')\,\dif s=\int_{0}^{t\wedge t'}\frac{\dif}{\dif s}(G_{t+t'-2s}*\rho^{*2})(x-x')\,\dif s\nonumber \\
                                 & =([G_{|t-t'|}-G_{t+t'}]*\rho^{*2})(x-x').\label{eq:psicovariance}
\end{align}
In fact, $(\psi,V)$ is jointly Gaussian.
A special case of \cref{eq:psicovariance} is 
\begin{equation}
 \E \psi(t,x)\psi(t,x')=(\rho^{*2}-G_{2t}*\rho^{*2})(x-x').\label{eq:psicovariance-fixedtime}
\end{equation}
From this one can see that as $t\to \infty$, for fixed $x,x'\in\R$ we have 
\[
\E \psi(t,x)\psi(t,x')\to \rho^{*2}(x-x'), 
\]
and $\psi(t,\cdot)$ converges in law to a Gaussian process with covariance kernel $\rho^{*2}$ in the topology of an appropriate weighted space. We discuss the necessary weights in \cref{lem:suppsi} below.

Writing $u=\theta+\psi$, we see, as in \cite{DPDT94}, that a function $u$
is a strong solution to \cref{eq:uPDE-1} if and only if~$\theta=u-\psi$
is a classical solution to the PDE
\begin{equation}
\partial_{t}\theta=\frac{1}{2}\partial_{x}^{2}\theta-\frac{1}{2}\partial_{x}(\theta+\psi)^{2}.
\label{eq:thetajPDE}
\end{equation}
Our analysis will start from this equation, rather than directly from 
\cref{eq:uPDE-1}.
In particular, our first goal is to build strong solutions 
to \cref{eq:thetajPDE} in certain weighted spaces. Going forward, we can treat $\psi$ pathwise,
as if it were a deterministic object.

\subsection{Solutions in weighted function spaces\label{subsec:functionspaces}}

We now introduce some weighted function spaces that we will
use in constructing solutions to~\cref{eq:thetajPDE}. This is necessary as the force $\psi$ and thus also the 
solution $\theta$  in \cref{eq:thetajPDE} grow at infinity.
Given a weight~$w(x)>0$, the weighted space $L_{w}^{\infty}(\mathbb{R})$ is the
space of measurable functions~$v:\mathbb{R}\to\mathbb{R}$ such that
\[
\|v\|_{L_{w}^{\infty}(\mathbb{R})}=\esssup_{x\in\mathbb{R}}\frac{|v(x)|}{w(x)}<+\infty,
\]
and $\mathcal{C}_{w}(\mathbb{R})\subset L_{w}^{\infty}(\mathbb{R})$
is the subspace of continuous functions in $L_{w}^{\infty}(\mathbb{R})$, with the same norm.
For~$\alpha\in(0,1)$, the weighted
H\"older space is the subspace of $\mathcal{C}_{w}(\mathbb{R})$ 
of functions such that 
\[
\|v\|_{\mathcal{C}_{w}^{\alpha}(\mathbb{R})}=\|v\|_{\mathcal{C}_{w}(\mathbb{R})}+\sup_{|x-y|\le1}\frac{|v(x)-v(y)|}{w(x)|x-y|^{\alpha}}<+\infty.
\]
The higher order H\"older spaces $\mathcal{C}_{w}^{k+\alpha}(\mathbb{R})$, 
with $k\in\mathbb{N}$ and $\alpha\in[0,1)$, have the norms
\[
\|v\|_{\mathcal{C}_{w}^{k+\alpha}(\mathbb{R})}=\|v\|_{\mathcal{C}_{w}^{k,\alpha}(\mathbb{R})}=\sum_{j=0}^{k}\|\partial_{x}^{j}v\|_{\mathcal{C}_{w}^{\alpha}(\mathbb{R})}.
\]
Finally, for
$p\in[1,\infty)$,  the space $L_{w}^{p}(\mathbb{R})$ is  equipped with the norm
\[
\|v\|_{L_{w}^{p}(\mathbb{R})}=\left(\int_{\mathbb{R}}
\left(\frac{|v(x)|}{w(x)}\right)^{p}\,\dif x\right)^{1/p}<+\infty.
\]
We will often use the weights
\begin{equation}
\p_{\ell}(x) =\langle x\rangle^{\ell} \quad \textrm{with } \langle x\rangle  =\sqrt{4+x^{2}}\label{eq:weightdefs}
\end{equation}
and $\ell\in\mathbb{R}$. The constant $4$ rather than $1$ in the definition 
of  $\langle x\rangle$ ensures that $\log\langle x\rangle>0$ for all~$x \in \R$, 
which will be convenient when we use logarithmic weights.

For $m\in\mathbb{R}$, we define the Fréchet space
\[
  L_{\p_{m+}}^{\infty}(\mathbb{R})=\bigcap_{\ell>m}L_{\p_{\ell}}^{\infty}(\mathbb{R})
\]
equipped with the topology generated by all $L^\infty_{\p_\ell}(\mathbb{R})$
norms for $\ell>m$. This space is metrizable, for example by the
metric
\[
  d_{L_{\p_{m+}}^{\infty}(\mathbb{R})}(v_{1},v_{2})=\sum_{k=1}^{\infty}2^{-k}
  \frac{ \|v_{1}-v_{2}\|_{L_{\p_{m+1/k}}^{\infty}(\mathbb{R})}}{1+\|v_{1}-v_{2}\|_{L_{\p_{m+1/k}}^{\infty}(\mathbb{R})}}.
\]
A sequence $v_{n}$ converges to $v$ in $L_{\p_{m+}}^{\infty}(\mathbb{R})$
if and only if $v_{n}$ converges to $v$ in the topology of $L_{\p_{\ell}}^{\infty}(\mathbb{R})$
for each $\ell>m$. Therefore, for any topological space $\mathcal{Z}$,
a map $f:\mathcal{Z}\to L_{\p_{m+}}^{\infty}(\mathbb{R})$ is continuous
if and only if $f$ is a continuous map $\mathcal{Z}\to L_{\p_{\ell}}^{\infty}(\mathbb{R})$
for each $\ell>m$. In particular, the inclusion maps~$L_{\p_{m+}}^{\infty}(\mathbb{R})\to L_{\p_{\ell}}^{\infty}(\mathbb{R})$
for $\ell>m$ are continuous.

The key space for us is $\mathcal{X}_{m}$, which we define as the closed subspace of continuous
functions in $L_{\p_{m+}}^{\infty}(\mathbb{R})$. 
The space
of continuous, compactly-supported functions is dense in~$\mathcal{X}_{m}$. This means that $\mathcal{X}_{m}$ is separable and hence a Polish space, unlike the 
spaces $\mathcal{C}_{w}(\mathbb{R})$ and $L_{\p_{m+}}^{\infty}(\mathbb{R})$ which
are not separable.
We 
prefer to work with continuous functions when possible, since the
separability of the space $\mathcal{X}_{m}$ will allow us to use
probabilistic tools about random variables on Polish spaces. Solutions will be 
continuous at all positive times due to the smoothing effect
of the Laplacian in \cref{eq:thetajPDE}. 
However, we will have occasion to
solve \cref{eq:uPDE-1} and \cref{eq:thetajPDE} 
with discontinuous initial data. In particular, this will be relevant 
in the proof of the stability result in \cref{thm:maintheorem-stability}.
Thus, we make the following definition. Here and henceforth, if $\mathcal{Y}_1$ is a metric space and $\mathcal{Y}_2$ is a topological vector spaces, we use the notation $\mathcal{C}_{\mathrm{b}}(\mathcal{Y}_1;\mathcal{Y}_2)$ to refer to the space of bounded continuous functions from $\mathcal{Y}_1$ to $\mathcal{Y}_2$.

\begin{defn}
\label{def:Zmdef}We define $\mathcal{Z}_{m,T}$ to be the space of
functions $u\in\mathcal{C}_{\mathrm{b}}((0,T];\mathcal{X}_{m})$
such that for each~$\ell>m$ the limit
\begin{equation}
u(0,\cdot):=\lim_{t\downarrow0}u(t,\cdot)\label{eq:weakstarlimit}
\end{equation}
exists in the weak-$*$ topology on $L_{\p_{\ell}}^{\infty}(\mathbb{R})$,
and the initial condition $u(0,\cdot)\in L_{\p_{m+}}^{\infty}(\mathbb{R})$.
\end{defn}

In particular, if $u\in\cZ_{m,T}$, then 
$t\mapsto u(t, \cdot)\in L_{\p_{\ell}}^{\infty}(\mathbb{R})$
is continuous on~$[0,T]$ if $L_{\p_{\ell}}^{\infty}(\mathbb{R})$ is endowed
with the weak-$*$ topology.
We endow $\mathcal{Z}_{m,T}$ with the
  subspace topology inherited from the embedding
  \[
    \mathcal{Z}_{m,T}\ni u\mapsto(u(0,\cdot),u|_{(0,T]\times\mathbb{R}})\in L_{\p_{m+}}^{\infty}(\mathbb{R})\times\mathcal{C}_{\mathrm{b}}((0,T];\mathcal{X}_{m}).
  \]
We further define $\mathcal{Z}_{m}=\mathcal{Z}_{m,\infty}$ to be
  the space of functions $u:[0,\infty)\times\mathbb{R}\to\mathbb{R}$
  such that the restriction~$u|_{[0,T]\times\mathbb{R}}\in\mathcal{Z}_{m,T}$ for each
  $T>0$, equipped with the weakest topology such that each restriction
  map $u\mapsto u|_{[0,T]\times\mathbb{R}}$ is continuous.
 
To discuss classical solutions to \cref{eq:thetajPDE}, we define a smaller class 
$\tilde{\mathcal{Z}}_{m,T}$ of functions $\theta \in \mathcal{Z}_{m,T}$
that are twice-differentiable in space and once in time on $(0, T) \times \R$, and moreover are such that for every compact~$I \subset (0, T)$ and~$\eps > 0$, there exists $C < \infty$ such that
\begin{equation}
  \label{eq:time-deriv-control}
  |\partial_t\theta(t, x)| \leq C \e^{\eps x^2} \quad \textrm{for all } (t, x) \in I \times \R.
\end{equation}
The reason why we impose this bound on $\partial_t\theta$ rather than just on $\theta$
itself, as is done for the Cauchy problem for the heat equation, is explained
in \cref{lem:classicalismild} in \cref{appendix:classicalismild}.
We also define~$\tilde{\mathcal{Z}}_{m}=\tilde{\mathcal{Z}}_{m,\infty}\subset\mathcal{Z}_{m}
=\mathcal{Z}_{m,\infty}$.
The following proposition implies the existence and uniqueness claims for the solutions
to \cref{eq:uPDE-1} in~\cref{thm:existence}.
\begin{prop}\label{prop:thetawellposed}
Let $m\in(0,1)$. Almost surely, 
there is a map 
\[
\Phi :L_{\p_{m+}}^{\infty}(\mathbb{R})\to\tilde{\mathcal{Z}}_{m},
\]
so that for each ${v}\in L_{\p_{m+}}^{\infty}(\mathbb{R})$,
$\theta =\Phi(v)$ is the unique strong
solution to \cref{eq:thetajPDE} with the initial condition $\theta(0,\cdot)=v$.
The map $v\mapsto\Phi(v)|_{[0,T]\times\mathbb{R}}$ is measurable with 
respect to $\mathcal{F}_{T}$. Moreover, also almost surely, 
the map $\Phi$ is continuous and for any bounded set $A \subset L_{\p_{m+}}^{\infty}(\mathbb{R})$ and $T>0$, 
the image~$\Phi(A)|_{[0, T] \times \R}$ is bounded in $\mathcal{Z}_{m, T}$ and the restriction $\Phi(A)|_{\{T\} \times \R}$ is compact in $\mathcal{X}_m$.
\end{prop}
The proof of this proposition occupies most of the rest of this section. Its immediate
consequence is that we can define a solution 
map $\Psi:L_{\p_{m+}}^{\infty}(\mathbb{R})\to\mathcal{Z}_{m}$ so that for 
each $v\in L_{\p_{m+}}^{\infty}(\mathbb{R})$,
$u=\Psi(v)$ is the unique strong solution to \cref{eq:uPDE-1}
satisfying $u(t,\cdot)=v$,  given   by
\begin{equation}\label{jun2204}
\Psi(v)(t,x)=\Phi(v)(t,x)+\psi(t,x),
\end{equation}
where $\psi(t,x)$ is the solution to the linearized problem \cref{jun2206}
given by \cref{eq:psidef}. As with $\Phi$, the map $\Psi$ is continuous and for any bounded set $A \subset L_{\p_{m+}}^{\infty}(\mathbb{R})$ and $T > 0$, 
the image $\Psi(A)|_{[0, T] \times \R}$ is bounded in~$\mathcal{Z}_{m, T}$ and $\Psi(A)|_{\{T\} \times \R} \subset \mathcal{X}_m$ is compact in $\mathcal{X}_m$. 
This is the existence and uniqueness claim of \cref{thm:existence} but, in addition, it allows discontinuous initial data. This result also generalizes immediately 
to the system  \cref{eq:uPDE-many}
of $N$ such decoupled equations.

Another  consequence of \cref{prop:thetawellposed} is that for any $t>0$ and $N\in\mathbb{N}$ there
is a map $P_{t}$ from the space of measurable functions on $\mathcal{X}_{m}^{N}$
to the space of measurable functions on $L_{\p_{m+}}^{\infty}(\mathbb{R})^{N}$
given by
\[
  (P_{t}f)(\mathbf{v})=\mathbb{E}f(\Psi(\mathbf{v})(t,\cdot)),
\]
where $\bv$ is the solution to the decoupled system \cref{eq:uPDE-many}.  
Because $\Psi$ is almost surely continuous, we get the following
Feller property.
\begin{prop}[Feller property]\label{prop:fellerproperty}
  If $m\in(0,1)$ and $f\in\mathcal{C}_{\mathrm{b}}(\mathcal{X}_{m}^{N})$,
  then $P_{t}f\in\mathcal{C}_{\mathrm{b}}(L_{\p_{m+}}^{\infty}(\mathbb{R})^N)$ for all $t>0$.
  In particular, $P_{t}f$ (strictly speaking, $P_{t}f|_{\mathcal{X}_{m}^{N}}$)
  is an element of $\mathcal{C}_{\mathrm{b}}(\mathcal{X}_{m}^{N})$.
\end{prop}

\begin{proof}
Let $\mathbf{v}^{(n)}\to\mathbf{v}$ in $L_{\p_{m+}}^{\infty}(\mathbb{R})^{N}$ and $t>0$. 
Then for each component $i=1,\ldots,N$, we have that 
\[
\Psi(v_i^{(n)})(t,\cdot)\to\Psi(v_i)(t,\cdot)~~\hbox{ in $\mathcal{X}_{m}^{N}$,}
\]
almost surely. It follows that 
\[
\mathbb{E}f(\Psi(v_1^{(n)})(t,\cdot),\ldots,\Psi(v_N^{(n)})(t,\cdot))\to\mathbb{E}f(\Psi(v_1)(t,\cdot),\ldots,\Psi(v_N){(t,\cdot)}),
\]
by the bounded convergence theorem since $f$ is bounded.
\end{proof}

\cref{prop:fellerproperty} implies the Feller property of  \cref{eq:uPDE-1} claimed
in \cref{thm:existence}. Hence, the proof of \cref{thm:existence} is reduced to that of 
\cref{prop:thetawellposed}.

\subsection{The periodized problem\label{subsec:periodized}}

The proof of \cref{prop:thetawellposed} proceeds in two steps. First,
we show the existence of solutions to a \emph{periodized} stochastic
Burgers equation, and then we take the limit as the period tends to
infinity. We use the notation $\mathcal{C}^{\alpha}(\mathbb{R}/L\mathbb{Z})$
for functions in $\mathcal{C}^{\alpha}(\mathbb{R})$ which are $L$-periodic,
with the convention~$\mathcal{C}^{\alpha}(\mathbb{R}/\infty\mathbb{Z})=\mathcal{C}^{\alpha}(\mathbb{R})$.
We emphasize that the use of periodicity in this section is fundamentally different from the use of periodicity in later sections. In this section, we periodize both the initial conditions and the driving noise, so that the solutions are periodic almost surely. In later sections, we will consider solutions to \cref{eq:uPDE-many} with noise that is not periodic, but whose \emph{laws} are periodic.

For $L>0$, let $ \chi^{[L]}$ be a smooth, compactly-supported bump
function, taking values in $[0,1]$, such that
\begin{align}
  \chi^{[L]}|_{[-L/2,L/2]} & \equiv1, & \chi^{[L]}|_{[-L/2 - 1, L/2 + 1]} & \equiv0, & \|\chi^{[L]}\|_{\mathcal{C}^{k}(\mathbb{R})} & \le C_{k}\label{eq:chiLdef}
\end{align}
for each $k$, where $C_{k}<\infty$ is a constant that may depend
on $k$ but not on $L$. The $L$-periodized version of $V$ is  %
\begin{equation}\label{eq:VLdef}
V^{[L]}(x)=\sum_{j\in\mathbb{Z}} (\chi^{[L]}V)(x-jL),  %
\end{equation}
For notational convenience, we define $V^{[\infty]}=V$. For $L\in(0,\infty]$,
let
\begin{equation}
  \psi^{[L]}(t,\cdot)=\int_{0}^{t}\partial_{x}G_{t-s}*\dif V^{[L]}(s,\cdot),\label{eq:psiLdef}
\end{equation}
so that $\psi^{[L]}$ is $L$-periodic and solves the SPDE
\begin{align*}
  &\dif\psi^{[L]}      =\frac{1}{2}\partial_{x}^{2}\psi^{[L]}\dif t+\dif(\partial_{x}V^{[L]}), \\
  &\psi^{[L]}(0,\cdot)  \equiv0.
\end{align*}
The family $\{\psi^{[L]}\}_{L\in(0,\infty]}$ is coupled by taking
the stochastic convolutions of the same realization of~$V$. We
will always assume that we have taken modifications of $\psi^{[L]}$
with continuous paths.

We will consider the $L$-periodic approximation to \cref{eq:thetajPDE}
\begin{align}
&  \dif\theta^{[L]}      =\frac{1}{2}\partial_{x}^{2}\theta^{[L]}-\frac{1}{2}\partial_{x}(\theta^{[L]}+\psi^{[L]})^{2},\label{eq:thetaLproblem} \\
&  \theta^{[L]}(0,\cdot) =v^{[L]},\label{eq:thetaLic}
\end{align}
with some initial condition
\begin{equation}
  v^{[L]}\in\begin{cases}
    L^{\infty}(\mathbb{R}/L\mathbb{Z}), & \text{if }L<\infty,                                              \\
    L_{\p_{m+}}^{\infty}(\mathbb{R}),   & \text{for some \ensuremath{m\in(0,1)} if \ensuremath{L=\infty}.}
  \end{cases}\label{eq:vLcond}
\end{equation}
This PDE has classical solutions, and can be solved
pathwise in the noise. Indeed, the following lemma is the only fact
about $\psi^{[L]}$ that we will use in this section.
\begin{lem}
  \label{lem:suppsi}Define the weight $g(x) = (\log\langle x\rangle)^{3/4}$.
  For any $T<\infty$ and $j \in \mathbb{N}$, with probability $1$ we have
  \begin{equation}
    \sup_{L\in[1,\infty]}\|\psi^{[L]}\|_{\mathcal{C}_{\mathrm{b}}([0,T];\mathcal{C}_{g}^{j}(\mathbb{R}))}<\infty.\label{eq:suppsi}
  \end{equation}
\end{lem}

\begin{proof} Let us first recall a standard bound on the growth of $V(t,x)$ and its derivatives
at infinity which implies that
$\partial_{x}^{j}V\in\mathcal{C}_{\mathrm{b}}([0,T];\mathcal{C}_{g}^{1}(\mathbb{R}))$
almost surely.
Since the noise $V$ is a convolution of a cylindrical Wiener process
with a spatially smooth process, it  is
  Gaussian, continuous in time, and smooth in space. In particular,
  we have, for any $j\in\mathbb{Z}_{\ge0}$ and $k\in\mathbb{Z}$ fixed, that
  \[
    \mathbb{E}\left(\sup_{x\in[k,k+1]}\sup_{t\in[0,T]}|\partial_{x}^{j}V(t,x)|\right)<\infty
  \]
  by Fernique's inequality (see e.g. \cite[Theorem 4.1]{Adl90}). Therefore,
  by the Borell--TIS inequality (see e.g.~\cite[Theorem 2.1]{Adl90})
there exist constants $c>0$ and $C<\infty$, depending on $T$, so that for all~$z>0$ we have
  \[
    \mathbb{P}\left(\sup_{x\in[k,k+1]}\sup_{t\in[0,T]}|\partial_{x}^{j}V(t,x)|>z\right)\le C\e^{-cz^{2}}.
  \]
  By a union bound, this means that
  \begin{align*}
    \mathbb{P}\left(\sup_{k\in\mathbb{Z}}\frac{1}{g(k)}\sup_{x\in[k,k+1]}\sup_{t\in[0,T]}|\partial_{x}^{j}V(t,x)|>z\right) & \le\sum_{k\in\mathbb{Z}}\mathbb{P}\left(\sup_{x\in[k,k+1]}\sup_{t\in[0,T]}|\partial_{x}^{j}V(t,x)|>zg(k)\right) \\
                                                                                                                           & \le C\sum_{k\in\mathbb{Z}}\e^{-cz^{2}(\log\langle k\rangle)^{3/2}},
  \end{align*}
  and the sum on the right is finite and goes to $0$ as $z\to\infty$.
  In particular,
  \[
    \sup_{k\in\mathbb{Z}}\frac{1}{g(k)}\sup_{x\in[k,k+1]}\sup_{t\in[0,T]}|\partial_{x}^{j}V(t,x)|
  \]
is finite almost surely, and so $\partial_{x}^{j}V\in\mathcal{C}_{\mathrm{b}}([0,T];\mathcal{C}_{g}^{1}(\mathbb{R}))$
almost surely, as we have claimed. In addition, it is clear from \cref{eq:VLdef} that
  \begin{equation}
    \sup_{L\in[1,\infty]}\|\partial_{x}^{j}V^{[L]}\|_{\mathcal{C}_{\mathrm{b}}([0,T];\mathcal{C}_{g}^{1}(\mathbb{R}))}\le3\|\partial_{x}^{j}V\|_{\mathcal{C}_{\mathrm{b}}([0,T];\mathcal{C}_{g}^{1}(\mathbb{R}))}<\infty\label{eq:partialxjVL}
  \end{equation}
  almost surely. 
  
Now we can turn our attention to $\psi^{[L]}$. We
  have
  \begin{equation}
    \begin{aligned}\psi^{[L]}(t,x)&=\int_{0}^{t}\partial_{x}G_{t-s}*\dif V^{[L]}(s,\cdot)(x) =-\int_{0}^{t}\partial_{tx}G_{t-s}*V^{[L]}(s,\cdot)(x)\,\dif s                \\
      & =-\frac{1}{2}\int_{0}^{t}G_{t-s}*\partial_{x}^{3}V^{[L]}(s,\cdot)(x)\,\dif s,
    \end{aligned}
    \label{eq:psiLexpr}
  \end{equation}
  and so
  \begin{equation}
    \partial_{x}^j\psi^{[L]}(t,x)=-\frac{1}{2}\int_{0}^{t}G_{t-s}*\partial_{x}^{j + 3}V^{[L]}(s,\cdot)(x)\,\dif s.\label{eq:dxpsiLexpr}
  \end{equation}
  The conclusion \cref{eq:suppsi} follows from \cref{eq:partialxjVL}--\cref{eq:dxpsiLexpr}
  and \cref{lem:weightedhkbound} in \cref{subappendix:hkbounds}.
\end{proof}

\subsection{Mild solutions\label{subsec:mildsolutions}}

In this section, we  show that the mild formulation of \cref{eq:thetaLproblem}--\cref{eq:thetaLic},
namely
\begin{equation}
  \theta^{[L]}(t,\cdot)=G_{t}*v^{[L]}-\frac{1}{2}\int_{0}^{t}\partial_{x}G_{t-s}*(\theta^{[L]}(s,\cdot)+\psi^{[L]}(s,\cdot))^{2}\,\dif s,\qquad t>0.\label{eq:mildformulation}
\end{equation}
is equivalent
to its classical formulation. As part of the proof, we
establish regularity estimates on mild solutions which will be important later on.
\begin{lem}\label{lem:mildregularitypoly}
  Given $\ell>0$, $\alpha\ge0$,
  $\beta\in[\alpha,\alpha+1)$, $\ell>0$, and $S>0$, there exists~$C=C(\alpha,\beta,\ell,S)$ so that for
  all $L\in(0,\infty]$ and $v^{[L]}$ as in \cref{eq:vLcond}, if
  $\theta^{[L]}\in\mathcal{C}_{\mathrm{b}}((0,S];\mathcal{C}_{\p_{\ell}}^{\alpha}(\mathbb{R}))$
  satisfies \cref{eq:mildformulation}, then for all $t\in(0,S]$ we have
  \begin{align*}
    \|\theta^{[L]}(t,\cdot)\|_{\mathcal{C}_{\p_{2\ell}}^{\beta}(\mathbb{R})}  &\le Ct^{-\frac{\beta-\alpha}{2}}\|v^{[L]}\|_{L_{\p_{2\ell}}^{\infty}(\mathbb{R})}                                                                                                                           \\
 &\qquad+C\int_{0}^{t}(t-s)^{-\frac{\beta-\alpha+1}{2}}%
 \Big[\|\theta^{[L]}(s,\cdot)\|_{\mathcal{C}_{\p_{\ell}}(\mathbb{R})}\|\theta^{[L]}(s,\cdot)\|_{\mathcal{C}_{\p_{\ell}}^{\alpha}(\mathbb{R})}
 \\&\qquad\qquad\qquad\qquad\qquad\qquad\qquad+\|\psi^{[L]}(s,\cdot)\|_{\mathcal{C}_{\p_{\ell}}(\mathbb{R})}\|\psi^{[L]}(s,\cdot)\|_{\mathcal{C}_{\p_{\ell}}^{\alpha}(\mathbb{R})}\Big]
\,\dif s.
  \end{align*}
\end{lem}

\begin{proof}
  This follows from applying \cref{lem:weightedhkbound} to \cref{eq:mildformulation}
  and using the triangle inequality.
\end{proof}
\cref{lem:mildregularitypoly} can be iterated to obtain
bounds on higher derivatives. For simplicity, we only state the results
we use later on.
\begin{cor}
  Fix $\ell>0$, $L\in(0,\infty]$, and $S>0$. If $\theta^{[L]}\in\mathcal{C}_{\mathrm{b}}((0,S];\mathcal{C}_{\p_{\ell}}(\mathbb{R}))$
  satisfies \cref{eq:mildformulation} for some $v^{[L]}$ satisfying
  \cref{eq:vLcond}, then for all $t\in(0,S]$ there is a constant $C(t)$
  (depending also on $\ell$) so that
  \begin{equation}
    \|\theta^{[L]}(t,\cdot)\|_{\mathcal{C}_{\p_{4\ell}}^{1}(\mathbb{R})}\le C(t)\left(\|\theta^{[L]}\|_{\mathcal{C}_{\mathrm{b}}((0,t];\mathcal{C}_{\p_{\ell}}(\mathbb{R}))}^{4}+\|\psi^{[L]}\|_{\mathcal{C}_{\mathrm{b}}((0,t];\mathcal{C}_{\p_{\ell}}^{1/2}(\mathbb{R}))}^{4} + 1\right).\label{eq:C1bound}
  \end{equation}
\end{cor}

\begin{proof} Applying \cref{lem:mildregularitypoly} on the time interval $[t/2,t]$ gives 
\begin{equation*}
\|\theta^{[L]}(t,\cdot)\|_{\mathcal{C}_{\p_{4\ell}}^{1}(\mathbb{R})} \le  
C\|\theta^{[L]}\|_{\mathcal{C}_{\mathrm{b}}([t/2,t];\mathcal{C}_{\p_{2\ell}} (\mathbb{R}))}+
 C\left(\|\theta^{[L]}\|_{\mathcal{C}_{\mathrm{b}}([t/2,t];\mathcal{C}_{\p_{2\ell}}^{1/2}(\mathbb{R}))}^{2}
 +\|\psi^{[L]}\|_{\mathcal{C}_{\mathrm{b}}([t/2,t];\mathcal{C}_{\p_{2\ell}}^{1/2}(\mathbb{R}))}^{2}\right),                                                                       
 \end{equation*}
with a constant $C$ depending on $\ell$ and $t$. Applying this lemma again to the terms with $\theta^{[L]}$ in the right side, 
now on the time interval~$[0,t/2]$, gives
\begin{align*}
 \|\theta^{[L]}(t,\cdot)\|_{\mathcal{C}_{\p_{4\ell}}^{1}(\mathbb{R})} &   
 \le C\left(\|\theta^{[L]}\|_{\mathcal{C}_{\mathrm{b}}((0,t];\mathcal{C}_{\p_{\ell}}(\mathbb{R}))}^{4} 
 + \|\psi^{[L]}\|_{\mathcal{C}_{\mathrm{b}}((0,t];\mathcal{C}_{\p_{\ell}}(\mathbb{R}))}^{4} 
 + \|\psi^{[L]}\|^{{2}}_{\mathcal{C}_{\mathrm{b}}([t/2,t];\mathcal{C}_{\p_{2\ell}}^{1/2}(\mathbb{R}))} + 1\right),
  \end{align*}
  which implies \cref{eq:C1bound}.
\end{proof}

We also have the following bound on the nonlinear term in \cref{eq:mildformulation}: set
\begin{equation}
  \label{eq:Duhamel-nonlinearity}
  D_t = -\frac{1}{2}\int_{0}^{t}\partial_{x}G_{t-s}*(\theta^{[L]}(s,\cdot)+\psi^{[L]}(s,\cdot))^{2}\,\dif s.
\end{equation}

\begin{cor}
  \label{cor:nonlinearity-norm-continuous}
  Fix $m \in (0,1)$, $L \in (0,\infty]$, $v^{[L]} \in L_{\p_{m+}}^{\infty}(\mathbb{R})$, and $T > 0$.
  If $\theta^{[L]} \in \mathcal{C}_{\mathrm{b}}((0,T]; \mathcal{X}_{m})$ satisfies \cref{eq:mildformulation}, then $\lim_{t \downarrow 0} D_t = 0$ in $\mathcal{X}_m$.
\end{cor}

\begin{proof}
Fix $m<\ell_{1}<\ell$, so that $v^{[L]}\in L_{\p_{\ell_{1}}}^{\infty}(\mathbb{R})$.
By \cref{lem:weightedhkbound} with $\alpha=\beta=0$, we see that $D_{t} \to 0$ as $t\downarrow0$ in the topology of $\mathcal{C}_{\p_{2\ell}}(\mathbb{R})$.
In addition, we have
\[
D_{t} = \theta^{[L]}(t, \cdot) - G_{t} * v^{[L]}.
\]
  But $\theta^{[L]} \in \mathcal{C}_{\mathrm{b}}((0,T]; \mathcal{C}_{\p_{\ell_{1}}}(\mathbb{R}))$ by hypothesis, and $G_{t}*v^{[L]}$ is bounded in $\mathcal{C}_{\p_{\ell_{1}}}(\mathbb{R})$ uniformly in $t$ by \cref{lem:weightedhkcts}.
  Therefore, by \cref{prop:boosttheweight} in \cref{sec:append-weighted}, we actually have
  \begin{equation}
    \label{eq:Dttozero}
    \lim_{t\downarrow0} D_{t} = 0
  \end{equation}
  in the topology of $\mathcal{C}_{\p_{\ell}}(\mathbb{R})$.
  Since $\ell > m$ was arbitrary, $D_t \to 0$ in $\mathcal{X}_m$ as well.
\end{proof}
It follows that if $\theta^{[L]}\in\mathcal{C}_{\mathrm{b}}((0,T];\mathcal{X}_{m})$
satisfies \cref{eq:mildformulation}, and $v^{[L]}\in\cX_m$, then if we extend $\theta^{[L]}$
to~$t=0$ by setting~$\theta^{[L]}(0,\cdot)=v^{[L]}$, then the extension satisfies $\theta^{[L]}\in\mathcal{C}_{\mathrm{b}}([0,T];\mathcal{X}_{m})$.
We can now show that mild solutions to~\cref{eq:thetaLproblem}--\cref{eq:thetaLic}
are in $\tilde{\mathcal{Z}}_{m}$ and are in fact classical solutions.
\begin{lem}[Mild solutions are classical]\label{lem:mildisclassical}
Fix $m\in(0,1)$, $L\in(0,\infty]$, $T>0$, and $v^{[L]}\in L_{\p_{m+}}^{\infty}(\mathbb{R})$. If $\theta^{[L]}\in\mathcal{C}_{\mathrm{b}}((0,T];\mathcal{X}_{m})$
satisfies \cref{eq:mildformulation}, and we extend it to $t=0$ by $\theta^{[L]}(0,\cdot)=v^{[L]}$, then~$\theta^{[L]}\in\tilde{\mathcal{Z}}_{m,T}$ and $\theta^{[L]}$ is
  a classical solution to the PDE \cref{eq:thetaLproblem}. Also, for
  any $p\in[1,\infty)$ and any $\ell'>m+1/p$, we have
  \begin{equation}
    \theta^{[L]}\in\mathcal{C}_{\mathrm{b}}([0,T];L_{\p_{\ell'}}^{p}(\mathbb{R})).\label{eq:thetactsinLp}
  \end{equation}
  Finally, if $v^{[L]}$ is continuous, then  
  \begin{equation}
    \theta^{[L]}\in\mathcal{C}_{\mathrm{b}}([0,T];\mathcal{X}_{m}).\label{eq:thetaLctsconverges}
  \end{equation}
\end{lem}

\begin{proof}
  Since $\psi^{[L]}$ is smooth, we can iterate \cref{lem:mildregularitypoly}
  to show that $\theta^{[L]}(t,\cdot)$ is a smooth tempered distribution
  for each $t>0$. Differentiating \cref{eq:mildformulation}, we can
  easily check that $\theta^{[L]}$ is differentiable in time and is a classical solution to
  \cref{eq:thetaLproblem}.

  Now, fix $\ell' > m + 1/p$ and $\ell \in (m, \ell' - 1/p)$.
  It follows from \cref{cor:nonlinearity-norm-continuous} that $D_t \to 0$ as~$t \downarrow 0$ in~$\mathcal{C}_{\p_{\ell}}(\mathbb{R})$.
  Since $\ell' > \ell + 1/p,$ this implies convergence in $L_{\p_{\ell'}}^{p}(\mathbb{R})$ as well.
 In addition,  \cref{lem:weightedhkcts} shows that~$G_t \ast v^{[L]} \to v^{[L]}$ in $L_{\p_{\ell'}}^{p}(\mathbb{R})$, and \cref{eq:thetactsinLp} follows.
  Furthermore, \cref{lem:weightedhkcts} implies that
  \[
  G_t \ast v^{[L]} \overset{\mathrm{w}^*}{\longrightarrow} v^{[L]},~~\hbox{ as $t \downarrow 0$ in $L_{\p_\ell}^\infty(\mathbb{R})$,}
  \]
and thus $\theta^{[L]}\in\tilde{\mathcal{Z}}_{m,T}$.
  Similarly, \cref{eq:thetaLctsconverges} follows from \cref{cor:nonlinearity-norm-continuous,lem:weightedhkcts}.
\end{proof}
Conversely, we can show that classical solutions in $\tilde{\mathcal{Z}}_{m}$
are mild. The proof of this is quite standard and is presented 
in \cref{lem:classicalismild} in \cref{appendix:classicalismild}.
Let us mention that one reason not to skip the proof completely is that
it is there that the bound \cref{eq:time-deriv-control} is used.

\subsection{Local-in-time existence for the periodized 
problem} \label{subsec:localintimexistence}

We now show existence of solutions to the periodized problem \cref{eq:thetaLproblem}--\cref{eq:thetaLic}
with $L<\infty$. We use a fixed-point argument based on the mild formulation \cref{eq:mildformulation}. For the moment, the solution we obtain exists only
up to a time depending on $L$ and the random forcing $\psi^{[L]}$.
This dependence will be eliminated in \cref{prop:Sis1} below. Let
$\mathcal{Z}_{T}^{[L]}\subset\mathcal{Z}_{0,T}$ be the subspace of
functions that are~$L$-periodic in space.
\begin{prop}
  \label{prop:localintimeexistence}For each $L\in(0,\infty)$ and $v^{[L]}\in L^{\infty}(\mathbb{R}/L\mathbb{Z})$,
  there is a time $S\in(0,2]$, depending on $L$, $\psi^{[L]}$, and
  $v^{[L]}$, so that there exists a solution $\theta^{[L]}\in \bigcup_{S' \in (0, S)}\mathcal{Z}_{S'}^{[L]}$
  to \cref{eq:thetaLproblem}--\cref{eq:thetaLic}.
  Moreover, if $S<2$, then
  \begin{equation}
    \limsup_{t\to S}\|\theta^{[L]}(t, \cdot)\|_{L^{\infty}(\mathbb{R}/L\mathbb{Z})}=\infty.\label{eq:limsuptoinfty}
  \end{equation}
\end{prop}

\begin{proof}
  In light of \cref{lem:mildisclassical}, it suffices to find a solution
  \begin{equation*}
    \theta^{[L]}\in \bigcup_{S' \in (0, S)}\mathcal{C}_{\mathrm{b}}((0,S'];\mathcal{C}(\mathbb{R}/L\mathbb{Z}))
  \end{equation*}
  to \cref{eq:mildformulation}. %
  We use a fixed-point argument similar
  to that of \cite[Lemma 2.1]{DPDT94}. Let $M\in(0,\infty)$ and~$\Xi_{M}$
  be a smooth cutoff function so that $\Xi_{M}(x)=x$ for $|x|\le M$
  and $|\Xi_{M}(x)|\le M+1$ for all $x$. Fix a time $T\in(0,2]$ to
  be chosen later and note that the operator
  \[
    \mathcal{A}_{M}\theta(t,\cdot)=G_{t}*v^{[L]}-\int_{0}^{t}\partial_{x}G_{t-s}*((\Xi_{M}\circ\theta)(s,\cdot)+\psi^{[L]}(s,\cdot))^{2}\,\dif s
  \]
  maps $\mathcal{C}_{\mathrm{b}}((0,T];\mathcal{C}(\mathbb{R}/L\mathbb{Z}))$
  into itself. By the triangle inequality, the maximum principle for
  the heat equation, and \cref{lem:weightedhkbound}, we have
\begin{equation} \label{eq:nonlin-strong}
\begin{aligned}
\| & \mathcal{A}_{M}\theta(t,\cdot)\|_{\mathcal{C}_{\mathrm{b}}(\mathbb{R}/L\mathbb{Z})}
\le\|G_{t}*v^{[L]}\|_{L^\infty(\mathbb{R}/L\mathbb{Z})}+\int_{0}^{t}\|\partial_{x}G_{t-s}*((\Xi_{M}\circ\theta)(s,\cdot)+\psi^{[L]}(s,\cdot))^{2}\|_{\mathcal{C}_{\mathrm{b}}(\mathbb{R}/L\mathbb{Z})}\,\dif s\\
       & \le\|v^{[L]}\|_{L^\infty(\mathbb{R}/L\mathbb{Z})}+Ct^{1/2}\left((M+1)^{2}+\|\psi^{[L]}\|_{\mathcal{C}_{\mathrm{b}}((0,2];\mathcal{C}(\mathbb{R}/L\mathbb{Z}))}^{2}\right), 
  \end{aligned}
  \end{equation}
with some universal constant $C$, so that
\[
    \|\mathcal{A}_{M}\theta\|_{\mathcal{C}_{\mathrm{b}}((0,T];\mathcal{C}_{\mathrm{b}}(\mathbb{R}/L\mathbb{Z}))}\le\|v^{[L]}\|_{L^\infty(\mathbb{R}/L\mathbb{Z})}+CT^{1/2}\left((M+1)^{2}+\|\psi^{[L]}\|_{\mathcal{C}_{\mathrm{b}}((0,2];\mathcal{C}(\mathbb{R}/L\mathbb{Z}))}^{2}\right)\eqqcolon K_{M}.
\]
Let $\mathcal{B}$ be the ball of radius $K_{M}$ about the origin
in $\mathcal{C}_{\mathrm{b}}((0,T];\mathcal{C}_{\mathrm{b}}(\mathbb{R}/L\mathbb{Z}))$.
If $\theta,\tilde{\theta}\in\mathcal{B}$ and $t\in[0,T]$, then, once again, by 
\cref{lem:weightedhkbound}, we obtain
  \begin{align*}
    \| & \mathcal{A}_{M}\theta(t,\cdot)-\mathcal{A}_{M}\tilde{\theta}(t,\cdot)\|_{\mathcal{C}_{\mathrm{b}}(\mathbb{R}/L\mathbb{Z})}                                                                                                                                \\
       & \le\int_{0}^{t}\left\Vert \partial_{x}G_{t-s}*\left[((\Xi_{M}\circ\theta)(s,\cdot)+\psi^{[L]}(s,\cdot))^{2}-((\Xi_{M}\circ\tilde{\theta})(s,\cdot)+\psi^{[L]}(s,\cdot))^{2}\right]\right\Vert _{\mathcal{C}_{\mathrm{b}}(\mathbb{R}/L\mathbb{Z})}\,\dif s \\
       & \le\int_{0}^{t}(t-s)^{-{1}/{2}}\left\Vert ((\Xi_{M}\circ \theta)(s,\cdot)+\psi^{[L]}(s,\cdot))^{2}-((\Xi_{M}\circ\tilde{\theta})(s,\cdot)+\psi^{[L]}(s,\cdot))^{2}\right\Vert _{\mathcal{C}_{\mathrm{b}}(\mathbb{R}/L\mathbb{Z})}\,\dif s                  \\
       & \le CT^{1/2}\left(M+1+\|\psi^{[L]}\|_{\mathcal{C}_{\mathrm{b}}((0,2];\mathcal{C}(\mathbb{R}/L\mathbb{Z}))}\right)\|\theta-\tilde{\theta}\|_{\mathcal{C}_{\mathrm{b}}((0,T];\mathcal{C}_{\mathrm{b}}(\mathbb{R}/L\mathbb{Z}))},
  \end{align*}
  with another (possibly larger) constant $C$. Therefore, for all $\theta,\tilde{\theta}\in\mathcal{B}$,
  we have
  \[\begin{aligned}
    \|&\mathcal{A}_{M}\theta-\mathcal{A}_{M}\tilde{\theta}\|_{\mathcal{C}_{\mathrm{b}}((0,T];\mathcal{C}_{\mathrm{b}}(\mathbb{R}/L\mathbb{Z}))}\\&\le CT^{1/2}\left(M+1+\|\psi^{[L]}\|_{\mathcal{C}_{\mathrm{b}}((0,2];\mathcal{C}(\mathbb{R}/L\mathbb{Z}))}\right)\|\theta-\tilde{\theta}\|_{\mathcal{C}_{\mathrm{b}}((0,T];\mathcal{C}_{\mathrm{b}}(\mathbb{R}/L\mathbb{Z}))}.\end{aligned}
  \]
  Thus, if
  \begin{equation}
    T<\left(C\left(M+1+\|\psi^{[L]}\|_{\mathcal{C}_{\mathrm{b}}((0,2;\mathcal{C}(\mathbb{R}/L\mathbb{Z}))}\right)\right)^{-2},\label{eq:maxT}
  \end{equation}
  then the Banach fixed point theorem ensures the existence of a unique
  $\theta_{M}\in\mathcal{B}$ so that $\mathcal{A}_{M}\theta_{M}=\theta_{M}$,
  which is to say that
  \begin{equation}\label{eq:thetaMmildsoln}
    \theta_{M}(t,\cdot)=G_{t}*v^{[L]}-\int_{0}^{t}\partial_{x}G_{t-s}*\left((\Xi_{M}\circ\theta_{M})(s,\cdot)+\psi^{[L]}(s,\cdot)\right)^{2}\,\dif s 
  \end{equation}
  for all $t\in[0,T]$. Iterating this construction in the usual way,
  noting that the required bound \cref{eq:maxT} does not depend on the
  initial condition $v^{[L]}$, we conclude that $\theta_{M}\in\mathcal{C}_{\mathrm{b}}((0,2];\mathcal{C}_{\mathrm{b}}(\mathbb{R}/L\mathbb{Z}))$
  and \cref{eq:thetaMmildsoln} holds for all $t\in[0,2]$.

To remove the $M$-cutoff, define
\[
S_{M}=\inf\{t\in[0,2]\;:\;\|\theta_{M}(t,\cdot)\|_{L^\infty(\mathbb{R}/L\mathbb{Z})}\ge M\},
\]
or $S_{M}=2$ if $\|\theta_{M}(t,\cdot)\|_{L^\infty(\mathbb{R}/L\mathbb{Z})}<M$
for all $t\in[0,2]$.

First, we claim that $\|\theta_{M}(t,\cdot)\|_{L^\infty(\mathbb{R}/L\mathbb{Z})}$ is 
continuous at $t = 0$, so that $S_M>0$ for $M$ sufficiently large.
The estimate on the integral in \cref{eq:nonlin-strong} shows that its
 contribution to $\theta_M$ 
vanishes strongly at $t = 0$. This leaves only the term $G_t \ast v^{[L]}$ 
in \cref{eq:thetaMmildsoln} to estimate.
By \cref{lem:weightedhkcts}, we have
  \begin{equation*}
    \|v^{[L]}\|_{L^\infty(\mathbb{R}/L\mathbb{Z})} \leq \liminf_{t \downarrow 0} \|G_t \ast v^{[L]}\|_{L^\infty(\mathbb{R}/L\mathbb{Z})}.
  \end{equation*}
  On the other hand, the comparison principle implies
  \begin{equation*}
    \|v^{[L]}\|_{L^\infty(\mathbb{R}/L\mathbb{Z})} \geq \|G_t \ast v^{[L]}\|_{L^\infty(\mathbb{R}/L\mathbb{Z})}
  \end{equation*}
  for all $t \geq 0$.
  Together, these imply that $\|G_t \ast v^{[L]}\|_{L^\infty(\mathbb{R}/L\mathbb{Z})}$, and thus $\|\theta_{M}(t,\cdot)\|_{L^\infty(\mathbb{R}/L\mathbb{Z})}$, are continuous at $t = 0$.
  Therefore if $\|\theta_{M}(t,\cdot)\|_{L^\infty(\mathbb{R}/L\mathbb{Z})} < M$, we have $S_M > 0$.
  
Finally, it is clear from the uniqueness of $\theta_{M}$ that $S_{M}$ is
increasing as a function of $M$ and 
\[
\hbox{$\theta_{M}|_{[0,S_{M}]}\equiv\theta_{M'}|_{[0,S_{M}]}$
  whenever $M'\ge M$.}
\] 
Therefore, there exists
\[
\theta^{[L]}\in\bigcup_{S'\in(0, S)}
\mathcal{C}_{\mathrm{b}}((0,S'];\mathcal{C}(\mathbb{R}/L\mathbb{Z}))
\]
satisfying \cref{eq:mildformulation} for all $t \in (0,S)$, with 
\[
S = \lim\limits _{M\to\infty}S_{M}>0.
\]
Moreover, if $S<2$, then for each $M$ we have $\|\theta^{[L]}\|_{\mathcal{C}_{\mathrm{b}}((0,S_{M}];\mathcal{C}_{\mathrm{b}}(\mathbb{R}/L\mathbb{Z}))}\ge M$,
  and \cref{eq:limsuptoinfty} follows.
\end{proof}

\subsection{Global-in-time existence for the periodized problem\label{subsec:globalintime}}

The results of \cref{prop:localintimeexistence} are insufficient to
pass to the limit $L\to\infty$ because the time of existence $S$
is not uniform in $L$. In addition, the existence time $S$ is not
uniform in $\psi^{[L]}$ and $v^{[L]}$. We now  obtain a weighted bound on 
solutions that is independent of $L$, and 
shows that the solutions~$\theta^{[L]}$ can always be extended up
to time $t=1$ (and thus by iteration to all positive times). 
This last point is a new ingredient, compared
to, for example, the method of \cite{DPDT94}, since we need to control
solutions on the whole space.

The $\theta^{2}$ term in \cref{eq:thetaLproblem} is dangerous from
the perspective of global in time existence, and
we need to use the fact that it is inside a gradient. This is not
unrelated to our use (in \cref{sec:tightness} below) of the gradient
on the noise to obtain uniform-in-time bounds for solutions to \cref{eq:uPDE-1}.
However, the growth of the forcing and thus the solution at infinity
requires the use of a weighted space, which breaks the symmetry used
in \cite{DPDT94,DPG94} to eliminate the gradient term in \cref{eq:thetaLproblem}
altogether. In a sense, the gradient term simply moves mass around,
and in an unweighted space that does not affect the norm. In a weighted
space, however, mass moving closer to zero causes the norm to grow.
To control this growth, in the proof (but not the statement) of the following proposition
we use a custom-built family of weights that grow at a similar rate
to $\p_{\ell}$ far from the origin, but close to zero are much flatter.
This means that the effect
of mass moving closer to the origin is reduced. The amount of flatness required depends on the initial condition and the noise, so the required weight is in fact random.
\begin{prop}
  \label{prop:weightedspacesizebound}
  Fix $0  < \ell < \ell' < 1$ and $A<\infty$.
  Then there is a constant $C=C(\ell,\ell',A)<\infty$ so that the following holds.
  Fix $L\in(0,\infty)$ and $v^{[L]}\in L^{\infty}(\mathbb{R}/L\mathbb{Z})$
  with $\|v^{[L]}\|_{L_{\p_{\ell}}^{\infty}(\mathbb{R})}\le A$ and
  suppose that $\|\psi^{[L]}\|_{\mathcal{C}_{\mathrm{b}}([0,1];\mathcal{C}_{g}^{1}(\mathbb{R}))}\le A$, where $g$ is the weight defined in \cref{lem:suppsi}. If, for some~$S\in(0,1]$,
a function $\theta^{[L]}\in\mathcal{Z}_{S}^{[L]}$
solves \cref{eq:thetaLproblem}--\cref{eq:thetaLic}, then
  \begin{equation}
    \sup_{t\in[0,S]}\|\theta^{[L]}(t,\cdot)\|_{\mathcal{C}_{\p_{\ell'}}(\mathbb{R})}\le C.\label{eq:thetaLbound}
  \end{equation}
\end{prop}

The constant $C$ does not depend on $L$. This
will allow us to pass to a limit as $L\to \infty$.

\begin{proof}
Fix some $K\ge A$, to be chosen later, and $\ell < \ell_1 < \ell'$.
Let $\eps = \ell' - \ell_1$, and define
\begin{equation}
a(t,x)=K^{-1}\left(\langle x\rangle^{2}+K^{2/(1-\ell')}\right)^{-(\ell_1+\eps t)/2}.\label{eq:adef}
\end{equation}
The function $z=a\theta^{[L]}$ satisfies
  \begin{equation}
    \|z(0,\cdot)\|_{L^{\infty}(\mathbb{R})}\le1\label{eq:zbddby1}
  \end{equation}
  and also, for $t>0$,
  \begin{equation}
    \begin{aligned}\partial_{t}z & =z\partial_{t}(\log a)+\frac{1}{2}\left[\partial_{x}^{2}z-z\left(\partial_{x}^{2}(\log a)-(\partial_{x}(\log a))^{2}\right)-2(\partial_{x}z)\partial_{x}(\log a)\right] \\
      & \qquad-(a^{-1}z+\psi^{[L]})(\partial_{x}z-z\partial_{x}(\log a))-(z+a\psi^{[L]})\partial_{x}\psi^{[L]}.
    \end{aligned}
    \label{eq:dtz}
  \end{equation}
In addition, since $\theta^{[L]}(t,\cdot)$ is periodic and hence bounded in space,
  we have
  \begin{equation}
    \lim_{|x|\to\infty}z(t,x)=0\label{eq:zto0}
  \end{equation}
  for each $t\ge0$.

We claim that the map $t \mapsto \|z(t,\cdot)\|_{L^{\infty}(\mathbb{R})}$ is continuous.
Since $\theta^{[L]} \in \mathcal{Z}_S^{[L]}$, this is only in doubt at $t = 0$.
We again write the mild formulation \cref{eq:mildformulation} as 
\[
\theta^{[L]} = G_t \ast v^{[L]} + D_t,
\]
with $D_t$ as in \cref{eq:Duhamel-nonlinearity}.
By \cref{cor:nonlinearity-norm-continuous}, $D_t \to 0$ as $t \downarrow 0$ 
in $\mathcal{C}_{\p_{\ell_1}}(\R)$, so $a D_t \to 0$ in $L^\infty(\R)$.

  We thus need only consider $w(t, x) = a(t, x) [G_t \ast v^{[L]}(x)]$.
  We first show that $w(t, \cdot) \overset{\mathrm{w}^*}{\longrightarrow} w(0, \cdot)$ in $L^\infty(\mathbb{R})$.
  To do so, fix $\phi \in L^1(\mathbb{R})$.
  By duality and the symmetry of the heat semigroup, it suffices to show that
  \begin{equation*}
    G_t \ast[a(t, \cdot) \phi] \to a(0, \cdot) \phi
  \end{equation*}
  in $L^1(\mathbb{R})$ as $t \downarrow 0$.
  This follows from an approximation of the identity argument identical to that presented in the proof of \cref{lem:weightedhkcts}.
  Since $w$ is weak-$*$ continuous in $L^\infty(\R)$, we obtain
  \begin{equation}
    \label{eq:liminf-norm}
    \|w(0, \cdot)\|_{L^\infty(\R)} \leq \liminf_{t \downarrow 0} \|w(t, \cdot)\|_{L^\infty(\R)}.
  \end{equation}
To establish the opposite bound, we use the comparison principle:
\begin{equation*}
|G_t \ast v^{[L]}(x)| \leq \|a(0, \cdot) v^{[L]}\|_{L^\infty(\R)} [G_t \ast a(0, \cdot)^{-1}](x)
=\|w(0, \cdot)\|_{L^\infty(\R)} [G_t \ast a(0, \cdot)^{-1}](x).
\end{equation*}
We can easily check that $\e^{\kappa t} a(0, \cdot)^{-1}$ is a supersolution 
to the heat equation for $\kappa > 0$ sufficiently large depending on $\ell'$,
so that
\[
[G_t \ast a(0, \cdot)^{-1}](x)\le \farc{e^{\kappa t}}{a(0,x)}.
\]
Thus, we have
\begin{equation*}
\begin{aligned}
|w(t, x)| &\le a(t,x)|G_t \ast v^{[L]}(x)|\le 
a(t,x)\|w(0, \cdot)\|_{L^\infty(\R)} [G_t \ast a(0, \cdot)^{-1}](x)
\\
&\leq \e^{\kappa t} \|w(0, \cdot)\|_{L^\infty(\R)} \, \frac{a(t, x)}{a(0, x)} \leq \e^{\kappa t} \|w(0, \cdot)\|_{L^\infty(\R)},
\end{aligned}
  \end{equation*}
so
  \begin{equation*}
    \|w(0, \cdot)\|_{L^\infty(\R)} \geq \limsup_{t \downarrow 0} \|w(t, \cdot)\|_{L^\infty(\R)}.
  \end{equation*}
  In concert with \cref{eq:liminf-norm}, this implies
  \begin{equation*}
    \|w(0, \cdot)\|_{L^\infty(\R)} = \lim_{t \downarrow 0} \|w(t, \cdot)\|_{L^\infty(\R)}.
  \end{equation*}
  Since $a D_t \to 0$ in $L^\infty(\R)$ and $z = w + aD_t$, we have shown that $t \mapsto \|z(t,\cdot)\|_{L^{\infty}(\mathbb{R})}$ is continuous.

  In light of \cref{eq:zbddby1}, this implies that
  \[
    t_{*}\coloneqq\max\{t\in[0,S]\ :\ \|z(t,\cdot)\|_{L^{\infty}(\mathbb{R})}\le2\}
  \]
  exists and is positive. Moreover, if $t_{*}<S$, then $\|z(t_{*},\cdot)\|_{L^{\infty}(\mathbb{R})}=2$
  and, by \cref{eq:zto0}, there exists~$x_{*}\in\mathbb{R}$ so that
  \[
    |z(t_{*},x_{*})|=2=\max_{[0,t_{*}]\times\mathbb{R}}|z|.
  \]
As a consequence, we have
\[
\hbox{ $\sgn(z(t_{*},x_{*}))\partial_{t}z(t_{*},x_{*})\ge0$,~~~~$\partial_{x}z(t_{*},x_{*})=0$,
~~~~$\sgn(z(t_{*},x_{*}))\partial_{x}^{2}z(t_{*},x_{*})\le0$.}
\] Then
  \cref{eq:dtz} yields, at the point $(t_{*},x_{*})$,
  \begin{equation}
    0\le2\partial_{t}(\log a)+|\partial_{x}^{2}(\log a)|+|\partial_{x}(\log a)|^{2}+4|a^{-1}\partial_{x}(\log a)|+2|\psi^{[L]}||\partial_{x}(\log a)|+(2+a|\psi^{[L]}|)|\partial_{x}\psi^{[L]}|.\label{eq:dtztstarxstar}
  \end{equation}
  Note that
  \[
    |\partial_{x}(\log a)|=\left|-\frac{(\ell_1+\eps t)x}{\langle x\rangle^{2}+K^{2/(1-\ell')}}\right|\le1
  \]
  and
  \begin{align}
    \left|\frac{\partial_{x}(\log a)}{a}\right| & =(\ell_1+\eps t)|x|K(\langle x\rangle^{2}+K^{2/(1-\ell')})^{(\ell_1-\eps t)/2-1}\nonumber                                                                             \\
                                                & \le\frac{|x|}{(\langle x\rangle^{2}+K^{2/(1-\ell')})^{1/2}}\cdot\frac{K}{(\langle x\rangle^{2}+K^{2/(1-\ell')})^{(1-\ell')/2}}\le1,\label{eq:dtlogaovera}
  \end{align}
  since $t\le1$ and $\ell's<1$. We further compute
  \begin{equation}
    |\partial_{x}^{2}(\log a)|=\left|-\frac{\ell_1 +\eps t}{\langle x\rangle^{2}+K^{2/(1-\ell')}}+\frac{2(\ell_1+\eps t)x^{2}}{(\langle x\rangle^{2}+K^{2/(1-\ell')})^{2}}\right|\le3 \label{eq:dx2loga}
  \end{equation}
  and
  \begin{equation}
    \partial_{t}(\log a)=-\frac{\eps}{2}\log(\langle x\rangle^{2}+K^{2/(1-\ell')}).\label{eq:dtloga}
  \end{equation}
  Applying \cref{eq:adef} and \cref{eq:dtlogaovera}--\cref{eq:dtloga}
  to \cref{eq:dtztstarxstar}, and using the triangle inequality, we obtain
  \begin{equation}
    0\le-\eps\log(\langle x_{*}\rangle^{2}+K^{2/(1-\ell')})+8+\|\psi^{[L]}\|_{\mathcal{C}_{\mathrm{b}}([0,1];\mathcal{C}_{g}^{1}(\mathbb{R}))}g(x_{*})(4+\|\psi^{[L]}\|_{\mathcal{C}_{\mathrm{b}}([0,1];\mathcal{C}_{g}^{1}(\mathbb{R}))}g(x_{*})\langle x_{*}\rangle^{-\ell_1}).\label{eq:finalderivineq}
  \end{equation}
  Now we can choose $K$ large enough, depending on $\ell_1,\ell',$ and $A$,
  so that the right side of \cref{eq:finalderivineq} is guaranteed to
  be strictly negative regardless of $x_{*}$, which is a contradiction.

  Therefore, we have $|z|\le2$ on $[0,S]\times\mathbb{R}$, so
  \[
    |\theta^{[L]}(t,x)|\le2K(\langle x\rangle^{2}+K^{2/(1-\ell')})^{(\ell_1+\eps t)/2}
  \]
  for all $t\in[0,S]$, $x\in\mathbb{R}$, and \cref{eq:thetaLbound}
  follows.
\end{proof}
Now we can upgrade \cref{prop:localintimeexistence} to eliminate the
variable existence time $S$.
\begin{prop}
  \label{prop:Sis1}For each $L\in(0,\infty)$ and $v^{[L]}\in L^{\infty}(\mathbb{R}/L\mathbb{Z})$
  there exists a solution $\theta^{[L]}\in\mathcal{Z}_{1}^{[L]}$ to
  \cref{eq:thetaLproblem}--\cref{eq:thetaLic}.
\end{prop}

\begin{proof}
  By \cref{prop:weightedspacesizebound} and the fact that there is a
  continuous embedding from the space of $L$-periodic functions in
  $\mathcal{C}_{\p_{\ell}}(\mathbb{R})$ into $L^{\infty}(\mathbb{R}/L\mathbb{Z})$
  for every $\ell$ and $L$, \cref{eq:limsuptoinfty} implies that  $S > 1$ in \cref{prop:localintimeexistence}, so
$\theta^{[L]}\in\mathcal{Z}_{1}^{[L]}$.
\end{proof}

\subsection{Solutions on the whole space\label{subsec:solutionsonthewholespace}}

We now pass to the limit $L\to\infty$ to obtain global in time
solutions to \cref{eq:thetaLproblem}--\cref{eq:thetaLic}
on the whole space. In order to show that the sequence $\theta^{[L]}$
is Cauchy as $L\to\infty$, we will need some continuity of the solutions
to the periodic problem with respect to the forcing and the initial
conditions. The following proposition does this in a weaker topology,
which uses weights growing superexponentially at infinity.
\begin{prop}
  \label{prop:vCauchy}Fix $\ell\in(0,1)$, $T>0$. Suppose $v_{i}\in L_{\p_{\ell}}^\infty(\mathbb{R})$
  and $\theta_{i},\psi_{i}\in\mathcal{C}_{\mathrm{b}}((0,T];\mathcal{C}_{\p_{\ell}}(\mathbb{R}))$
  satisfy
  \[
    \theta_{i}(t,\cdot)=G_{t}*v_{i}-\int_{0}^{t}\partial_{x}G_{t-s}*(\theta_{i}(s,\cdot)+\psi_{i}(s,\cdot))^{2}\,\dif s
  \]
  for all $t\in[0,T]$ and $i\in\{1,2\}$. Fix $\beta\in(2\ell\vee(3/2),2)$
  and define the weight $\q_{\lambda}(x)=\exp(\lambda\langle x\rangle^{\beta})$
  as in \cref{lem:hk-superexp}. Then there exists a constant $C=C(\ell,\beta,T)<\infty$
  so that
  \begin{equation}
    \|\theta_{1}-\theta_{2}\|_{\mathcal{C}_{\mathrm{b}}((0,T];\mathcal{C}_{\q_{1+T}}(\mathbb{R}))}\le\e^{CX}\left(\|v_{1}-v_{2}\|_{L_{\p_{\ell}}^\infty(\mathbb{R})}+\|\psi_{1}-\psi_{2}\|_{\mathcal{C}_{\mathrm{b}}([0,T];\mathcal{C}_{\p_{\ell}}(\mathbb{R}))}\right),\label{eq:vCauchyineq}
  \end{equation}
where
\[
X=1+\sum_{i=1}^2\Big[
    \|v_i\|_{\mathcal{C}_{\mathrm{b}}((0,T];\mathcal{C}_{\p_{\ell}}(\mathbb{R}))}
    +\|\psi_i\|_{\mathcal{C}_{\mathrm{b}}((0,T];\mathcal{C}_{\p_{\ell}}(\mathbb{R}))}\Big].
\]
\end{prop}

\begin{proof}
  The proof is similar in spirit to that of \cite[Proposition 4.2]{HL15}.
  We begin with
  \begin{align}
    \|\theta_{1}(t,\cdot)-\theta_{2}(t,\cdot)\|_{\mathcal{C}_{\q_{1+t}}(\mathbb{R})} & \le\|G_{t}*(v_{1}-v_{2})\|_{\mathcal{C}_{\q_{1+t}}(\mathbb{R})}+\nonumber                                                                                                                                            \\
                                                                                     & \hspace{-1cm}\qquad+\int_{0}^{t}\left\Vert \partial_{x}G_{t-s}*[(\theta_{1}+\psi_{1})(s,\cdot)^{2}-(\theta_{2}+\psi_{2})(s,\cdot)^{2}]\right\Vert _{\mathcal{C}_{\q_{1+t}}(\mathbb{R})}\,\dif s\cdot\label{eq:triangleinequality}
  \end{align}
  By \cref{lem:weightedhkbound} and the fact that $|x|\le\e^{|x|}$ for
  all $x\in\mathbb{R}$, there is a constant $C<\infty$ so that
  \begin{equation}
    \|G_{t}*(v_{1}-v_{2})\|_{\mathcal{C}_{\q_{1+t}}(\mathbb{R})}\le C\|v_{1}-v_{2}\|_{\mathcal{C}_{\p_{\ell}}(\mathbb{R})}.\label{eq:firstconvbound}
  \end{equation}
  We handle the integral term in \cref{eq:triangleinequality} using \cref{lem:hk-superexp}:
  \begin{align}
    & \left|\partial_{x}G_{t-s}*[(\theta_{1}+\psi_{1})(s,\cdot)^{2}-(\theta_{2}+\psi_{2})(s,\cdot)^{2}](x)\right|\nonumber                                                                                                                                 \\
    & \quad\le C(t-s)^{-\frac{1}{2}}\e^{C(t-s)\langle x\rangle^{2(\beta-1)}}\q_{1+\frac{1}{2}(s+t)}(x)\|(\theta_{1}+\psi_{1})(s,\cdot)^{2}-(\theta_{2}+\psi_{2})(s,\cdot)^{2}\|_{\mathcal{C}_{\q_{1+\frac{1}{2}(s+t)}}(\mathbb{R})}.\label{eq:partialxGts}
  \end{align}
  We note that, as $0<\beta<2$, we have
  \begin{equation}
    \sup_{x\in\mathbb{R}}\frac{\e^{C(t-s)\langle x\rangle^{2(\beta-1)}}\q_{1+\frac{1}{2}(s+t)}(x)}{\q_{1+t}(x)}=\exp\left\{ (t-s)\sup_{x\in\mathbb{R}}\left[C\langle x\rangle^{2(\beta-1)}-\frac{1}{2}\langle x\rangle^{\beta}\right]\right\} \le\e^{C(t-s)}\label{eq:exponentialthing}
  \end{equation}
  for some new constant $C$. We also have
  \begin{align}
    \|(\theta_{1}+\psi_{1})(s,\cdot)^{2} & -(\theta_{2}+\psi_{2})(s,\cdot)^{2}\|_{\mathcal{C}_{\q_{1+\frac{1}{2}(s+t)}}(\mathbb{R})}\nonumber                                                                                                                                  \\
                                         & \le\|(\theta_{1}+\theta_{2}+\psi_{1}+\psi_{2})(s,\cdot)\|_{\mathcal{C}_{\q_{\frac{1}{2}(t-s)}}(\mathbb{R})}\|(\theta_{1}-\theta_{2}+\psi_{1}-\psi_{2})(s,\cdot)\|_{\mathcal{C}_{\q_{1+s}}(\mathbb{R})},\label{eq:splitupsquarediff}
  \end{align}
  and  
  \begin{align}
    \|(\theta_{1}+\theta_{2}+\psi_{1}+\psi_{2})(s,\cdot)\|_{\mathcal{C}_{\q_{\frac{1}{2}(t-s)}}(\mathbb{R})} & \le\left(\sup_{y\in\mathbb{R}}\frac{\p_{\ell}(y)}{\q_{\frac{1}{2}(t-s)}(y)}\right)\|(\theta_{1}+\theta_{2}+\psi_{1}+\psi_{2})(s,\cdot)\|_{\mathcal{C}_{\p_{\ell}}(\mathbb{R})}\nonumber \\
                                                                                                             & \le C(t-s)^{-\ell/\beta}\|(\theta_{1}+\theta_{2}+\psi_{1}+\psi_{2})(s,\cdot)\|_{\mathcal{C}_{\p_{\ell}}(\mathbb{R})}\label{eq:ptoq}
  \end{align}
  for another constant $C$. Using the bounds \cref{eq:exponentialthing}--\cref{eq:ptoq}
  in \cref{eq:partialxGts}, we obtain
  \begin{align}
    & \|\partial_{x}G_{t-s}*[(\theta_{1}+\psi_{1})(s,\cdot)^{2}-(\theta_{2}+\psi_{2})(s,\cdot)^{2}\|_{\mathcal{C}_{\q_{1+t}}(\mathbb{R})}\nonumber                                                                                                                               \\
    & \ \le C\e^{C(t-s)}(t-s)^{-\frac{1}{2}-\frac{\ell}{\beta}}\|(\theta_{1}+\theta_{2}+\psi_{1}+\psi_{2})(s,\cdot)\|_{\mathcal{C}_{\p_{\ell}}(\mathbb{R})}\|(\theta_{1}-\theta_{2}+\psi_{1}-\psi_{2})(s,\cdot)\|_{\mathcal{C}_{\q_{1+s}}(\mathbb{R})}.\label{eq:finalconvbound}
  \end{align}
  Using \cref{eq:firstconvbound} and \cref{eq:finalconvbound} in \cref{eq:triangleinequality},
  we obtain
  \[
    \|\theta_{1}(t,\cdot)-\theta_{2}(t,\cdot)\|_{\mathcal{C}_{\q_{1+t}}(\mathbb{R})}\le A+CX\int_{0}^{t}\e^{C(t-s)}(t-s)^{-\frac{1}{2}-\frac{\ell}{\beta}}\|(\theta_{1}-\theta_{2})(s,\cdot)\|_{\mathcal{C}_{\q_{1+s}}(\mathbb{R})},
  \]
  where
  \[
    A=C\|v_{1}-v_{0}\|_{\mathcal{C}_{\p_{\ell}}(\mathbb{R})}+CX\|\psi_{1}-\psi_{2}\|_{\mathcal{C}_{\mathrm{b}}([0,T];\mathcal{C}_{\q_{1}}(\mathbb{R}))}.
  \]
  Therefore, Grönwall's inequality implies
  \[
    \|\theta_{1}(t,\cdot)-\theta_{2}(t,\cdot)\|_{\mathcal{C}_{\q_{1+t}}(\mathbb{R})}\le A\exp\left\{ CX\int_{0}^{t}\e^{C(t-s)}(t-s)^{-\frac{1}{2}-\frac{\ell}{\beta}}\,\dif s\right\} ,
  \]
  and \cref{eq:vCauchyineq} follows.
\end{proof}
We can now take $L\to\infty$ and prove \cref{prop:thetawellposed}.
\begin{proof}[Proof of \cref{prop:thetawellposed}.]
It suffices to prove that that there exists such a $\theta\in\mathcal{Z}{}_{m,1}$;
then the result follows by iteration from $t=1$ to $t=2$, etc.
Fix constants $m<\ell_{1}<\ell_{2}<\ell$, and note 
that for~$v\in L_{\p_{\ell_{1}}}^{\infty}(\mathbb{R})$ we have
  \begin{equation}
    \lim_{L\to\infty}\|v^{[L]}-v\|_{L_{\p_{{\ell}_1}}^{\infty}(\mathbb{R})}=0,\label{eq:vLtov}
  \end{equation}
  and
  \begin{equation}
    \|v^{[L]}\|_{L_{\p_{\ell_{1}}}^{\infty}(\mathbb{R})}\le 3\|v\|_{L_{\p_{\ell_{1}}}^{\infty}(\mathbb{R})}.\label{eq:Vl2v}
  \end{equation}
  \cref{prop:Sis1} implies that there exists a solution $\theta^{[L]}\in\mathcal{Z}_{1}^{[L]}$
  to \cref{eq:thetaLproblem}--\cref{eq:thetaLic}, while \cref{prop:weightedspacesizebound},
the bound  \cref{eq:Vl2v}, and \cref{lem:suppsi} imply that there is a number $Y<\infty$,
  depending on~$\|v\|_{L_{\p_{\ell_{1}}}^{\infty}(\mathbb{R})}$, $\psi$,
  $\ell_{1}$, and $\ell_{2}$, but \emph{not} on $L$, so that
  \begin{equation}\label{eq:bddbyY}
    \|\theta^{[L]}\|_{\mathcal{C}_{\mathrm{b}}((0,1];\mathcal{C}_{\p_{\ell_{2}}}(\mathbb{R}))},\|\psi^{[L]}\|_{\mathcal{C}_{\mathrm{b}}([0,1]\mathcal{C}_{\p_{\ell_{2}}}(\mathbb{R}))}\le Y, 
  \end{equation}
  for each $L\in[1, \infty)$. It is also clear from \cref{eq:psiLdef}
  that
  \begin{equation}
    \lim_{L\to\infty}\|\psi^{[L]}-\psi\|_{\mathcal{C}_{\mathrm{b}}([0,1];\mathcal{C}_{\p_{\ell_{1}}}(\mathbb{R}))}=0.\label{eq:psiLtopsi}
  \end{equation}
  Hence by \cref{prop:vCauchy}, there is a constant $C=C(\ell_{1})$
  so that
  \begin{equation}
    \|\theta^{[L]}-\theta^{[L']}\|_{\mathcal{C}_{\mathrm{b}}((0,1];\mathcal{C}_{\q_{2}}(\mathbb{R}))}\le\e^{CY}\left(\|v^{[L]}-v^{[L']}\|_{\mathcal{C}_{\p_{\ell_{1}}}(\mathbb{R})}+\|\psi^{[L]}-\psi^{[L']}\|_{\mathcal{C}_{\mathrm{b}}((0,1];\mathcal{C}_{\p_{\ell_{1}}}(\mathbb{R}))}\right),\label{eq:psiLbd}
  \end{equation}
  where $\q_{\lambda}$ is defined as in the statement of \cref{prop:vCauchy}
  for some $\beta\in((2\ell_{1})\vee(3/2),2)$ fixed. It follows from
  \cref{eq:vLtov} and \cref{eq:psiLtopsi} that the right side of
  \cref{eq:psiLbd} goes to $0$ as $L,L'\to\infty$, so the left 
  side does as well. Hence there is a $\theta\in\mathcal{C}_{\mathrm{b}}((0,1];\mathcal{C}_{\q_{2}}(\mathbb{R}))$
  such that
  \[
    \lim_{L\to\infty}\|\theta^{[L]}-\theta\|_{\mathcal{C}_{\mathrm{b}}((0,1];\mathcal{C}_{\q_{2}}(\mathbb{R}))}=0.
  \]
  By \cref{eq:bddbyY} and \cref{prop:boosttheweight}, this implies that
  in fact $\theta\in\mathcal{C}_{\mathrm{b}}((0,1];\mathcal{C}_{\p_{\ell}}(\mathbb{R}))$
  and
  \begin{equation}
    \lim_{L\to\infty}\|\theta^{[L]}-\theta\|_{\mathcal{C}_{\mathrm{b}}((0,1];\mathcal{C}_{\p_{\ell}}(\mathbb{R}))}=0.\label{eq:thetaLtotheta}
  \end{equation}

  We claim that $\theta$ is a mild solution to \cref{eq:thetajPDE},
  i.e. that $\theta$ satisfies \cref{eq:mildformulation} with $L=\infty$.
  By \cref{eq:vLtov} and \cref{lem:weightedhkbound}, we have
  \begin{equation}
    \lim_{L\to\infty}G_{t}*v^{[L]}(x)=G_{t}*v(x)\label{eq:GtvLtoGtv}
  \end{equation}
  for each $(t,x)\in[0,1]\times\mathbb{R}$. Also, we have
  \begin{align*}
    \int_{0}^{t} & \|\partial_{x}G_{t-s}*[(\theta(s,\cdot)+\psi(s,\cdot))^{2}-(\theta^{[L]}(s,\cdot)+\psi^{[L]}(s,\cdot))^{2}\|_{\mathcal{C}_{\p_{2\ell_{2}}}(\mathbb{R})}\,\dif s                   \\
                 & \le\int_{0}^{t}(t-s)^{-\frac{1}{2}}\|(\theta(s,\cdot)+\psi(s,\cdot))^{2}-(\theta^{[L]}(s,\cdot)+\psi^{[L]}(s,\cdot))^{2}\|_{\mathcal{C}_{\p_{2\ell_{2}}}(\mathbb{R})}\,\dif s\to0
  \end{align*}
  as $L\to\infty$ by \cref{lem:weightedhkbound}, \cref{eq:psiLtopsi},
  and \cref{eq:thetaLtotheta}, which in particular means that for each
  $(t,x)\in[0,1]\times\mathbb{R}$ we have
  \begin{equation}
    \lim_{L\to\infty}\int_{0}^{t}\partial_{x}G_{t-s}*(\theta^{[L]}(s,\cdot)+\psi^{[L]}(s,\cdot))^{2}(x)\,\dif s=\int_{0}^{t}\partial_{x}G_{t-s}*(\theta(s,\cdot)+\psi(s,\cdot))^{2}(x)\,\dif s.\label{eq:nonlineartermtononlinearterm}
  \end{equation}
  Since we also have $\theta^{[L]}(t,x)\to\theta(t,x)$ as $L\to\infty$,
  and each $\theta^{[L]}$ satisfies \cref{eq:mildformulation}, we have
  from \cref{eq:GtvLtoGtv} and \cref{eq:nonlineartermtononlinearterm}
  that $\theta$ satisfies \cref{eq:mildformulation} with $L=\infty$.

  The measurability in the statement of \cref{prop:thetawellposed} is
  obvious, and the uniqueness of $\theta$ and the continuity of the
  map $v\mapsto\theta$ follow immediately from \cref{prop:vCauchy,prop:boosttheweight}.

  Now suppose $A \subset L_{\p_{m+}}^\infty(\R)$ is bounded, and hence bounded in $L_{\p_{\ell_1}}^\infty(\R)$.
  Take $v \in A$ and fix~$\ell_3 \in (\ell_2, \ell)$.
  By \cref{eq:bddbyY}, we have
 \[
 \|\theta\|_{\mathcal{C}_{\mathrm{b}}((0,1];\mathcal{C}_{\p_{\ell_3}}(\mathbb{R}))} \leq Y
 <+\infty,
\]
with some $Y$ that depends on $\psi, \ell_1, \ell_2$, and $A$.
  In particular, $\Phi(A)|_{(0, 1] \times \R}$ is bounded in the $\mathcal{C}_{\mathrm{b}}((0, 1]; \mathcal{C}_{\p_\ell}(\R))$ norm.
  For $t \in (0, 1]$, \cref{eq:C1bound} yields
  \begin{equation*}
    \|\theta(t, \cdot)\|_{\mathcal{C}_{\p_4}^1(\mathbb{R}))} \leq C(t, \psi, Y) < \infty.
  \end{equation*}
  Thus by \cref{prop:arzelaascoli}, $\Phi(A)|_{\{t\} \times \R}$ is compact in $\mathcal{C}_{\p_{\ell}}(\R)$.
  Since $\ell > m$ was arbitrary, $\Phi(A)|_{[0, 1] \times \R}$ is bounded in $\mathcal{Z}_{m, 1}$ and \cref{lem:Xmcompactness} implies that $\Phi(A)|_{\{t\} \times \R}$ is compact in $\mathcal{X}_m$.
\end{proof}
As we have mentioned, \cref{thm:existence} follows from 
\cref{prop:thetawellposed,prop:fellerproperty}.

\section{Comparison principle and \texorpdfstring{$L^{1}$}{L¹}-contraction\label{sec:comparison-L1}}

The uniqueness and stability results rely crucially on the comparison
principle and $L^{1}$-contraction that are well known for the deterministic Burgers equation \cref{eq:uPDE-many} with $V\equiv0$; 
see for example~\cite{Daf16,Ser04} and references therein.
Here, we establish these properties for the stochastic Burgers equation. The proofs are similar to those in the deterministic case, but some care 
is required to deal with the growth of solutions at infinity. Also, we will need the $L^1$-contraction 
with ``$L^1$'' interpreted separately as~$L^1(\R)$ and as $L^1(\Omega\times I)$ for an interval $I$. (Recall that $\Omega$ is the probability space.)
These are different statements, the latter being irrelevant in the deterministic case.

Moreover, in \cref{subsec:ordering} below we prove an ordering property for time-stationary solutions. This is a novel element here,
as it holds for invariant measures and does not require a comparison at a fixed
initial time.

Throughout, we rely on an equation for the difference of two solutions.
If $\mathbf{u}=(u_{1},u_{2})\in\mathcal{Z}_{m}^{2}$ is a solution
to the system \cref{eq:uPDE-many}, and we define $\psi$ as in \cref{eq:psidef},
then we have
\[
  \eta=u_{1}-u_{2}=(u_{1}-\psi)-(u_{2}-\psi),
\]
and each of the two terms in parentheses satisfies the PDE \cref{eq:thetajPDE},
with the corresponding initial conditions. 
Subtracting these two copies of \cref{eq:thetajPDE}, we see
that $\eta$ is differentiable in time and satisfies the partial differential equation
\begin{align}
 & \partial_{t}\eta  =\frac{1}{2}\partial_{x}^{2}\eta-\frac{1}{2}\partial_{x}(\eta\xi),\label{eq:etaPDE} \\
  &\eta(0,x)     =u_{1}(0,x)-u_{2}(0,x)\label{eq:etaic}
\end{align}
almost surely, with $\xi=u_{1}+u_{2}$.

\subsection{Pathwise results}

First we state the almost-sure comparison, contraction, and conservation properties, which involve the whole space $\R$.
\begin{prop}[Comparison principle]
  \label{prop:comparisonprinciple}Fix $m\in(0,1)$. If $\mathbf{u}=(u_{1},u_{2})\in\mathcal{Z}_{m}^{2}$
  solves \cref{eq:uPDE-many} and satisfies $u_{1}(0,x)\le u_{2}(0,x)$
  for all $x\in\mathbb{R}$, then $u_{1}(t,x)\le u_{2}(t,x)$ for all
  $t\ge 0$ and $x\in\mathbb{R}$.
\end{prop}

\begin{prop}[$L^{1}(\mathbb{R})$-contraction]
  \label{prop:L1Rcontraction}Fix $m\in(0,1)$. If $\mathbf{u}=(u_{1},u_{2})\in\mathcal{Z}_{m}^{2}$
  solves \cref{eq:uPDE-many}, then for all $t\ge 0$ we have
  \begin{equation}
    \|u_{1}(t,\cdot)-u_{2}(t,\cdot)\|_{L^{1}(\mathbb{R})}\le\|u_{1}(0,\cdot)-u_{2}(0,\cdot)\|_{L^{1}(\mathbb{R})}.\label{eq:L1decreasing}
  \end{equation}
\end{prop}

\begin{prop}[Conservation of mass]
  \label{prop:conservation}Fix $m\in(0,1)$. If $\mathbf{u}=(u_{1},u_{2})\in\mathcal{Z}_{m}^{2}$
  solves \cref{eq:uPDE-many} and
  \begin{equation}
    \|u_1(0,\cdot)-u_2(0,\cdot)\|_{L^1(\R)}<\infty,\label{eq:L1startsfinite}
  \end{equation}
  then for all $t\ge 0$ we have
  \begin{equation}
    \int_\R [u_1(t,x)-u_2(t,x)]\,\dif x = \int_\R[u_1(0,x)-u_2(0,x)]\,\dif x.\label{eq:meanpreserved}
  \end{equation}
\end{prop}

Of course, \cref{eq:L1decreasing} has no content unless the right 
side is finite (i.e. if \cref{eq:L1startsfinite} holds). Note that we will not use \cref{prop:L1Rcontraction} 
or \cref{prop:conservation} in the sequel. They are included for completeness and
because they can be proved very
similarly to \cref{prop:comparisonprinciple} (which we will indeed use).
\cref{lem:Fdecreasing} below is the core of the argument.
First, we make the following elementary observation.
\begin{lem}
  \label{lem:approxconvex}If $F:\mathbb{R}\to\mathbb{R}_{\geq 0}$ satisfies $F(0)=0$ and
  $F''=c_{1}\delta_{0}$ with some $c_{1}\ge 0$, then there
  is a family~$F_{\eps}\in\mathcal{C}^{2}(\mathbb{R})$
  of convex functions   that has the following properties:
  \begin{enumerate}
  \item There is a constant $C<\infty$ so that for all $\eps \in (0, 1]$ and $x \in \R$, we have
    \begin{align}
    &  F_{\eps}(x)       \le C(|x| + \eps),\label{eq:Fbd}                           \\
     & |xF_{\eps}'(x)|   \le CF_{\eps}(x),\label{eq:Fprimebd}                   \\
     & |F_{\eps}'(x)|    \le C,\label{eq:Fepsprimebounded}                      \\
     & |x|F_{\eps}''(x)  \le C\mathbf{1}_{[0,\eps]}(x).\label{eq:Fprimeprimebd}
    \end{align}
  \item The restriction $F_{\eps}|_{\mathbb{R}\setminus[-1,1]}$ is independent
    of $\eps$.
  \item We have
    \begin{equation}
      \lim_{\eps\to0}\|F_{\eps}-F\|_{\mathcal{C}_{\mathrm{b}}(\mathbb{R})}=0.\label{eq:FepstoF}
    \end{equation}
  \end{enumerate}
\end{lem}

\begin{proof}
  If $F$ is as in the statement of the lemma, then there is some $c_{2}\in\mathbb{R}$
  so that $F(x)=c_{1}|x|+c_{2}x$. From this it is straightforward to
  construct such a family directly.
\end{proof}

\cref{prop:comparisonprinciple,prop:L1Rcontraction} are special cases of the following lemma.
\begin{lem}
  \label{lem:Fdecreasing}
  Let $F:\mathbb{R}\to\mathbb{R}_{\ge0}$ satisfy $F(0)=0$ and
  $F''=c_{1}\delta_{0}$ with some $c_{1}\ge 0$.  Assume that~$m\in(0,1)$,
  and ~$\mathbf{u}=(u_{1},u_{2})\in\mathcal{Z}_{m}$ solves \cref{eq:uPDE-many},  
  and set $\eta=u_{1}-u_{2}$, then for all $t\ge0$, we have
  \begin{equation}
    \int_{\mathbb{R}}[F\circ\eta](t,x)\,\dif x + \frac{c_1}{4}\int_{0}^{t}\left[\sum_{y\in\eta(s,\cdot)^{-1}(0)}|\eta'(y)|\right]\,\dif s \le \int_{\mathbb{R}}[F\circ\eta](0,x)\,\dif x.\label{eq:Fdecreasing}
  \end{equation}
\end{lem}

Before we prove \cref{lem:Fdecreasing}, let us show how it implies \cref{prop:comparisonprinciple,prop:L1Rcontraction}.
\begin{proof}[Proof of \cref{prop:comparisonprinciple,prop:L1Rcontraction}.]
We use \cref{lem:Fdecreasing} with the nonnegative convex function 
\[
F(x)=x^{+}\coloneqq\max\{x,0\}
\]
that satisfies $F''=\delta_{0}$. This gives
\[
\|\eta(s,\cdot)^{+}\|_{L^{1}(\mathbb{R})}\le\|\eta(0,\cdot)^{+}\|_{L^{1}(\mathbb{R})}=0
\]
for all $s\in[0,T]$, so $\eta(s,\cdot)^{+}=0$ almost everywhere,
meaning $u_{1}\le u_{2}$. Similarly, using \cref{eq:Fdecreasing} with the function~$F(x)=|x|$, which satisfies $F''=2\delta_{0}$,
  implies \cref{eq:L1decreasing}.
\end{proof}
We now prove \cref{lem:Fdecreasing}.
\begin{proof}[Proof of \cref{lem:Fdecreasing}.]

  We first define an appropriate cutoff function. Fix
  $\ell\in(m,1)$ and define 
  \[
  \zeta(x)=\e^{2^{1-\ell}-\langle x\rangle^{1-\ell}},
  \]
  which
  is positive, decreasing in $|x|$, and in the Schwartz class. (The constant $2$ corresponds to the constant $4$ in \cref{eq:weightdefs} and is fixed to obtain \cref{eq:zetadeltato1} below.) Also,
  there is a constant $C$ so that
  \[
    \p_{\ell}(x)|\zeta'(x)| + |\zeta''(x)|\le C\zeta(x)
  \]
  for all $x\in\mathbb{R}$. For $\delta\in(0,1]$, define the rescaled
  version
  \[
    \zeta_{\delta}(x)\coloneqq\zeta(\delta x),
  \]
  which satisfies
  \begin{equation}
    \p_{\ell}(x)|\zeta'_{\delta}(x)|=\delta\p_{\ell}(x)|\zeta'(\delta x)|\le\delta^{1-\ell}\p_{\ell}(\delta x)|\zeta'(\delta x)|\le C\delta^{1-\ell}|\zeta_{\delta}(x)|\label{eq:zetaprimebd}
  \end{equation}
  and
  \begin{equation}
    |\zeta_{\delta}''(x)|=\delta^{2}|\zeta''(\delta x)|\le C\delta^{2}|\zeta_{\delta}(x)|.\label{eq:zetaprimeprimebd}
  \end{equation}
  Moreover, $\zeta_{\delta}(x)$ is decreasing in $\delta$ and, for
  each $x\in\mathbb{R}$, we have
  \begin{equation}
    \lim_{\delta\downarrow0}\zeta_{\delta}(x)=1.\label{eq:zetadeltato1}
  \end{equation}

  Let $\{F_{\eps}\}_{\eps\in(0,1]}$ be as in \cref{lem:approxconvex} and define
  \[
    I_{\eps,\delta}(t)\coloneqq\int_{\mathbb{R}}(F_{\eps}\circ\eta)(t,x))\zeta_{\delta}(x)\,\dif x. 
    \]
  By \cref{lem:weightedhkcts}, $\eta$ is strongly continuous in $L_{\p_2}^1(\R)$.
  Since $\zeta_\delta$ decays much faster than $\p_2$, \cref{eq:Fbd}  and~\cref{eq:Fprimebd} imply that $I_{\eps, \delta}$ is continuous in $t$.
  
  Applying the chain rule to \cref{eq:etaPDE} and integrating by parts, we obtain
  \begin{align*}
    \frac{\dif}{\dif t}I_{\eps,\delta} & =\frac{1}{2}\int_{\mathbb{R}}\left[(F_{\eps}\circ\eta)''-((F_{\eps}'\circ\eta)\eta\xi)'-(F_{\eps}''\circ\eta)(\eta')^{2}+(F_{\eps}''\circ\eta)(\xi\eta\eta')\right]
    \zeta_{\delta}  \,\dif x     \\
                                       & =\frac{1}{2}\int_{\mathbb{R}}\left[(F_{\eps}\circ\eta)\zeta_{\delta}''+(F_{\eps}'\circ\eta)\eta\xi\zeta_{\delta}'-(F_{\eps}''\circ\eta)[(\eta')^{2}-\xi\eta\eta']\zeta_{\delta}\right]\,\dif x.
  \end{align*}
  The boundary terms vanish because $\zeta_{\delta}$ is in the Schwartz class.
  Now Young's inequality yields
  \begin{equation}
    \frac{\dif}{\dif t}I_{\eps,\delta}\le\frac{1}{2}\int_{\mathbb{R}}\left\{(F_{\eps}\circ\eta)\zeta_{\delta}''+(F_{\eps}'\circ\eta)\eta\xi\zeta_{\delta}'+\frac{1}{2}(F_{\eps}''\circ\eta)\left[-(\eta')^{2}+\xi^{2}\eta^{2}\right]\zeta_{\delta}\right\}\,\dif x.\label{eq:young}
  \end{equation}
  We have by \cref{eq:zetaprimeprimebd} that
  \[
    \int_{\mathbb{R}}|(F_{\eps}\circ\eta)\zeta_{\delta}''|\,\dif x\le C\delta^{2}\int_{\mathbb{R}}(F_{\eps}\circ\eta)\zeta_{\delta}\,\dif x,
  \]
  and by \cref{eq:Fprimebd} and \cref{eq:zetaprimebd} that
  \[
    \int_{\mathbb{R}}|(F_{\eps}'\circ\eta)\eta\xi\zeta_{\delta}'|\,\dif x
    \le C\|\xi\|_{\mathcal{C}_{\p_{\ell}}(\mathbb{R})}\int_{\mathbb{R}}|(F_{\eps}\circ\eta)\p_{\ell}\zeta_{\delta}'|\,\dif x
    \le C\delta^{1-\ell}\|\xi\|_{\mathcal{C}_{\p_{\ell}}(\mathbb{R})}\int_{\mathbb{R}}(F_{\eps}\circ\eta)\zeta_{\delta}\,\dif x.
  \]
  Also note that, by \cref{eq:Fprimeprimebd}, we have
  \[
    \int_{\mathbb{R}}|(F_{\eps}''\circ\eta)\xi^{2}\eta^{2}\zeta_{\delta}|\,\dif x
    \le C\eps\int_{\mathbb{R}}|\xi|^{2}\zeta_{\delta}\,\dif x\le C\eps\|\xi\|_{\mathcal{C}_{\p_{\ell}}(\mathbb{R})}^{2}\|\p_{\ell}^{2}\zeta_{\delta}\|_{L^{1}(\mathbb{R})}.
  \]
  Substituting the last three displays into \cref{eq:young}, we obtain
  \[
    \frac{\dif}{\dif t}I_{\eps,\delta}\le C(\delta^{2}+\delta^{1-\ell}\|\xi\|_{\mathcal{C}_{\p_{\ell}}(\mathbb{R})})I_{\eps,\delta}+\eps\|\xi\|_{\mathcal{C}_{\p_{\ell}}(\mathbb{R})}^{2}\|\p_{\ell}^{2}\zeta_{\delta}\|_{L^{1}(\mathbb{R})}-\frac{1}{4}\int_{\mathbb{R}}(F_{\eps}''\circ\eta)(\eta')^{2}\zeta_{\delta}\,\dif x.
  \]
Integrating in time, the continuity of $I_{\eps,\delta}$ implies
  \begin{align}
    I_{\eps,\delta}(t)-I_{\eps,\delta}(0) & \le B_t\delta^{1-\ell}\int_{0}^{t}I_{\eps,\delta}(s)\,\dif s+t\eps B_t^{{2}} \|\p_{\ell}^{2}\zeta_{\delta}\|_{L^{1}(\mathbb{R})}
    -\frac{1}{4}\int_{0}^{t}\int_{\mathbb{R}}(F_{\eps}''\circ\eta(s,\cdot))(\eta')^{2}(s,\cdot)\zeta_{\delta}\,\dif s,
                                            \label{eq:integrateI}
  \end{align}
  where we have defined
  \[
    B_t= C\Big(\delta^{1+\ell}+\sup_{s\in[0,t]}\|\xi(s,\cdot)\|_{\mathcal{C}_{\p_{\ell}}(\mathbb{R})}\Big).
  \]
Let us look at the last term in the right 
  side of \cref{eq:integrateI}. By the coarea formula, we have
  \[
    \int_{\mathbb{R}}F_{\eps}''(\eta(s,x))\eta'(s,x)^{2}\zeta_{\delta}(x)\,\dif x=\int_{\mathbb{R}}F_{\eps}''(\lambda)\Big[\sum_{y\in\eta(s,\cdot)^{-1}(\lambda)}|\eta'(s,y)|\zeta_{\delta}(y)\Big]\,\dif\lambda.
  \]
 Using \cref{lem:gcts} and the fact that $F''_{\eps}$ converges weakly to $c_1\delta_0$ as $\eps\to0$, we obtain
  \begin{equation}
    \lim_{\eps\downarrow0}\int_{0}^{t}\int_{\mathbb{R}}F_{\eps}''(\eta(s, x))\eta'(s, x)^{2}\zeta_{\delta}(x)\,\dif x\,\dif s
    =c_1\int_{0}^{t}\Big[\sum_{y\in\eta(s,\cdot)^{-1}(0)}|\eta'(y)|\zeta_{\delta}(s, y)\Big]\,\dif s.\label{eq:usethatitsadeltamass}
  \end{equation}
In addition, by \cref{eq:Fbd}, \cref{eq:FepstoF}, the fact that $\eta\in\mathcal{C}_{\mathrm{b}}((0,t];\mathcal{C}_{\p_{\ell}}(\mathbb{R}))$,
  and the dominated convergence theorem, we have for each fixed $\delta\in(0,1]$
  that
  \begin{align}
    \lim_{\eps\downarrow0}I_{\eps,\delta}(s) & =\int_{\mathbb{R}}(F\circ\eta)(s,x)\zeta_{\delta}(x)\,\dif x, 
    & \lim_{\eps\downarrow0}\int_{0}^{t}I_{\eps,\delta}(s)\,\dif s & =\int_{0}^{t}\int_{\mathbb{R}}(F\circ\eta)(s,x)\zeta_{\delta}(x)\,\dif x\,\dif s.\label{eq:limitsofIs}
  \end{align}
  We pass to the limit  $\eps\downarrow0$ in \cref{eq:integrateI} and apply
  \cref{eq:usethatitsadeltamass} and \cref{eq:limitsofIs} to obtain
  \begin{align*}
    I_{\delta}(t) &\coloneqq\int_{\mathbb{R}}(F\circ\eta)(t,x)\zeta_{\delta}(x)\,\dif x=\lim_{\eps\downarrow0}I_{\eps,\delta}(t)                                                                           \\
                  & \le I_{\delta}(0)+B_{t}\delta^{1-\ell}\int_{0}^{t}I_{\delta}(s)\,\dif s-\frac{c_1}{4}\int_{0}^{t}\Big[\sum_{y\in\eta(s,\cdot)^{-1}(0)}|\eta'(y)|\zeta_{\delta}(y)\Big]\,\dif s.
  \end{align*}
As $B_t$ is increasing in $t$,  it follows from the Grönwall inequality that
  \[
    I_{\delta}(t)\le I_{\delta}(0)-\frac{c_1}{4}\int_{0}^{t}\Big[\sum_{y\in\eta(s,\cdot)^{-1}(0)}|\eta'(y)|\zeta_{\delta}(y)\Big]\,\dif s+I_{\delta}(0)tB_{t}\delta^{1-\ell}\exp\left\{ \delta^{1-\ell}tB_{t}\right\} .
  \]
  By the monotone convergence theorem and \cref{eq:zetadeltato1}, we
conclude that
\[
I(t):=\int_{\mathbb{R}}(F\circ\eta)(t,x))\,\dif x
\]
satisfies
  \[
    I(t) + \frac{c_1}{4}\int_{0}^{t}\Big[\sum_{y\in\eta(s,\cdot)^{-1}(0)}|\eta'(y)|\Big]\,\dif s \le I(0)
  \]
  as claimed.
\end{proof}

We can use a similar argument to prove \cref{prop:conservation}.
\begin{proof}[Proof of \cref{prop:conservation}]
  Define $\zeta_\delta$ as in the proof of \cref{lem:Fdecreasing} and let
  \[
    I_\delta(t) = \int_\R \eta(t,x)\zeta_\delta(x)\,\dif x.
  \]
  Again, \cref{lem:weightedhkcts} shows that $\eta$ is strongly continuous in $L^1_{\p_2}(\R)$, so $I_\delta$ is continuous.
  As in the proof of \cref{lem:Fdecreasing}, we have
  \[
    \frac{\dif}{\dif t}I_\delta = \frac{1}{2}\int_\R [\eta''-(\eta\xi)']\zeta_\delta \,\dif x= -\frac{1}{2}\int_\R [\eta \zeta_\delta''-\eta\xi\zeta_\delta']\,\dif x,
  \]
  so that
  \begin{align*}
    \left|\frac{\dif}{\dif t}I_\delta\right| &\le \frac{1}{2}\int_\R |\eta| [| \zeta_\delta''|+|\xi||\zeta_\delta'|] \,\dif x 
    \le C(\delta^2+\delta^{1-\ell}\|\xi(t,\cdot)\|_{\mathcal{C}_{\p_\ell}(\R)})\|\eta(t,\cdot)\|_{L^1(\R)} \\&\le C(\delta^2+\delta^{1-\ell}\|\xi(t,\cdot)\|_{\mathcal{C}_{\p_\ell}(\R)})\|\eta(0,\cdot)\|_{L^1(\R)},
  \end{align*}
  where in the last inequality we used \cref{prop:L1Rcontraction} and \cref{eq:L1startsfinite}.
We now pass to the limit $\delta\to 0$ and use the dominated convergence theorem, \cref{prop:L1Rcontraction} and \cref{eq:L1startsfinite}, to obtain \cref{eq:meanpreserved}.
\end{proof}

\subsection{Conservation and  \texorpdfstring{$L^{1}$}{L¹}-contraction in the probability space\label{subsec:L1contraction}}

We now prove results similar to \cref{prop:conservation,prop:L1Rcontraction} for solutions stationary with respect to a group $G$ of spatial translations, either with  $G=\R$ or by $G=L\Z$, with some $L>0$. 
We define the fundamental domain $\Lambda_G = \{0\}$ if $G=\R$ and $\Lambda_G = [0,L)$ if $G=L\mathbb{Z}$,
and let $\lambda_G$ be the unique translation-invariant measure on $\R$ such that $\lambda_G(\Lambda_G) = 1$.
In other words, $\lambda_{\mathbb{R}}$ is the counting measure and $\lambda_{L\mathbb{Z}}$ is $1/L$ times the Lebesgue measure.
We recall that $\mathscr{P}_G(X)$ is the space of probability measures on~$X$ that are invariant under the action of $G$, as in \cref{jun2502}.
\begin{prop}[Conservation of mass in the probability space]\label{prop:L1omegaconservation}
Let~$\mathbf{u}=(u_1,u_2)\in\mathcal{Z}_m^2$ be a solution to~\cref{eq:uPDE-many} such that $\Law(\mathbf{u}(0,\cdot))\in\mathscr{P}_G(L^\infty_{\p_{m+}}(\R)^2)$,
with some $m\in(0,1)$, and set $\eta=u_1-u_2$. 
Assume also that
 \begin{equation}\label{eq:finitemoment}
    \sup_{\substack{i\in\{1,2\}\\
        t\in[0,T]
      }
    }\mathbb{E} \left[\int_{\Lambda_G} u_{i}(t,x)^{2} \, \dif \lambda_G(x)\right] < \infty.
  \end{equation}
Then, for all $t \geq 0$ we have
\begin{equation}\label{jun2510}
    \mathbb{E}\left[\int_{\Lambda_{G}}\eta(t, x)\,\dif\lambda_{G}(x)\right]=\mathbb{E}\left[\int_{\Lambda_{G}}\eta(0,x)\,\dif\lambda_{G}(x)\right].
 \end{equation}
\end{prop}
Strictly speaking, if $G = \R$, the integral in the right side of \cref{jun2510} is
\[
\int_{\Lambda_G}\eta(0, x)\,\dif\lambda_{G}(x) = \eta(0, 0),
\]
and is ill-defined because $\mathbf{u}(0, \cdot) \in L^\infty_{\p_{m+}}(\R)^2$ is not defined pointwise.
This obstruction is merely formal.
Here and in the sequel, we use the convention 
\[
\E f(0)=\E \int_0^1 f(x) \, \dif x,
\]
whenever $\Law(f) \in \mathscr{P}_\R(L^\infty_{\p_{m+}}(\R)^2)$.

\begin{proof}
First, assume that $G=L\mathbb{Z}$ for some $L \in (0, \infty)$, so $\Lambda_{G}=[0,L)$. 
Note that
\begin{equation}\label{jun2504}
    \frac{\dif}{\dif t}\int_{0}^{L}\eta(t,x)\,\dif x  =\frac{1}{2}\int_{0}^{L}\left[\partial_{x}^{2}\eta(t,x)-\partial_{x}(\eta\xi)(t,x)\right]\,\dif x
    =\frac{1}{2}[\partial_{x}\eta(t,x)-(\eta\xi)(t,x)]\bigg|_{x=0}^{x=L}
\end{equation}
  for any $t > 0$.
  Integrating in time,
  we have
  \begin{equation}
 \int_{0}^{L}\eta(T,x)\,\dif x -   \int_{0}^{L}\eta(0,x)\,\dif x \, %
 =\frac{1}{2}\int_{0}^{T}[\partial_{x}\eta(t,x)-(\eta\xi)(t,x)]\,\dif t \, \bigg|_{x=0}^{x=L}.\label{eq:integratedintime}
  \end{equation}
  The expected absolute value of the left side is finite by assumption 
  \cref{eq:finitemoment}, and thus so is the absolute expectation of the right side.
 Since the right side is the difference of two identically-distributed random variables, we can use \cref{lem:symmetrylemma} to conclude that
  \[
  \E\int_{0}^{L}\eta(T,x)\,\dif x =\E  \int_{0}^{L}\eta(0,x)\,\dif x, %
  \]
  as claimed.
  The statement for $G=\R$ follows immediately.
\end{proof}

\begin{prop}[$L^{1}$-contraction in the probability space]\label{prop:L1omega} 
Suppose that $\mathbf{u}=(u_{1},u_{2})\in\mathcal{Z}_{m}^{2}$, with some $m\in(0,1)$, is
a solution to \cref{eq:uPDE-many} such that 
$\Law(\mathbf{u}(0,\cdot))\in\mathscr{P}_G(L_{\p_{m+}}^{\infty}(\mathbb{R})^{2})$ 
for $G = L\mathbb{Z}$, with some $L>0$, and that 
  \begin{equation}
    A\coloneqq\sup_{\substack{i\in\{1,2\}\\
        t\in[0,T]
      }
    }\mathbb{E} \left[\int_{\Lambda_G} u_{i}(t,x)^{2} \, \dif \lambda_G(x)\right] < \infty.\label{eq:L1omegamoments}
  \end{equation}
Let $F:\mathbb{R}\to\mathbb{R}_{\ge 0}$
  satisfy $F''=c\delta_{0}$ for some $c>0$, and set $\eta=u_{1}-u_{2}$.
  Then for all $t\ge 0$ we have
  \begin{equation}
    \mathbb{E}\int_{0}^LF(\eta(t, x))\,\dif x + 
    \frac{c}{4}\E\int_0^t\sum_{\substack{y\in[0,L]\\\eta(s,y)=0}}
    |\partial_x\eta(s,y)|\,\dif s \le \mathbb{E}\int_0^LF(\eta(0,x))\,\dif x.\label{eq:L1omega}
  \end{equation}
  Furthermore, if  $\Law(\mathbf{u}(0,\cdot))\in\mathscr{P}_{\R}(L_{\p_{m+}}^{\infty}(\mathbb{R})^{2})$
then for all $T\ge 0$ and $x\in\R$ we have 
  \begin{equation}\label{jun2516}
    \mathbb{E}F(\eta(T,x))\le\mathbb{E}F(\eta(0,x)).
  \end{equation}
 \end{prop}

\begin{proof}
  Define the approximants $F_\eps$ as in \cref{lem:approxconvex}. Similarly to
  \cref{jun2504}, we have
  \begin{align*}
 &   \frac{\dif}{\dif t}\int_{0}^{L}F_{\eps}(\eta(t,x))\,\dif x  =\int_{0}^{L}F_{\eps}'(\eta(t,x))\left[\frac{1}{2}\partial_{x}^{2}\eta(t,x)-\frac{1}{2}\partial_{x}(\eta\xi)(t,x)\right]\,\dif x      \\
                                                               & =\frac{1}{2}F_{\eps}'(\eta(t,x))[\partial_{x}\eta(t,x)-(\eta\xi)(t,x)]\bigg|_{x=0}^{x=L}                                              %
                                                               -\frac{1}{2}\int_{0}^{L}F_{\eps}''(\eta(t,x))\left[(\partial_{x}\eta(t,x))^{2}-((\partial_{x}\eta)\eta\xi)(t,x)\right]\,\dif x,
\end{align*}
for all $t > 0$.
Integrating in time,
  we get
  \begin{align}
\int_{0}^{L}F_{\eps}(\eta(s,x))\,\dif x \, \bigg|_{s=0}^{s=t} 
& =\int_{0}^{t}\frac{1}{2}F_{\eps}'(\eta(s,x))[\partial_{x}\eta(s,x)
-(\eta\xi)(s,x)]\,\dif s\, \bigg|_{x=0}^{x=L}\nonumber                                                                 \\
 &-\frac{1}{2}\int_{0}^{t}\int_{0}^{L}F_{\eps}''(\eta(s,x))\left[(\partial_{x}\eta(s,x))^{2}-((\partial_{x}\eta)\eta\xi)(s,x)\right]\,\dif x\,\dif s.\label{eq:integratedbyparts}
  \end{align}
The expectation of the left side is finite by assumption \cref{eq:L1omegamoments}.
We also have
\begin{equation}\label{jun2512}
  \begin{aligned}
    & \int_{0}^{t}\int_{0}^{L}F_{\eps}''(\eta(s,x))\left[-(\partial_{x}\eta(s,x))^{2}+((\partial_{x}\eta)\eta\xi)(s,x)\right]\,\dif x\,\dif s                                             \\
 & \qquad\le\int_{0}^{t}\int_{0}^{L}F_{\eps}''(\eta(s,x))
    \left[-\frac{1}{2}(\partial_{x}\eta(s,x))^{2}+\frac{1}{2}((\eta\xi)(s,x))^{2}\right]\,\dif x\,\dif s                         \\
    & \qquad\le-\frac{1}{2}\int_{0}^{t}\int_{0}^{L}F_{\eps}''(\eta(s,x))(\partial_{x}\eta(s,x))^{2}\,\dif x\,\dif t+
    \frac{\eps}{2}\int_{0}^{t}\int_{0}^{L}\xi(s,x)^{2}\,\dif x\,\dif s,
  \end{aligned}
  \end{equation}
  where the first inequality is by Young's inequality and the second
  is by \cref{eq:Fprimeprimebd}. It follows that
  \begin{align*}
    \mathbb{E} & \left(\frac{1}{2}\int_{0}^{t}\int_{0}^{L}F_{\eps}''(\eta(s,x))\left[(\partial_{x}\eta(s,x))^{2}-((\partial_{x}\eta)\eta\xi)(s,x)\right]\,\dif x\,\dif s\right)^{-} \\
               & \qquad\qquad\le\mathbb{E}\left(\frac{\eps}{2}\int_{0}^{t}\int_{0}^{L}\xi(s,x)^{2}\,\dif x\,\dif s\right)^{+}<\infty.
  \end{align*}
In addition, the absolute expectation of the left side of
\cref{eq:integratedbyparts} is finite
by assumption \cref{eq:L1omegamoments}.
  Therefore, we can take expectations in \cref{eq:integratedbyparts}
  and apply \cref{lem:symmetrylemma}, using the fact that the first term
  in the right side of \cref{eq:integratedbyparts} is the difference
  of two identically-distributed random variables, to obtain, using \cref{jun2512}
  \begin{align}
    \mathbb{E}\int_{0}^{L}F_{\eps}(\eta(s,x))\,\dif x\, \bigg|_{s=0}^{s=t} & =\frac{1}{2}\mathbb{E}\int_{0}^{t}\int_{0}^{L}F_{\eps}''(\eta(s,x))\left[-(\partial_{x}\eta(s,x))^{2}+((\partial_{x}\eta)\eta\xi)(s,x)\,\dif x\,\dif s\right]\nonumber \\
                                                                           & \le-\frac{1}{4}\E\int_{0}^{t}\int_{0}^{L}F_{\eps}''(\eta(s,x))(\partial_{x}\eta(s,x))^{2}\,\dif x\,\dif s+\eps TLA.\label{eq:L1omegapenultimate}
  \end{align}
We would like to pass to the limit $\eps\downarrow 0$ in \cref{eq:L1omegapenultimate}. 
For the left side, using the bounded convergence theorem on $F_{\eps}|_{[-1,1]}$, and the assumption that $F_{\eps}|_{\mathbb{R}\setminus[-1,1]}$
  is independent of $\eps$, we have
  \begin{equation}
    \lim_{\eps\downarrow 0} \mathbb{E}\int_{0}^{L}F_{\eps}(\eta(s,x))\,\dif x \, \bigg|_{s=0}^{s=t} = \mathbb{E}\int_{0}^{L}F(\eta(s,x))\,\dif x \, \bigg|_{s=0}^{s=t}.\label{eq:LHSlimit}
  \end{equation}
Next, consider the first term in the right side of \cref{eq:L1omegapenultimate}. 
Take a nonnegative function $\zeta \in C_c^\infty(\R)$ with~$\|\zeta\|_{L^1(\R)}=1$, and let 
\[
\tilde{\zeta}(x) = \zeta * \mathbf{1}_{[0,L]}(x)=\int \zeta(x-y)\mathbf{1}_{[0,L]}(y)\,\dif y
=\int_{x-L}^{x}\zeta(y),\dif y.
\]
The $L\Z$-invariance in law of $\eta$ implies that 
\begin{equation}\label{jun2514}
\begin{aligned}
&{\E\int_{0}^{t}\int_{0}^{L}F_{\eps}''(\eta(s,x))(\partial_{x}\eta(s,x))^{2}\,\dif x\,\dif s =\E\int_{0}^{t} \int_\R
\int_{y}^{L+y}F_{\eps}''(\eta(s,x))(\partial_{x}\eta(s,x))^{2}\zeta(y)\dif x\, \dif y\, \dif s
}\\ 
&{=\E\int_{0}^{t} \int_\R
\int_{x-L}^{x}F_{\eps}''(\eta(s,x))(\partial_{x}\eta(s,x))^{2}\zeta(y)\dif y\,\dif x\, \dif s}
{=  \E\int_{0}^{t}\int_\R F_{\eps}''(\eta(s,x))(\partial_{x}\eta(s,x))^{2}\tilde{\zeta}(x)\,\dif x\,\dif s}.
\end{aligned}
\end{equation}
By the co-area formula, we have
  \[
    \int_{0}^{t}\int_\R F_{\eps}''(\eta(s,x))(\partial_{x}\eta(s,x))^{2}\tilde{\zeta}(x)\,\dif x\,\dif s=\int_0^T\int_\R F_\eps''(\lambda)\sum_{y\in\eta(s,\cdot)^{-1}(\lambda)}|\partial_x\eta (s,y)|\tilde{\zeta}(y)\,\dif \lambda\, \dif s,
  \]
  almost surely.
  Since $F_\eps''$ converges weakly to $c\delta_0$ as $\eps\downarrow 0$, and the map 
  \[
  \lambda\mapsto \sum\limits_{y\in\eta(s,\cdot)^{-1}(\lambda)}|\partial_x\eta (s,y)|\tilde{\zeta}(y)
  \]
   is almost surely continuous by \cref{lem:gcts}, we have
  \[
    \lim_{\eps\downarrow0} \int_{0}^{t}\int_\R F_{\eps}''(\eta(s,x))(\partial_{x}\eta(s,x))^{2}\tilde{\zeta}(x)\,\dif x\,\dif s=c\int_0^T \sum_{y\in\eta(s,\cdot)^{-1}(0)}|\partial_x\eta (s,y)|\tilde{\zeta}(y)\,\dif s,
  \]
  almost surely. Again using the $L\Z$-invariance of $\eta$ we obtain, 
  as in \cref{jun2514},
  \begin{equation}
    \begin{aligned}
      c\E\int_0^T\sum_{\substack{y\in[0,L] \\\eta(s,y)=0}}|\partial_x\eta(s,y)|\,\dif t&=c\E\int_0^T \sum_{y\in\eta(s,\cdot)^{-1}(0)}|\partial_x\eta (s,y)|\tilde{\zeta}(y)\,\dif t\\&\le \liminf_{\eps\downarrow0} \E\int_{0}^{t}\int_\R F_{\eps}''(\eta(s,x))(\partial_{x}\eta(s,x))^{2}\tilde{\zeta}(x)\,\dif x\,\dif t
    \end{aligned}\label{eq:takeexpectations}
  \end{equation}
  by Fatou's lemma.
  Now we can take $\eps\downarrow 0$ in \cref{eq:L1omegapenultimate} and use \cref{eq:LHSlimit} and \cref{eq:takeexpectations} to obtain
  \[
    \mathbb{E}\int_{0}^{L}F(\eta(s,x)) \,\dif x \, \bigg|_{s=0}^{s=t}\le -\frac{c}{4}\E\int_0^T\sum_{\substack{y\in[0,L]\\\eta(s,y)=0}}|\partial_x\eta(s,y)|\,\dif t,
  \]
  which is \cref{eq:L1omega}. Finally, \cref{jun2516} is a consequence of  \cref{eq:L1omega} and translation invariance. 
\end{proof}

\subsection{Ordering of stationary solutions\label{subsec:ordering}}

The key tool in the classification of space-time stationary solutions 
is an almost-sure ordering theorem for the components of spacetime
stationary solutions to \cref{eq:uPDE-many}. This theorem is very similar to the comparison and $L^1$-contraction results proved in the previous two sections.
The intuition is that if two solutions cross transversely, then the heat flow causes cancellation between the positive and negative parts of the difference at the crossing point, and so the $L^1$-norm 
of the difference decreases.
But if two solutions are jointly time-stationary, the $L^1$-norm of their difference 
must be conserved in time. Of course, the~$L^1$-norm must be taken with respect to the probability space and a compact interval, since the~$L^1$-norm of the difference of two spatially-stationary solutions on the whole line is not generally finite.

\begin{thm}\label{thm:ordering}
Let $m\in(0,1)$ and $\nu\in\overline{\mathscr{P}}(\mathcal{X}_{m}^{2})$.
  Let $\mathbf{v}=(v_{1},v_{2})\sim\nu$ and suppose that one of the
  following two conditions holds:
  \begin{enumerate}[leftmargin = 1.5cm, label = \textup{(H\arabic*)}, ref = (\textup{H\arabic*})]
  \item \label{enu:hyp-stat}
    $\nu\in\overline{\mathscr{P}}_{G}(\mathcal{X}_{m}^{2})$
    for $G=L\Z$ with some $L>0$ or $G=\mathbb{R}$, and
    \begin{equation}\label{jun2518}
      A = \sup\limits_{i\in\{1,2\}} \mathbb{E} \int_{\Lambda_G} v_{i}(x)^{2} \, \dif \lambda_G(x) < \infty.
    \end{equation}
  \item \label{enu:hyp-L1}$\|v_{1}-v_{2}\|_{L^{1}(\mathbb{R})}<\infty$
    almost surely.
  \end{enumerate}
  Then, almost surely, $\sgn(v_{1}(x)-v_{2}(x))$ is a random constant independent
  of $x$.
\end{thm}

In the sequel, we will only apply \cref{thm:ordering} with hypothesis
\cref{enu:hyp-stat}, but we include \cref{enu:hyp-L1}  
for completeness.  \cref{thm:ordering} is an immediate consequence of the following proposition.

\begin{prop}
  \label{prop:ordering-morespecific}Let $m\in(0,1)$.
  Let $\mathbf{u}=(u_1,u_2)\in\mathcal{Z}_{1/2}^2$ solve \cref{eq:uPDE-many} with random initial conditions independent of the noise, and let $T>0$ be such that one of the
  following two conditions holds:
  \begin{enumerate}[leftmargin = 1.5cm, label = \textup{(H\arabic*$'$)}, ref = (\textup{H\arabic*$'$})]
  \item
    \label{enu:hyp-stat-morespecific}
    $\Law(\mathbf{u}(0,\cdot))\in{\mathscr{P}}_{G}(L_{\p_{m+}}\infty(\R)^{2})$
 for $G=L\Z$ with some $L>0$ or $G=\mathbb{R}$, \cref{jun2518} holds, and
    $ \disp \int_{\Lambda_G}\E|u_1(t,x)-u_2(t,x)|\,\dif \lambda_G(x)$ is constant on $[0, T]$.
\item
    \label{enu:hyp-L1-morespecific}
Almost surely,  $\|u_{1}(t,\cdot)-u_{2}(t,\cdot)\|_{L^{1}(\mathbb{R})}=\|u_{1}(t',\cdot)-u_{2}(t',\cdot)\|_{L^{1}(\mathbb{R})}<\infty$ for all $t,t'\in [0,T]$.
  \end{enumerate}
  Then, almost surely, $\sgn(u_{1}(t,x)-u_{2}(t,x))$ is a random constant independent of $x \in \R$ and $t \in [0, T]$.
\end{prop}

The key
step in the proof of \cref{prop:ordering-morespecific} is the following observation, which
shows that $u_{1}$ and $u_{2}$ can only meet tangentially.
\begin{lem}
  \label{lem:onlytangentialintersections}
  Fix $m\in(0,1)$ and let $\mathbf{u}=(u_1,u_2)\in\mathcal{Z}_m^2$ and $T$ satisfy the assumptions of \cref{prop:ordering-morespecific},
 including either  \ref{enu:hyp-stat-morespecific} or \ref{enu:hyp-L1-morespecific}. 
Then, with probability
  $1$, for all $t\in[0,T]$ and $x\in\mathbb{R}$ such that $u_1(t,x)=u_2(t,x)$,
  we have $\partial_{x}u_1(t,x)=\partial_{x}u_2(t,x)$.
\end{lem}

\begin{proof}
  If \cref{enu:hyp-stat-morespecific} %
  holds, then it is sufficient to assume that $G=L\mathbb{Z}$
  for some $L\in(0,\infty)$, since if $G=\mathbb{R}$ then \cref{enu:hyp-stat-morespecific}
  holds for $G=L\mathbb{Z}$ for every $L$. In this case, by \cref{prop:L1omega},
  with $F(\eta)=|\eta|$, we have, for each $x\in\R$ that
  \[
    \E\int_0^T\sum_{\substack{y\in[x,x+L]\\\eta(t,y)=0}}|\partial_x\eta(t,y)|\,\dif t=0,
  \]
  which means that
  \[
    \int_0^T\sum_{\substack{y\in\R\\\eta(t,y)=0}}|\partial_x\eta(t,y)|\,\dif t=0
  \]
  almost surely.
Hence, with probability $1$, for almost all $t \in [0, T]$ we have $\partial_x\eta(t,x)=0$ whenever~$\eta(t,x)=0$.

Let us now strengthen the conclusion to all $t\geq 0$. %
  By \cref{lem:suppsi} and an iteration of \cref{lem:mildregularitypoly}, we know that $\theta_i = u_i - \psi$ and $\psi$ are spatially smooth.
  Differentiating~\cref{eq:uPDE-1} in $x$, we see that $\partial_{tx} u_i$ exists.
  In particular, $\partial_x\eta$ is continuous in space-time.
  Now, suppose there exists $(t_*,x_*) \in \R_+\times \R$ such that~$\eta(t_*,x_*) = 0$ but $\partial_x\eta(t_*,x_*) \neq 0$, and without loss of generality, 
  assume that $\partial_x\eta(t_*,x_*) > 0$.
  Since $\partial_x\eta$ is continuous, there exists a nonempty open rectangle $(t_1,t_2) \times (a,b)$ 
  containing $(t_*,x_*)$ such that $\partial_x\eta > 0$ on $(t_1,t_2)\times (a,b)$ while
  \begin{equation*}
    \eta(t,a) < 0 \quad \textrm{and} \quad \eta(t,b) > 0,~~\hbox{ for all $t_1<t<t_2$},
  \end{equation*}
and thus $\eta(t,x)$ vanishes in  $(a,b)$ for all $t\in (t_1,t_2)$, at some point $x_*(t)$ such that  $\partial_x\eta(t,x_*(t)) > 0$.
This event has probability 0, and the proof under \cref{enu:hyp-stat-morespecific} is complete.
  The proof assuming \cref{enu:hyp-L1-morespecific} is analogous, using \cref{lem:Fdecreasing} %
instead of \cref{prop:L1omega}, also with $F(\eta)=|\eta|$.
\end{proof}

\begin{proof}[Proof of \cref{prop:ordering-morespecific}.]
  By \cref{lem:onlytangentialintersections}, we may assume that $\partial_{x}\eta=0$
  whenever $\eta=0$. This contradicts a parabolic Hopf lemma \cite[Chapter 2]{L96}
  applied to the equation \cref{eq:etaPDE} for $\eta$ unless $u_1$
  and $u_2$ are ordered. To be precise, suppose that the set $Z\coloneqq\{t>0,~x\in\mathbb{R}\st\eta(t,x)=0\}$
  is not empty and that there exists $t_{0}>0$ and $x_{0}\in\mathbb{R}$
  so that $\eta(t_{0},x_{0})>0$. For $\lambda>0$, define the parabolic
  cone
  \[
    Q_{\lambda}=Q_{\lambda}(t_{0},x_{0})\coloneqq\{t>0,~x\in\mathbb{R}\st(x-x_{0})^{2}<t-t_{0}, \, t_{0}<t<t_{0}+\lambda^{2}\},
  \]
  and set
  \[
    \lambda_{*}\coloneqq\inf\{\lambda>0\st Q_{\lambda}\cap Z\ne\emptyset\}.
  \]
  Suppose first for the sake of contradiction that $\lambda_{*}$ is finite.
  As $\eta$ is continuous, we have $\lambda_{*}>0$ and~$\eta>0$ on $Q_{\lambda_{*}}$,
  and moreover there is a point $(t_{0}+\lambda_{*}^{2},x_{*})\in\overline{Q_{\lambda_{*}}}\cap Z$.
 The parabolic Hopf lemma~\cite[Lemma~2.8]{L96} implies that  $\partial_{x}\eta\ne0$
  at $(t_{0}+\lambda_{*}^{2},x_{*})$, contradicting our hypothesis
  on $\eta$. Strictly speaking, this formulation of the Hopf lemma
  only applies when $(t_{0}+\lambda_{*}^{2},x_{*})$ is a corner, that is~$|x_{*}-x|=\lambda_{*}$.
  However, if $|x_{*}-x_{0}|<\lambda_{*}$, we can  apply the
  Hopf lemma to a smaller parabolic cone contained in $Q_{\lambda_{*}}$,
  having $(t_{0}+\lambda_{*}^{2},x_{*})$ as a corner. It follows that
  $\lambda_{*}=\infty$ and $\eta>0$ on $Q_{\infty}$. In addition, if there exists $(t_{1},x_{1})$ such that $\eta(t_{1},x_{1})<0$
  then by a similar argument we have $\eta<0$ on~$Q_{\infty}(t_{1},x_{1})$.
  As the intersection of $Q_{\infty}(t_{0},x_{0})$ and $Q_{\infty}(t_{1},x_{1})$
  is not empty, this is a contradiction. Thus, $\eta(t,x)\ge0$ for
  all $t>0$ and $x\in\mathbb{R}$. The parabolic strong maximum principle~\cite[Theorem~2.7]{L96} implies then that $\eta>0$ for all $t>0$
  and $x\in\mathbb{R}$.

  Thus, almost surely, we have the following trichotomy: $\eta>0$,
  $\eta\equiv0$, or $\eta<0$.
\end{proof}

\section{Other properties of solutions\label{sec:basicproperties}}

In this section, we establish a few other properties of solutions to \cref{eq:uPDE-many}. In \cref{subsec:shearinvariance}, we formulate the shear-invariance property discussed in the introduction, in the context of solutions in weighted spaces. In \cref{subsec:krylovbogoliubov}, we use the Feller property in \cref{prop:fellerproperty} and the standard Krylov--Bogoliubov argument to show that subsequential limits of time-averaged laws of solutions to \cref{eq:uPDE-many} are stationary in time.
We then use this to show that any two invariant measures for \cref{eq:uPDE-many} can be coupled to create an invariant measure on a product space. Finally, in \cref{subsec:KPZsoln}, we show how solutions to \cref{eq:uPDE-1} can be used to build solutions to the KPZ equation \cref{eq:KPZ} in weighted spaces.

\subsection{Shear-invariance\label{subsec:shearinvariance}}

The shear-invariance of the Burgers
equation mentioned in the introduction can be stated as follows.
\begin{prop}
  \label{prop:shearinvariance}Suppose that $\mathbf{u},\mathbf{u}'\in\mathcal{Z}_{m}$
  are solutions to \cref{eq:uPDE-many}
  such that 
  \[
  \mathbf{u}(0,\cdot)=\mathbf{u}'(0,\cdot)+(a,\ldots,a).
  \]
  Then $\Law(\mathbf{u}')=\Law(\mathbf{u}(t,x+at)-(a,\ldots,a))$ as distributions on $\mathcal{Z}_m$.
\end{prop}

\begin{proof}
  If $\tilde{\mathbf{u}}=(\tilde{u}_{1},\ldots,\tilde{u}_{N})=\mathbf{u}(t,x+at)-(a,\ldots,a)$,
  then $\tilde{\mathbf{u}}$ is a strong solution to the SPDE
  \[
    \dif\tilde{u}_{i}=\frac{1}{2}[\partial_{x}^{2}\tilde{u}_{i}-\partial_{x}(\tilde{u}_{i}^{2})]\dif t+\dif(\partial_{x}\tilde{V}),
  \]
  where $\tilde{V}(t,x)=V(t,x+at)$. Define
  \[
    \tilde{\psi}(t,x)=\int_{0}^{t}\partial_{x}G_{t-s}*\dif\tilde{V}(s,\cdot)(x)=\int_{0}^{t}\int_\R \partial_{x}G_{t-s}*\rho(x+at-y)\,\dif W(s,y),
  \]
  so that if we set $\tilde{\boldsymbol{\psi}}=(\tilde{\psi},\ldots,\tilde{\psi})$,
  then $\tilde{\boldsymbol{\theta}}=(\tilde{\theta}_{1},\ldots,\tilde{\theta}_{N})=\tilde{\mathbf{u}}-\tilde{\boldsymbol{\psi}}$
  satisfies
  \[
    \partial_{t}\tilde{\theta}_{i}=\frac{1}{2}\partial_{x}^{2}\tilde{\theta}_{i}-\frac{1}{2}\partial_{x}(\tilde{\theta}_{i}+\tilde{\psi}).
  \]
Analogously, define $\boldsymbol{\psi}$ and $\boldsymbol{\theta}' = \mathbf{u}' - \boldsymbol{\psi}$, with $\psi$ defined in \cref{eq:psidef}. It satisfies
  \begin{equation*}
    \partial_{t}\theta_{i}' =\frac{1}{2}\partial_{x}^{2}\theta_{i}' - \frac{1}{2}\partial_{x}(\theta_{i}'+ \psi).
  \end{equation*}
  By construction, $\tilde{\boldsymbol{\theta}}(0, \cdot) = \boldsymbol{\theta}'(0, \cdot)$.
  Also, by computing the covariance structure, it is easy to see that $\tilde{\psi}$ has the same law as $\psi$.
  Therefore $\tilde{\boldsymbol{\theta}}$ and $\boldsymbol{\theta}'$ agree in law.
  It follows that $\tilde{\mathbf{u}}$ and $\mathbf{u}'$ also agree in law.
\end{proof}

\subsection{The Krylov--Bogoliubov theorem\label{subsec:krylovbogoliubov}}

The Feller property discussed in \cref{thm:existence}
allows us to apply the Krylov--Bogoliubov theorem. This
is a standard argument in the ergodic theory of SPDEs (see e.g. \cite[Theorem 3.1.1]{DPZ96})
but since it is simple and central to our argument, we present a proof
of the theorem and one consequence here.
\begin{prop}[Krylov--Bogoliubov]
  \label{prop:krylovbogoliubov}Fix $m\in(0,1)$, $N\in\mathbb{R}$,
  and $G=L\Zm$ with $L>0$ or $G=\Rm$.  %
  If $s\ge0$, $T_{k}\uparrow\infty$,
  and $\nu_{0}\in\mathscr{P}_{G}(L_{\p_{m+}}^{\infty})$ are such that
  \[
    \nu\coloneqq\lim\limits _{k\to\infty}\frac{1}{T_{k}}\int_{s}^{s+T_{k}}P_{t}^{*}\nu_{0}\,\dif t
  \]
  exists in the sense of weak convergence of measures on $\mathcal{X}_{m}^{N}$,
  then $\nu\in\overline{\mathscr{P}}_{G}(\mathcal{X}_{m}^{N})$.
\end{prop}

\begin{proof}
  If $f\in\mathcal{C}_{\mathrm{b}}(\mathcal{X}_{m}^{N})$ and $r\ge0$,
  then by the Feller property (\cref{prop:fellerproperty}) we have $P_{r}f\in\mathcal{C}_{\mathrm{b}}(\mathcal{X}_{m}^{N})$
  as well, so
  \begin{align*}
    \langle f,P_{r}^{*}\nu\rangle & =\langle\nu,P_{r}f\rangle=\Big\langle P_{r}f,\lim\limits _{k\to\infty}\frac{1}{T_{k}}\int_{s}^{s+T_{k}}P_{t}^{*}\nu_{0}\,\dif t\Big\rangle =\lim\limits _{k\to\infty}\Big\langle P_{r}f,\frac{1}{T_{k}}\int_{s}^{s+T_{k}}P_{t}^{*}\nu_{0}\,\dif t\Big\rangle \\
                                  & =\lim\limits _{k\to\infty}\Big\langle f,\frac{1}{T_{k}}\int_{s+r}^{s+r+T_{k}}P_{t}^{*}\nu_{0}\,\dif t\Big\rangle .
  \end{align*}
  On the other hand, we have
  \[
    \lim_{k\to\infty}\Big\langle f,\frac{1}{T_{k}}\int_{s+r}^{s+r+T_{k}}P_{t}^{*}\nu_{0}\,\dif t-\frac{1}{T_{k}}\int_{s}^{s+T_{k}}P_{t}^{*}\nu_{0}\,\dif t\Big\rangle =0
  \]
  for all $f\in\mathcal{C}_{\mathrm{b}}(\mathcal{X}_{m}^{N})$, so
  \[
    \lim_{k\to\infty}\frac{1}{T_{k}}\int_{s+r}^{s+r+T_{k}}P_{t}^{*}\nu_{0}\,\dif t=\nu,
  \]
so $\langle f,P_{r}^{*}\nu\rangle=\langle f,\nu\rangle$
  for all $f\in\mathcal{C}_{\mathrm{b}}(\mathcal{X}_{m}^{N})$, so $P_{r}^{*}\nu=\nu$.
  This holds for all $r\ge0$, so $\nu\in\overline{\mathscr{P}}_{G}(\mathcal{X}_{m}^{N})$.
\end{proof}
An important application for the classification
results, is that any two invariant measures for~\cref{eq:uPDE-many}
can be coupled to form an invariant measure on the product space.
\begin{prop}
  \label{prop:couple}Fix $m\in(0,1)$ and $N_{1},N_{2}\in\mathbb{N}$. 
  If~$\nu_{i}\in\overline{\mathscr{P}}_{G}(\mathcal{X}_{m}^{N_{i}})$,
  $i=1,2$, with either  $G=L\Zm$ with some $L>0$, or $G=\mathbb{R}$,
  then there is a coupling $\nu\in\overline{\mathscr{P}}_{G}(\mathcal{X}_{m}^{N_{1}+N_{2}})$
  of $\nu_{1}$ and $\nu_{2}$.
\end{prop}

\begin{proof}
First, let $\nu_{0}$ be a coupling of $\nu_{1}$ and $\nu_{2}$ such that
the two factors are independent. 
Then $\nu_{0}$ is $G$-invariant if $\nu_{1}$ and $\nu_{2}$ are. However, $\nu_0$
is not expected to be invariant under the time evolution, since letting both factors evolve according to the same noise introduces dependence. To find a coupling that is invariant under the evolution, 
we observe that the family 
\[
\tilde\nu_T:=\frac{1}{T}\int_{0}^{T}P_{t}^{*}\nu_{0}\,\dif t 
\]
of measures on $\mathcal{X}_{m}^{N_{1}+N_{2}}$ is tight.
Indeed, as $\cX_m$ is a Polish space, %
by Prokhorov's theorem both~$\nu_1$ and $\nu_2$ are themselves tight.
  Fixing $\eps > 0$, there exist compact sets~$K_i \subset \mathcal{X}_{m}^{N_i}$ such 
  that $\nu_i$ has at most~${\eps}/{2}$ mass outside $K_i$, for each $i =1, 2$.
  By the invariance of~$\nu_i$ and a union bound, we have
  \begin{equation*}
 \tilde\nu_T(K_1 \times K_2)=   \frac{1}{T}\int_{0}^{T}(P_{t}^{*}\nu_{0})(K_1 \times K_2 )
 \,\dif t \geq 1 - \nu_1(K_1^c) - \nu_2(K_2^c) > 1 - \eps.
  \end{equation*}
  It follows that the family $\tilde\nu_T$ 
  is tight.
Again by Prokhorov's theorem, there is a sequence $T_{k}\uparrow\infty$ and 
a measure $\nu$ on $\mathcal{X}_{m}^{N_{1}+N_{2}}$ so that $\tilde\nu_{T_k}\to \nu$ 
  weakly as $k\to\infty$.
Note that $\nu$ is $G$-invariant if $\nu_{0}$ is, and by \cref{prop:krylovbogoliubov}, $\nu$ is also invariant under the dynamics \cref{eq:uPDE-many}.
  Finally, marginals are preserved under weak limits, so $\nu$ couples $\nu_1$ and $\nu_2$.
\end{proof}

\subsection{The solution to the KPZ equation\label{subsec:KPZsoln}}

We will sometimes need to relate the solutions of
the Burgers equation to the solution of the KPZ equation. In this
section, we show how to obtain solutions to \cref{eq:KPZ} from those
to \cref{eq:uPDE-1}.
\begin{prop}
  \label{prop:KPZsolntheory}
  Fix $m\in(0,1)$ and a smooth, non-negative, compactly-supported function $\zeta \in \mathcal{C}^\infty(\R)$ with~$\|\zeta\|_{L^1(\Rm)}=1$. %
  Suppose that $u\in\mathcal{Z}_{m}$
  is a strong solution to \cref{eq:uPDE-1}, and let $\theta=u-\psi$,
  with $\psi$ defined as in \cref{eq:psidef}.
Then,
  \begin{equation}\label{jul140}
  \begin{aligned}
    h(t,x) & = \int_\R \zeta(y) \int_{y}^{x}\theta(t,z)\,\dif z \dif y - \frac{1}{2}\int_{0}^{t} \int_\R [\theta(s,y) \zeta'(y) + (\theta(s,y)+\psi(s,y))^{2} \zeta(y)]\, \dif y \dif s\\
           & \qquad+ \int_{0}^{t}G_{t-s}*\dif V(s,\cdot)(x) +\frac{t}{2}\|\rho\|_{L^{2}(\mathbb{R})}^{2}
  \end{aligned}
  \end{equation}
  is a strong solution to \cref{eq:KPZ}.
\end{prop}
Note that, in particular, we have
\begin{equation}\label{jul124}
\int _\R\zeta(x)h(0,x)\,\dif x=0,
\end{equation}
which fixes the ``constant of integration'' in taking the antiderivative of $u(0,\cdot)$.

\begin{proof}
We write $h(t,x)=\omega(t,x)+g(t,x)$, with
  \[
    \omega(t,x)=\int_{0}^{t}G_{t-s}*\dif V(s,\cdot)(x)
  \]
  and
  \[
    g(t,x)=\int_\R \zeta(y) \int_{y}^{x}\theta(t,z)\,\dif z \dif y - \frac{1}{2}\int_{0}^{t} \int_\R [\theta(s,y) \zeta'(y) + (\theta(s,y)+\psi(s,y))^{2} \zeta(y)]\, \dif y \dif s+\frac{t}{2}\|\rho\|_{L^{2}(\mathbb{R})}^{2}.
  \]
Note that $\partial_{x}\omega = \psi$, and $\omega$ is a strong
  solution to the SPDE
  \begin{equation}
    \dif\omega=\frac{1}{2}\partial_{x}^{2}\omega \, \dif t+\dif V.\label{eq:difomega}
  \end{equation}
In addition, the function $g$ satisfies $\partial_x g = \theta$ and
using \cref{eq:thetajPDE}, we see that  
  \begin{align*}
    \partial_{t}g(t,x) - \frac{1}{2}\|\rho\|_{L^{2}(\mathbb{R})}^{2} & = \int_\R \zeta(y) \int_{y}^{x}\partial_{t}\theta(t,z)\,\dif z \dif y - 
    \frac{1}{2} \int_\R [\theta({t},y) \zeta'(y) + 
    (\theta({t},y)+\psi({t},y))^{2} \zeta(y)]\, \dif y \\
 & = \frac{1}{2} \int_\R \zeta(y) \int_{y}^{x}\left[
 \partial_{x}^{2}\theta(t,z)-\partial_{x}(\theta(t,z)+\psi(t,z))^{2}\right] \,\dif z \dif y                                      \\
 & \qquad +\frac{1}{2}\int_\R [\partial_x \theta({t},y) - 
 (\theta({t},y)+\psi({t},y))^{2}] \zeta(y)\, \dif y                           \\
                                                                     & =\frac{1}{2}\partial_{x}\theta(t,x)-\frac{1}{2}(\theta(t,x)+\psi(t,x))^{2}
  =\frac{1}{2}\partial_{x}^{2}g(t,x)-\frac{1}{2}(\partial_{x}g(t,x)+\partial_{x}\omega(t,x))^{2}.
  \end{align*}
  From this and \cref{eq:difomega}, we see that $h=g+\omega$ is a strong
  solution to \cref{eq:uPDE-1}.
\end{proof}

\section{Limits of solutions started from bounded initial conditions\label{sec:tightness}}

In this section, we consider solutions to \cref{eq:uPDE-1} started at constant initial conditions. In \cref{subsec:tightness}, we show that the laws of solutions started at deterministic, globally-bounded initial conditions are tight as time goes to infinity. In \cref{subsec:extremality}, we show that, if the initial condition is in addition periodic in space, then the resulting subsequential limits are extremal.

\subsection{Uniform-in-time bounds and tightness\label{subsec:tightness}}

In this section, we start with a bounded initial condition and establish tightness of the resulting
family of solutions $u(t,\cdot)$, as $t\to+\infty$.
This requires a priori bounds in weighted spaces, 
that we are only able to prove in $\mathcal{X}_m$ for $m \geq {1}/{2}$.
In combination with the restriction $m<1$ from \cref{sec:solntheory}, this explains the range of exponents $m$ in \cref{thm:maintheorem-classification}.
The main result of this section is the following.
\begin{prop}
  \label{prop:boundedtightness}
  Let $\mathbf{u}\in\mathcal{Z}_{1/2}^{N}$
  solve \cref{eq:uPDE-1} with bounded initial condition $\mathbf{u}(0,\cdot)\in L^{\infty}(\mathbb{R})^{N}$.
  Then the family of random variables $\{\mathbf{u}(t,\cdot)\}_{t\ge1}$
  is tight with respect to the topology of $\mathcal{X}_{1/2}^{N}$,
  and
  \begin{equation}
    \sup_{t\ge0,x\in\R
    }\mathbb{E}|\mathbf{u}(t,x)|^{2}<\infty.\label{eq:supu2bdd}
  \end{equation}
\end{prop}

Tightness arguments on compact domains \cite{Bor16multi,Bor13sharp,DPDT94,DPG94}
have generally relied on bounding the size of the solution in terms
of the size of its derivative, using a Poincaré-type inequality. %
Such
an argument cannot work on the whole space. Instead, we
control the second moment of the Burgers solution $u$ by first controlling the
expectation of the KPZ solution $h$.
The following proposition will
be key to the proof of \cref{prop:boundedtightness}.
\begin{prop}
  \label{prop:gamma} 
  Let $h$ be a mild solution to \cref{eq:KPZ}
  with initial condition $h(0,\cdot)\equiv0$. Then,
  the function~$\gamma(t)=\mathbb{E}h(t,0)$
  is increasing for $t>0$. 
\end{prop}

\begin{proof}
  Using the Feynman--Kac formula as in \cite{BC95,BCJL94}, we have
  \begin{equation}
    h(t,0) = -\log\mathrm{E}_{X^{t,0}}\exp\left\{ -\int_{0}^{t}\dif V(s,X(s))-\frac{t}{2}\|\rho\|_{L^{2}(\mathbb{R})}^{2}\right\}.\label{eq:FKformula}
  \end{equation}
Here, $\mathrm{E}_{X^{t,x}}$ denotes expectation with respect to
  the measure in which $X$ is a two-sided Brownian motion satisfying
  $X(t)=x$, that is independent of the noise. Note that the right side of \cref{eq:FKformula} is still a random quantity due to the randomness of the noise.
As in \cite{MSZ16}, the time-reversed process
  \[
    Z_{t}=\mathrm{E}_{X^{0,0}}\exp\left\{ -\int_{0}^{t}\dif V(s,X(s))-\frac{t}{2}\|\rho\|_{L^{2}(\mathbb{R})}^{2}\right\} ,
  \]
satisfies
\[
h(t,0)\overset{\mathrm{law}}{=}-\log Z_{t}.
\]
The point of the time-reversal is that the process
 $Z_{t}$ is a martingale and satisfies  
  \[
    \dif Z_{t}=\mathrm{E}_{X^{0,0}}\exp\left\{ -\int_{0}^{t}\dif V(s,X(s))-\frac{t}{2}\|\rho\|_{L^{2}(\mathbb{R})}^{2}\right\} \dif V(t,X(t)).
  \]
  The convexity of $(-\log)$ implies that $-\log Z_{t}$ is a submartingale,
so
\[
\mathbb{E}(-\log Z_{t})=\mathbb{E}h(t,0)=\gamma(t)
\]
is increasing in $t$.
\end{proof}
An important application of \cref{prop:gamma} is the following $L^2$-bound, which is key to the compactness result.
\begin{lem}
  \label{lem:L2bound}
  Suppose that $a\in\mathbb{R}$ and $u$ is a solution
  to \cref{eq:uPDE-1} with initial condition $u(0,\cdot)\equiv a$. Then
  for all $t\ge0$ and $x\in\mathbb{R}$, we have
  \begin{equation}
    \mathbb{E}(u(t,x)-a)^{2}\le\|\rho\|_{L^{2}(\mathbb{R})}^{2}.\label{eq:L2bound}
  \end{equation}
\end{lem}

\begin{proof}
  By \cref{prop:shearinvariance}, it suffices to consider the case $a=0$.
  Let $h$ be a solution to \cref{eq:KPZ} with initial condition $h(0,\cdot)\equiv0$,
  so we can take $u=\partial_{x}h$.   %
Averaging \cref{eq:KPZ} in space, integrating it time over an interval
$(t_1,t_2)$, and taking expectations, we get
  \begin{equation}
  \begin{aligned}
    \frac{1}{2}(t_2-t_1)\|\rho\|_{L^{2}(\mathbb{R})}^{2}&+\frac{1}{2}\E\int_{t_1}^{t_2}[\partial_x u(t,{x+1})-\partial_xu(t,{x})]\,\dif t-\frac{1}{2}\mathbb{E}\int_{t_1}^{t_2}u(t,x)^{2}\,\dif t\\&=\E h(t_2,x)-\E h(t_1,x),
    \end{aligned}
    \label{eq:EKPZ}
  \end{equation} 
  for an arbitrary $x\in\R$ by space-stationarity of $u$ and $h$. 
  Also by space-stationarity of $u$, the laws of 
  \[
  \int_{t_1}^{t_2}\partial_x u(t,x)\,\dif t\text{ and } \int_{t_1}^{t_2}\partial_x u(t,{x+1})\,\dif t
  \]
are the same. %
In addition, 
the right side of \cref{eq:EKPZ} is nonnegative by \cref{prop:gamma}. Hence,
by \cref{lem:symmetrylemma}, the expectation of the second term in the
left side of \cref{eq:EKPZ} is well-defined and equals to $0$. This 
gives \cref{eq:L2bound} for almost all $t\ge 0$ and $x\in\R$. The statement for all $t\ge 0 $ 
and~$x\in\R$ follows from the almost-sure continuity of $u$ and Fatou's lemma.
\end{proof}

Next, we  get a bound on the derivative of the solution.
The following lemma is based on an energy estimate similar to that leading
to \cite[formula (16)]{Bor13sharp}.
\begin{lem}
  \label{lem:gradientbound}
  Suppose that $a\in\mathbb{R}$ and $u$
  is a solution to \cref{eq:uPDE-1} with initial condition $u(0,\cdot)\equiv a$.
  Then for all $t> q\ge0$ and $x\in\mathbb{R}$, we have
  \begin{equation}\label{eq:derivL2bound}
    \frac{1}{t-q}\int_q^{t}\mathbb{E}(\partial_{x}u(s,x))^{2}\,\dif s\le\frac{2}{t-q}\|\rho\|^2_{L^2(\mathbb{R}^2)}+\|\partial_{x}\rho\|_{L^{2}(\mathbb{R})}^{2}.
  \end{equation}
\end{lem}

\begin{proof}
  Again using \cref{prop:shearinvariance}, we can assume that $a=0$.
  By the Itô formula (see e.g.~\cite[Theorem 3.32]{DPZ14}) applied
  to \cref{eq:uPDE-1}, we get that whenever $t > q\ge0$,
  \begin{align*}
    u&(t,  x)^{2}-u(q,x)^{2}                                                                                                                                                                                                              \\
         & =\int_{q}^{t}2u(s,x)\,\dif(\partial_{x}V)(s,x)+\int_{q}^{t}2u(s,x)\left[\partial_{x}^{2}u-\partial_{x}(u^{2})\right](s,x)\,\dif s\\&\qquad+(t-q)\|\partial_{x}\rho\|_{L^{2}(\mathbb{R})}^{2}                                            \\
         & =\int_{q}^{t}2u(s,x)\,\dif(\partial_{x}V)(s,x)+2\int_{q}^{t}\left[\frac{1}{2}\partial_{x}^{2}(u^{2})-(\partial_{x}u)^{2}-\frac{1}{3}\partial_{x}(u^{3})\right](s,x)\,\dif s\\&\qquad+(t-q)\|\partial_{x}\rho\|_{L^{2}(\mathbb{R})}^{2}.
  \end{align*}
  Let $x_{1}<x_{2}$ and integrate in space to obtain
\begin{align*}
\int_{x_{1}}^{x_{2}} & [u(t,x)^{2}-u(q,x)^{2}]\,\dif x  \\
&                        =\int_{x_{1}}^{x_{2}}\int_{q}^{t}2u(s,x)\,\dif(\partial_{x}V)(s,x)\\&\qquad+2\int_{x_{1}}^{x_{2}}\int_{q}^{t}\left[\frac{1}{2}\partial_{x}^{2}(u^{2})-\frac{1}{2}(\partial_{x}u)^{2}-\frac{1}{3}\partial_{x}(u^{3})\right](s,x)\,\dif s\,\dif x    \\
                         & \qquad+(t-q)(x_{1}-x_{2})\|\partial_{x}\rho\|_{L^{2}(\mathbb{R})}^{2}                                                                                                                                                                     \\
                         & =\int_{x_{1}}^{x_{2}}\int_{q}^{t}2u(s,x)\,\dif(\partial_{x}V)(s,x)\\&\qquad+\int_{q}^{t}\Big[\Big(\partial_{x}(u^{2})-\frac{2}{3} u^{3}\Big)\Big|_{x=x_{1}}^{x=x_{2}}-\int_{x_{1}}^{x_{2}}(\partial_{x}u)^{2}(s,x)\,\dif x\Big]\,\dif s \\
                         & \qquad+(t-q)(x_{1}-x_{2})\|\partial_{x}\rho\|_{L^{2}(\mathbb{R})}^{2}.
  \end{align*}
  Taking expectations, we get
  \begin{equation}
    \begin{aligned}\int_{x_{1}}^{x_{2}}&\mathbb{E}[u(t,x)^{2}-u(q,x)^{2}]\,\dif x \\& = \int_{q}^{t}\left[\mathbb{E}\left[\partial_{x}(u^{2})-\frac{2}{3}u^{3}\right]_{x=x_{1}}^{x=x_{2}}-\mathbb{E}\int_{x_{1}}^{x_{2}}(\partial_{x}u)^{2}(s,x)\,\dif x\right]\,\dif s \\
    &\qquad+(t-q)(x_{2}-x_{1})\|\partial_{x}\rho\|_{L^{2}(\mathbb{R})}^{2}.
    \end{aligned}
    \label{eq:takeexpectationsinthegradientthing}
  \end{equation}
  Since the first expectation in \cref{eq:takeexpectationsinthegradientthing}
  is   finite by \cref{lem:L2bound} and the second term under the
 expectation in the right side has a sign, we can use \cref{lem:symmetrylemma} to conclude that
  \begin{equation}
    \mathbb{E}\left[\partial_{x}(u^{2})-\frac{2}{3}u^{3}\right]_{x=x_{1}}^{x=x_{2}}=0. \label{eq:usesymmetry}
  \end{equation}
 Using \cref{lem:L2bound,eq:usesymmetry} in \cref{eq:takeexpectationsinthegradientthing}, we get
  \[
    \mathbb{E}\int_{q}^{t}\int_{x_{1}}^{x_{2}}(\partial_{x}u)^{2}(s,x)\,\dif x\,\dif s\le 2(x_2-x_1)\|\rho\|^2_{L^2(\mathbb{R})} + (t-q)(x_{2}-x_{1})\|\partial_{x}\rho\|_{L^{2}(\mathbb{R})}^{2}.
  \]
  Since this holds for all $x_{1}<x_{2}$, we can conclude \cref{eq:derivL2bound}.
\end{proof}

Before proceeding, we record a version of the derivative bound in \cref{lem:gradientbound}
for solutions that start at a stationary distribution rather than a constant initial
condition. %
\begin{lem}\label{lem:grad-bound-stationary}
Let {$G=L\Zm$ or $G=\Rm$},  %
and let $X \sim \mathrm{Uniform}(\Lambda_G)$ be independent of all else.
  Suppose that~$\nu \in \overline{\mathscr{P}}_G(\R)$ is such that $\E v(X)^2 < \infty$ if 
  $\hbox{Law}(v)=\nu$.
  Then we have
  \begin{equation*}
    \E (\partial_x v(X))^2 = \|\partial_x \rho\|_{L^2(\R)}^2.
  \end{equation*}
\end{lem}

\begin{proof}
  Without loss of generality, we may assume that $G = \R/L\mathbb{Z}$ for some $L > 0$.
  Let $u$ solve \cref{eq:uPDE-1} with random initial condition $v$ independent of the noise $V$.
  We again integrate \cref{eq:uPDE-1} in space and time, and obtain \cref{eq:takeexpectationsinthegradientthing}.
The time-invariance of $\nu$ and the assumption $\E v(X)^2 < \infty$ implies
  \begin{equation*}
    0 = (t - q)\left\{\mathbb{E}\left[\partial_{x}(v^{2})-\frac{2}{3}v^{3}\right]_{x=x_{1}}^{x=x_{2}} - \mathbb{E} \int_{x_{1}}^{x_{2}}(\partial_{x}v)^{2}(x)\,\dif x + (x_{1}-x_{2})\|\partial_{x}\rho\|_{L^{2}(\mathbb{R})}^{2}\right\}.
  \end{equation*}
Furthermore, if we take $x_1 = 0$ and $x_2 = L$, the $G$-invariance of $\nu$ and \cref{lem:symmetrylemma} imply
  \begin{equation*}
    \mathbb{E}\left[\partial_{x}(v^{2})-\frac{2}{3}v^{3}\right]_{x=0}^{x=L}=0.
  \end{equation*}
  Therefore, we have
  \begin{equation*}
    \frac{1}{L} \mathbb{E} \int_{0}^{L}(\partial_{x}v)^{2}(x) \,\dif x = \|\partial_{x}\rho\|_{L^{2}(\mathbb{R})}^{2},
  \end{equation*}
  as desired.
\end{proof}

Combining \cref{lem:L2bound} and \cref{lem:gradientbound} yields the following uniform weighted
bound.
As a preliminary remark, we note that if $v\in H_{w}^{1}(\mathbb{R})$, then, by the Sobolev
embedding theorem applied locally, $v$ is continuous. Moreover, if
$w$ satisfies the condition $w(x)/w(y)\le C$ whenever $|x-y|\le1$,
then
\begin{equation}
  \|v\|_{\mathcal{C}_{w}(\mathbb{R})}\le 
  C\sup_{{j}\in\mathbb{Z}}\frac{\|v\|_{\mathcal{C}_{\mathrm{b}}([j,j+1])}}{w(j)}\le C\sup_{{j}\in\mathbb{Z}}\frac{\|v\|_{H^1([{j,j+1}])}}{w({j})}.\label{eq:goodweight}
\end{equation}

\begin{prop}\label{prop:H1bound}
  Consider the weight $w(x)=\langle x\rangle^{1/2}\log\langle x\rangle$.
  Suppose that $a\in\mathbb{R}$ and $u$ is a solution to \cref{eq:uPDE-1}
  with initial condition $u(0,\cdot)\equiv a$. Then there is a constant
  $C<\infty$ so that for all $t\ge0$, we have
  \[
    \int_t^{t+1}\mathbb{E}\|u(s,\cdot)-a\|_{\mathcal{C}_{w}(\mathbb{R})}^{2}\,\dif s\le C.
  \]
\end{prop}

\begin{proof}
  We have
  \begin{equation}\label{jul102}
  \begin{aligned} 
    \int_t^{t+1}\mathbb{E}\|u(s,\cdot)-a\|_{\mathcal{C}_{w}(\mathbb{R})}^{2}\,\dif s & \le C\sum_{j\in\mathbb{Z}}\frac{\int_t^{t+1}\mathbb{E}\|u(s,\cdot)-a\|_{\mathcal{C}_{\mathrm{b}}([j,j+1])}^{2}\,\dif s}{w(j)^{2}}\\&\le C\sum_{j\in\mathbb{Z}}\frac{\int_t^{t+1}\mathbb{E}\|u(s,\cdot)-a\|_{H^{1}([j,j+1])}^{2}\,\dif s}{w(j)^{2}} \\
                                                                 & \le C\sum_{j\in\mathbb{Z}}w(j)^{-2}\le C.
  \end{aligned}
  \end{equation}
The first two inequalities in \cref{jul102} are by the first and 
second inequality in \cref{eq:goodweight}, respectively, and the last is by \cref{lem:L2bound,lem:gradientbound}.
\end{proof}

\begin{proof}[Proof of \cref{prop:boundedtightness}.]
We need to extend the above estimates to the case of bounded rather than constant initial data. %
Since in this case the solutions are not space-stationary, an analogue of \cref{prop:gamma}
is less clear. However, as the initial condition is bounded, using the comparison
principle \cref{prop:comparisonprinciple}, we can sandwich the solution
between two solutions starting with constant initial conditions to
achieve a bound on the size of the solution, and then use the parabolic
regularity in \cref{lem:mildregularitypoly} to control the derivatives.

Let us now describe the details. It suffices to consider
  the case $N=1$, so let $u$ be the sole component
  of $\mathbf{u}$, and let $u_{\pm}$ solve \cref{eq:uPDE-1} with initial
  condition $u_{\pm}(0,\cdot)=\pm\|u(0,\cdot)\|_{L^\infty(\mathbb{R})}$,
  so that 
  \begin{equation}\label{jul104}
  u_-(t,x)\le u(t,x)\le u_+(t,x)~~\text{ for all $t\ge 0$ and $x\in\R$.}
  \end{equation}
  \cref{prop:comparisonprinciple} and \cref{lem:L2bound} then imply \cref{eq:supu2bdd}.
  
  To prove tightness, 
  consider the weight $\tilde w(x)=\langle x\rangle^{1/2}\log\langle x\rangle$ as in \cref{prop:H1bound}.
  By \cref{jul104} and \cref{prop:H1bound}, we have,
  \begin{align}
    \int_t^{t+1}\mathbb{E}\|u(s,\cdot)\|_{\mathcal{C}_{\tilde{w}}(\mathbb{R})}^{2}\,\dif s & \le\int_t^{t+1}\mathbb{E}\|u_{+}(s,\cdot)\|_{\mathcal{C}_{\tilde{w}}(\mathbb{R})}^{2}\,\dif s+\int_t^{t+1}\mathbb{E}\|u_{-}(s,\cdot)\|_{\mathcal{C}_{\tilde{w}}(\mathbb{R})}^{2}\,\dif s\le C.\label{eq:Cplbd}
  \end{align}
  Now let $\eps>0$.
  By \cref{prop:thetawellposed} and the discussion following it, we see that there is an $M<\infty$ such that, for any $t\ge 2$ and any $s\in [t-2,t-1]$, we have
  \begin{equation}\label{eq:applythm-bdd}
    \mathbb{P}(\|u(t-1,\cdot)\|_{\mathcal{C}_{\p_{3/4}}}\le M\mid\mathcal{F}_{s})\ge 1-\eps/3\qquad\text{a.s. on the event }\|u(s,\cdot)\|_{\mathcal{C}_{\tilde w}(\mathbb{R})}\le (3C/\eps)^{1/2}.
  \end{equation}
  Also by \cref{prop:thetawellposed} and the discussion following it, we see that for every $\eps>0$, there is a compact set $\mathcal{K}\subset\mathcal{X}_{1/2}$ such that for any $t\ge 2$, we have
  \begin{equation}\label{eq:applythm-cpct}
    \mathbb{P}(u(t,\cdot)\in\mathcal{K}\mid\mathcal{F}_{t-1})\ge 1-\eps/3\qquad\text{a.s. on the event }\|u(t-1,\cdot)\|_{\mathcal{C}_{\p_{3/4}}(\mathbb{R})}\le M.
  \end{equation}
Moreover, by \cref{eq:Cplbd}, for any $t\ge 2$ we have some $s\in [t-2,t-1]$ such that $\mathbb{E}\|u(s,\cdot)\|_{\mathcal C_{\tilde w}(\mathbb R)}^2\le C$, which means by Markov's inequality that
\begin{equation}\mathbb{P}(\|u(s,\cdot)\|_{\mathcal{C}_{\tilde{w}}(\mathbb{R})}\ge (3C/\eps)^{1/2})\le \eps/3.\label{eq:apply-Markov}\end{equation}
Combining \cref{eq:applythm-bdd,eq:applythm-cpct,eq:apply-Markov}, we see that for any $t\ge 2$ we have%
\[
\mathbb{P}\left(u(t,\cdot)\in \mathcal K \right) \ge 1-\eps,
\]
which completes the proof of the tightness.
\end{proof}

\subsection{Extremality of the limits\label{subsec:extremality}}

In this section we show that if the initial condition is periodic, any subsequential limit of the tight
family of laws considered in \cref{prop:boundedtightness} is extremal.

We will
need a couple of preparatory lemmas. The first is a simple calculation
similar to \cref{lem:L2bound}. 
\begin{lem}\label{lem:L2bound-periodic}
Let $G=L\Zm$ or $G=\R$, $X\sim\Uniform(\Lambda_G)$ be independent of everything else, and fix~$T > 0$.
  Suppose that $v \sim \nu \in \overline{\mathscr{P}}_G(\mathcal{X}_{1/2})$ 
  is independent of the noise $V$ and that $\E v(X)^2<\infty$.
  Let $h$ be a solution to~\cref{eq:KPZ} with initial condition such that $\partial_{x}h(0, \cdot)=v$.
  Then we have
  \begin{equation}
    \mathbb{E}v(X)^{2}=\|\rho\|_{L^{2}(\mathbb{R})}^{2}
    -\frac{2}{T}\mathbb{E}{[}h(T,X)-h(0,X)].\label{eq:EuST}
  \end{equation}
\end{lem}

\begin{proof} It is again sufficient to consider the case $L>0$. 
  Let $u$ solve \cref{eq:uPDE-1} with initial condition $v$.
  Integrating~\cref{eq:KPZ} in time from $0$ to $T$ and in space over $\Lambda_G$, taking expectations, and using \cref{eq:colehopf}, we have
  \begin{align}
    \mathbb{E}\int_{\Lambda_G}[h(T,x)-h(0,x)]\,\dif \lambda_G(x)=\frac{1}{2}\E\int_{0}^{T}\int_{\Lambda_G}[\partial_{x}u(t,x)-u(t,x)^{2}+\|\rho\|_{L^{2}(\mathbb{R})}^{2}]\,\dif\lambda_G(x)\,\dif t.\label{eq:integrateKPZ}
  \end{align}
Note that 
\begin{equation}\label{jul120}
\E \int_{\Lambda_G} \partial_xu(t,x)\,\dif \lambda_G(x)=\E[u((t,L)-u(t,0)]=0
\end{equation}
simply by the $G$-invariance in law of $u$ and since $\E(u(t,x))^2)<+\infty$. (We do not have to use \cref{lem:symmetrylemma} here.)

As $\nu$ is time-stationary,~\cref{eq:integrateKPZ} becomes
  \[
    \mathbb{E}{[}h(T,X)-h(0,X)]=-\frac{T}{2}\mathbb{E}v(X)^{2}+\frac{T}{2}\|\rho\|_{L^{2}(\mathbb{R})}^{2},
  \]
  which implies \cref{eq:EuST}.
\end{proof}

The next lemma compares the expectation of a KPZ evolution started with random initial 
condition to that of the KPZ equation started at deterministic bounded initial data.
\begin{lem}\label{lem:KPZexpsubadd}
  Fix $m \in (0, 1)$ and $G=L\Zm$ or $G=\R$, and
  let $X \sim \mathrm{Uniform}(\Lambda_G)$ be independent of all else.
  Let $v \in L^\infty(\R)$ be $G$-invariant and let $\tilde{v} \sim \tilde{\nu} \in \overline{\mathscr{P}}_{G}(\mathcal{X}_{m})$ satisfy $\E \tilde{v}(X)^2 < \infty$.
  Let $u$ and $\tilde{u}$ solve \cref{eq:uPDE-1} with initial data $v$ and $\tilde{v}$, respectively.
  Consider solutions $h$ and $\tilde{h}$ to \cref{eq:KPZ} obtained 
  from $u$ and $\tilde{u}$, respectively, as in \cref{prop:KPZsolntheory}.
  Then for all $t \geq 0$ and $x \in \R$ we have
  \begin{equation}\label{eq:expectationboundofh}
    \mathbb{E}[\tilde{h}(t,x)-h(t,x)]\le\sup_{x\in\R}[\E \tilde{h}(0, x) - h(0, x)].
  \end{equation}
\end{lem}
This is a version of the comparison principle for the KPZ equation. 
The difficulty comes from the fact that $\tilde h(0,x)$ is not necessarily uniformly
bounded, so the standard pathwise comparison principle would be vacuous, as the right
side of \cref{eq:expectationboundofh} would be infinite without taking the expectation. 
Note that,  as we use solutions coming from \cref{prop:KPZsolntheory}, the expectations $\E\tilde h(t,x)$ and~$\E h(t,x)$ are finite
under the assumptions of \cref{lem:KPZexpsubadd}.
We postpone the proof of this lemma until \cref{sec:5.3} and first explain how it is used
to show the extremality of the limiting invariant measures.
The first step is to show that a limiting solution started from deterministic bounded initial 
condition has minimal variance among all spacetime-stationary solutions.
\begin{prop}\label{prop:minvariance}
Fix $m\in(0,1)$ and let $G=L\Z$ or $G=\R$. Fix
a deterministic, $G$-invariant function $v\in L^{\infty}(\mathbb{R})$
and let $\delta_{v}\in {\mathscr{P}}_{G}(\mathcal{X}_{m})$ 
be the delta measure on $v$. Suppose that~$T_{k}\to+\infty$, and 
  \begin{equation}\label{jul118}
  \frac{1}{T_k}\int_{0}^{T_{k}}P_{t}^{*}\delta_{v}\,\dif t\to \nu\in\mathscr{P}(\mathcal{X}_{m}),
  \end{equation}
weakly. 
Let $w\sim\nu$, $\tilde{w}\sim\tilde{\nu}\in\overline{\mathscr{P}}_{G}(\mathcal{X}_{m})$, and $X\sim\Uniform(\Lambda_G)$
  be independent of everything else. Then we have
  \[
    \Var [w(X)] \le\Var[\tilde{w}(X)].
  \]
\end{prop}
\begin{proof}
If $\Var \tilde{w}(X) = \infty$ then there is nothing to show, so we may 
assume that $\Var \tilde{w}(X) <\infty$.
  By \cref{prop:shearinvariance}, it suffices to consider the case 
  \begin{equation}\label{jul110}
  \E v(X) =\int_{\Lambda_G}v(x)d\lambda_G(x)=0.
  \end{equation}
  Let $u\in\mathcal{Z}_{m}$ be a solution to \cref{eq:uPDE-1} with initial
  condition $u(0,\cdot)=v$. By \cref{prop:L1omegaconservation} and the uniform integrability coming
  from \cref{lem:L2bound}, we have that
  \begin{equation}
    \mathbb{E}w(X)=\mathbb{E}\left[\int_{\Lambda_G}w(x)\,\dif \lambda_G(x)\right]=\lim_{k\to\infty}\int_{\Lambda_G}\mathbb{E}u(S_{T_{k}},x)\,\dif \lambda_G(x)=0,\label{eq:Eintwiszero}
  \end{equation}
  where~$S_{T}\sim\Uniform([0,T])$
  is independent of everything else. 
  Let $h$ be a solution to \cref{eq:KPZ} such that
  \begin{equation}\label{jul128}
  \partial_{x}h(0, \cdot)=v(0, \cdot),~~\hbox{ and }
  \int_{\Lambda_G}h(0,x)\,\dif \lambda_G(x)=0,
  \end{equation}
  so that $\partial_{x}h(t,\cdot) = u(t, \cdot)$.
  We see from \cref{jul110}   and the second condition in \cref{jul128} that 
  \begin{equation}\label{jul112}
  \|h(0,\cdot)\|_{L^\infty(\R)}\le\|v\|_{L^\infty(\R)}.
  \end{equation}
 Hence, the comparison principle implies that 
 \[
 |h(t,x)-h_0(t,x)|\le \|v\|_{L^\infty(\R)}.
 \]
 Here, $h_0(t,x)$ is the solution  
 to \cref{eq:KPZ} with the initial condition $h_0(0,x)\equiv 0$. As a consequence,
 we know that $\E h(t,x)$ is finite for all $t\ge 0$. 
  
  By %
  Fatou's lemma and \cref{jul118},
  we have
  \begin{equation}
    \Var w(X)=\mathbb{E}\left[\int_{\Lambda_G}w(x)^{2}\,\dif \lambda_G(x)\right] \le\liminf_{k\to\infty}\mathbb{E}\left[\int_{\Lambda_G}u(S_{T_{k}},x)^{2}\,\dif \lambda_G(x)\right]. 
  \end{equation}
The expectation in the right side is finite by \cref{prop:boundedtightness}.
Recalling \cref{eq:integrateKPZ},
noting that \cref{jul120} relies only on the spatial stationarity of $u$, and using 
\cref{jul112}, 
we see that
\begin{equation}
\begin{aligned}
\Var w(X)\le\liminf_{k\to+\infty}
\mathbb{E}\left[\int_{\Lambda_G}u(S_{T_{k}},x)^{2}\,\dif \lambda_G(x)\right]
     =\|\rho\|_{L^{2}(\mathbb{R})}^{2}-\limsup_{k\to\infty}\frac{2}{T_{k}}\mathbb{E}\left[\int_{\Lambda_G}h(T_{k},x)\,\dif \lambda_G(x)\right].\label{eq:VarwXL}
\end{aligned}
\end{equation}
Next, by \cref{prop:shearinvariance}, we can also assume without loss of generality that
\begin{equation}\label{jul122}
 \int_{\Lambda_G} \E \tilde{w}(y)\,\dif\lambda_G( y) = 0,
 \end{equation}
 since we can subtract off the appropriate mean without changing the variance.
  Let $\tilde{u}$ solve \cref{eq:uPDE-1} with initial condition $\tilde{w}$, and construct a solution $\tilde{h}$ of \cref{eq:KPZ} from $\tilde{u}$ using \cref{prop:KPZsolntheory}, so that, 
in particular,~$\partial_x \tilde{h} = \tilde{u}$. As  $\tilde h(0,x)$ satisfies \cref{jul124}, there exists a random point $x_0\in\supp\zeta$ such that~$\tilde h(0,x_0)=0$. 
As $\E\tilde{w}$ is $G$-invariant and satisfies \cref{jul122}, we may   write 
\begin{equation}\label{jul130}
|\E \tilde h(0,x)|=\Big|\E\int_{x_0}^x  \tilde w(y)\,\dif y\Big|\le \E\Big|\int_{x_0}^x\tilde w(y)\,\dif y\Big|\le \E\int_0^L|\tilde w(y)|\,\dif y
<L^{1/2}(\E(\tilde w(X)^2))^{1/2}
<\infty.
\end{equation}
In addition, directly from \cref{jul140} we get
\[
\E |\tilde{h}(t, x)| < \infty \hbox{ for all $t>0$, $x\in\R$.}
\]
As $\mathbb{E}\tilde{u}(0,X)= \E \tilde{w}(X) =0$, \cref{lem:L2bound-periodic} implies
that for all $t > 0$  we have
  \begin{equation}
    \Var\tilde{w}(X)=\mathbb{E}\tilde{u}(t,X)^{2}=\|\rho\|_{L^{2}(\mathbb{R})}^{2}-\frac{2}{t}\mathbb{E}[\tilde{h}(t,X)-\tilde{h}(0,X)].\label{eq:Varwtilde}
  \end{equation}
  Also, it follows from \cref{lem:KPZexpsubadd}
  that
  \begin{align}
    \mathbb{E}\tilde{h}(t,X) & \le\mathbb{E}h(t,X)+\|\E\tilde h(0,\cdot)\|_{\mathcal{C}_{\mathrm{b}}(\R)} + \|h(0,\cdot)\|_{\mathcal{C}_{\mathrm{b}}(\mathbb{R})}.
                               \label{eq:comparehhtilde-2}
  \end{align}
  Therefore, we have
  \begin{align*}
    \Var w(X) & \le\|\rho\|_{L^{2}(\mathbb{R})}^{2}-\limsup_{k\to\infty}\frac{2}{T_{k}}\mathbb{E}h(T_{k},X)                                         \\
              & \le\|\rho\|_{L^{2}(\mathbb{R})}^{2}-\limsup_{k\to\infty}\frac{2}{T_{k}}\left(\mathbb{E}\tilde{h}(T_{k},X)
		-\|\E\tilde h(0,\cdot)\|_{\mathcal{C}_{\mathrm{b}}(\R)}-\|h(0,\cdot)\|_{\mathcal{C}_{\mathrm{b}}(\mathbb{R})}\right)\nonumber                         \\
              & =\limsup_{n\to\infty}\left(\|\rho\|_{L^{2}(\mathbb{R})}^{2}-\frac{2}{T_{k}}[\mathbb{E}\tilde{h}(T_{k},X)-\E\tilde{h}(0,X)]\right)=\Var\tilde{w}(X),                                                                                                                        
  \end{align*}
  where the first inequality is \cref{eq:VarwXL}; the second inequality
  is by \cref{eq:comparehhtilde-2}; the first equality is by {\cref{jul112}}; and the second equality is 
  by \cref{eq:Varwtilde}.
\end{proof}
\begin{prop}
  \label{prop:limitsareergodic}
  Suppose that $G=L\Zm$ with $L>0$ or $G=\Rm$, and that a deterministic 
  function~$v\in L^{\infty}(\mathbb{R})$ is~$G$-invariant. Let $\delta_{v}$ be the measure on $L_{\p_{1/2+}}^{\infty}(\mathbb{R})$
  with a single atom at $v$. If~$s\ge0$,~$T_{k}\uparrow\infty$,
  and $\nu\in\mathscr{P}(\mathcal{X}_{1/2})$ are such that
  \[
    \nu = \lim_{k\to\infty}\frac{1}{T_{k}}\int_{s}^{s+T_{k}}P_{t}^{*}\delta_{\mathbf{v}}\,\dif t
  \]
  in the sense of weak convergence of probability measures on $\mathcal{X}_{1/2}$,
  then $\nu\in\overline{\mathscr{P}}_{G}^{\mathrm{e}}(\mathcal{X}_{1/2})$.
\end{prop}

\begin{proof}
  The fact that $\nu\in\overline{\mathscr{P}}_{G}(\mathcal{X}_{1/2})$
  is \cref{prop:krylovbogoliubov}, so we only need to show that $\nu$
  is extremal. Suppose that we can decompose $\nu$ as
  \[
  \nu=(1-q)\mu_{0}+q\mu_{1},
  \]
  for some $q\in(0,1)$ and $\mu_{0},\mu_{1}\in\overline{\mathscr{P}}_{G}(\mathcal{X}_{1/2})$.
  By \cref{prop:couple}, there exists a coupling $\mu\in\overline{\mathscr{P}}_{G}(\mathcal{X}_{1/2}^{2})$
  of $\mu_{0}$ and $\mu_{1}$. Let $(v_{0},v_{1})\sim\mu$ and consider $v_I$,
 where 
  $I\sim\Bernoulli(q)$ is a random variable, independent of everything
  else.   Then, $v_I$ is distributed according to $\nu$: 
  \[
  \Law(v_{I})=(1-q)\mu_{0}+q\mu_{1}=\nu.
  \]
  By \cref{thm:ordering},  the sign 
  $\chi\coloneqq\sgn(v_{0}(x)-v_{1}(x))$ {$\mu$}-almost surely
  does not depend on $x$, and
  by the comparison principle in \cref{prop:comparisonprinciple} we know that~$\chi$ is invariant under the dynamics \cref{eq:uPDE-many}. Hence, the restrictions   
  of $\mu$ onto the sets $\{\chi=b\}\subset\cX_{1/2}^2$ are invariant, for each $b\in\{-1,0,1\}$,
  as are
  \begin{equation}
    \nu_{i,b}\coloneqq\Law(v_{I}\mid I=i,\chi=b)\in\overline{\mathscr{P}}_{G}(\mathcal{X}_{m})\label{eq:conditionalinvariance}
  \end{equation}
  for all $i\in\{0,1\}$ and $b\in\{-1,0,1\}$ such that $\mathbb{P}[\chi = b] > 0$.

  Let now $X\sim\Uniform(\Lambda_G)$ be independent of everything else.
  By the law of total variance, we have
  \begin{equation}\label{eq:lawoftotalvariance}
 \Var(v_{I}(X))=\mathbb{E}\Var(v_{I}(X)\mid I,\chi)+\Var(\mathbb{E}[v_{I}(X)\mid I,\chi]).
  \end{equation}
 As $\Law(v_{I})=\nu$, and $\nu_{i,b}$ are invariant, 
 by \cref{prop:minvariance} and \cref{eq:conditionalinvariance},
  we have
  \[
  \Var(v_I(X)\mid I,\chi)
 \ge \Var(v_I(X)),~~\hbox{a.s.}
  \]
 In light of \cref{eq:lawoftotalvariance},
  this means that 
  \[
  \Var(\mathbb{E}[v_{I}(X)\mid I,\chi])=0,
  \]
  so that $\mathbb{E}[v_{I}(X)\mid I,\chi]$ 
  is constant almost surely. On
  the other hand, if $\chi=1$, then $v_{0}(x)>v_{1}(x)$ almost surely, thus
  \[
  \mathbb{E}[v_{0}(X)\mid\chi=1]>\mathbb{E}[v_{1}(X)\mid\chi=1].
  \]
  Similarly, if $\chi=-1$, then $v_{0}(x)<v_{1}(x)$ almost surely, and
  \[
  \mathbb{E}[v_{0}(X)\mid\chi=-1]<\mathbb{E}[v_{1}(X)\mid\chi=-1].
  \]
  The only way these facts can be consistent with the almost-surely-constant
  nature of $\mathbb{E}[v_{I}(X)\mid I,\chi]$ is that $\chi=0$ almost
  surely. Therefore, $v_{0}=v_{1}$ almost surely, which means that
  $\mu_{0}=\mu_{1}$. This implies that $\nu$ is extremal.
\end{proof}

\subsection{Proof of \texorpdfstring{\cref{lem:KPZexpsubadd}}{Lemma~\ref{lem:KPZexpsubadd}}}
\label{sec:5.3}

\begin{proof} Let us recall again that the claim of \cref{lem:KPZexpsubadd}
would be just the comparison principle for the KPZ equation, and would hold
without taking the expectation, if only
it were true that the initial conditions $\tilde h(0,x)$ and $h(0,x)$ are bounded.
Indeed, the difference $H(t,x)=h(t,x)-\tilde h(t,x)$ satisfies
\[
\partial_tH+\farc{1}{2}(\partial_xh+\partial_x\tilde h)\partial_xH=\farc{1}{2}\partial_x^2H,
\]
and the maximum principle would imply
\[
H(t,x)\le \sup_{x}H(0,x)\hbox{ for all $t>0$ and $x\in\Rm$,}
\]
if the supremum in the right side were finite.
  
As usual, we assume without loss of generality that $G = L \mathbb{Z}$ for some $L > 0$,
and let $X$ be uniformly distributed on $[0,L]$.
Note that, as spatial integrals of a periodic functions, $h(0,\cdot)$ and~$\tilde h(0,\cdot)$ 
satisfy
\[
h(0, x) \sim [\E v(X)] x,~~\tilde{h}(0, x) \sim [\E \tilde{v}(X)] x,~\hbox{ as $|x| \to \infty$.}
\]
Thus we can assume $\E v(X) = \E \tilde{v}(X)$, since otherwise the supremum on the right side 
of \cref{eq:expectationboundofh} is infinite and the statement is
vacuous.
By \cref{prop:shearinvariance}, we may assume that~$\E v(X) = \E \tilde{v}(X) = 0$,
so that $\E\tilde{h}(0,x)$ is continuous and $L$-periodic in $x$.
  
Let $\mathrm{E}$ denote expectation in which only the initial condition $\tilde{h}(0, \cdot)$ 
and $X$ are treated as random and the noise is considered to be deterministic.
  That is, $\EE[Y]=\E[Y\mid\mathcal{F}_\infty]$, where $\mathcal{F}_\infty$ is the~$\sigma$-algebra generated by the noise. Taking formally the $\EE$ expectation of \cref{eq:KPZ} gives 
  \begin{equation}\label{jul212}
    \dif (\EE \tilde{h}) \le \frac{1}{2}[\partial^2_x (\EE \tilde h) -  (\partial_x (\EE \tilde h))^2 + \|\rho\|_{L^2(\R)}^2]\dif t+\dif V,
  \end{equation}
Subtracting another copy of \cref{eq:KPZ} for $h$, we find 
  \begin{equation}\label{jul202}
    \partial_t (\EE \tilde h - h) \le \frac{1}{2}[\partial_x^2(\EE\tilde{h}-h)-\partial_x(\EE\tilde h-h)\cdot\partial_x(\EE\tilde h +h)]. 
  \end{equation}  
Once again, if the comparison principle could be applied to \cref{jul202}, we would be essentially
done. However, we need to justify that $\EE \tilde h$ exists and is sufficiently regular, so that the formal manipulations
are justified. %
This propagation of moments does not follow from the regularity results we have obtained so far,
and we will use stationarity of $\tilde\nu$ instead.

We denote $\tilde u=\partial_x\tilde h$.  
The time-invariance and \cref{lem:grad-bound-stationary} imply that
  \begin{equation}\label{jul208}
    \E \|\EE \tilde{u}(t, \cdot)\|_{H^1(\Lambda_G)}^2 \leq \E \|\tilde{u}(t, \cdot)\|_{H^1(\Lambda_G)}^2 = \E \|\tilde{v}\|_{H^1(\Lambda_G)}^2 = L\E \tilde v(X)^2 + L \|\partial_x \rho\|_{L^2(\R)}^2 < \infty
  \end{equation}
  for any $t \in [0, T]$. Here,  $\Lambda_G$ is  a fundamental domain for $G$, and as usual
  we use the notation $\E$ for expectation on the largest probability space involving all objects.
  Now fix $\ell \in \big(m \vee {1}/{2}, 1\big)$.
  Morrey's inequality implies
  \begin{equation*}
    \|\EE \tilde{u}(t, \cdot)\|_{\mathcal{C}_{\p_\ell}^{1/2}(\R)} \leq C \sup_{j \in \mathbb{Z}} \frac{\|\EE \tilde{u}(t, \cdot)\|_{\mathcal{C}^{1/2}(\Lambda_G + Lj)}}{{\langle Lj \rangle}^{\ell}} \leq C \sup_{j \in \mathbb{Z}} \frac{\|\EE \tilde{u}(t, \cdot)\|_{H^1(\Lambda_G + jL))}}{{\langle j \rangle }^{\ell}}
  \end{equation*}
with  some constant $C$ depending on $L$.
  Arguing as in the proof of \cref{prop:H1bound}, and using \cref{jul208}, we obtain 
  \[
  \E \|\EE \tilde{u}(t, \cdot)\|_{\mathcal{C}_{\p_\ell}^{1/2}(\R)}^2 < C, %
  \] 
  uniformly in $t$.
  Integrating in $t$, we find that for any $T\ge 0$, we have
  \begin{equation}
    \label{eq:exp-u}
    \EE \tilde{u} \in L^2([0, T]; \mathcal{C}_{\p_\ell}^{1/2}(\R))
  \end{equation}
  almost surely in the noise. Next, recall that 
  \[
  \EE(\partial_x \tilde{h})^2 = \EE \tilde{u}^2,
  \]
 and
 \[
 \E \tilde{u}(t, X)^2 = \E \tilde{v}(X)^2 < \infty.
 \]
  Also, \cref{lem:grad-bound-stationary} implies
  \begin{equation*}
    \E |\partial_x(\tilde{u}^2)(t, X)| \leq 2 \E[ |\tilde{u}(t, X)| |\partial_x \tilde{u}(t, X)|] \leq \E \|\tilde{v}\|_{H^1(\Lambda_G)}^2  < \infty.
  \end{equation*}
It follows that 
\[
\E \|\EE \tilde{u}^2(t, \cdot)\|_{W^{1,1}(\Lambda_G)} < C, %
\]
uniformly in $t$.
  Since $W^{1,1}(\Lambda_G)$ embeds continuously into $\mathcal{C}_{\mathrm{b}}(\Lambda_G)$, we have
  \begin{equation*}
    \E \|\EE \tilde{u}^2(t, \cdot)\|_{\mathcal{C}_{\p_{\ell + 1/2}}(\R)} \leq C \sum_{j \in \mathbb{Z}} \frac{\E \|\EE\tilde{u}^2(t, \cdot))\|_{W^{1,1}(\Lambda_G + Lj)}}{{\langle Lj \rangle}^{\ell + 1/2}} \leq C \sum_{j \in \mathbb{Z}} {\langle j \rangle}^{-(\ell + 1/2)} < \infty.
  \end{equation*}
  Thus,
  \begin{equation}
    \label{eq:exp-u-square}
    \EE \tilde{u}^2 = \EE (\partial_x \tilde{h})^2 \in L^1([0, T]; \mathcal{C}_{\p_{\ell + 1/2}}(\R))
  \end{equation}
  almost surely in $V$.
Hence, it makes sense to take the expectation $\EE$ of the
  mild formulation
  \begin{equation}
    \label{eq:KPZ-mild}
    \begin{split}
      \tilde{h}(t, x) &= G_t \ast \tilde{h}(0, x) - \frac{1}{2}\int_0^t[G_{t - s} \ast (\partial_x \tilde{h}(s, \cdot))^2](x) \, \dif s + \frac{t}{2} \|\rho\|_{L^2(\R)}^2\\
      & \qquad+ \int_0^t \int_\R (G_{t - s} \ast \rho)(x - y) \, \dif W(s, y),
    \end{split}
  \end{equation}
of \cref{eq:KPZ}. %
  As in the proof of \cref{prop:KPZsolntheory}, define
  \begin{equation*}
    \omega(t, x) = \int_0^t \! \int_\R (G_{t - s} \ast \rho)(x - y) \, \dif W(s, y).
  \end{equation*}
  Then $\tilde{g} = \EE\tilde{h} - \omega$ satisfies the mild equation
  \begin{equation}
    \label{eq:KPZ-PDE-mild}
    \tilde{g}(t, \cdot) = G_t \ast \EE \tilde{h}(0, \cdot) - \frac{1}{2} \int_0^t G_{t - s} \ast [\EE \tilde{u}(s, \cdot)^2] \, \dif s + \frac{t}{2} \|\rho\|_{L^2(\R)}^2.
  \end{equation}
  By the fact that $\EE \tilde{h}(0, \cdot) \in \mathcal{C}(\R/ L \mathbb{Z})$ and \cref{eq:exp-u-square}, we have
  \begin{equation*}
    \tilde{g} \in L^\infty([0, T]; L_{\p_{\ell + 1/2}}^\infty(\R))
  \end{equation*}
  and $\tilde{g}$ is continuous in time. Standard Gaussian process estimates show that 
  $ \omega(t,\cdot)$ is a continuous process in $\mathcal{C}_{\p_{\ell'}}(\R)$ for any $\ell'>0$,
  so
  \begin{equation}
    \label{eq:exp-h}
    \EE \tilde{h} \in L^\infty([0, T]; L_{\p_{\ell + 1/2}}^\infty(\R)).
  \end{equation}
  As a spatial integral, $\EE \tilde{h}$ is differentiable in space.
  Since $\tilde{g}$ is continuous in time, $\EE \tilde{h}$ is in fact spacetime-continuous.

We may now apply Jensen's inequality to \cref{eq:KPZ-mild} to obtain \cref{jul212} and
  \cref{jul202}
in the mild sense.
The next natural step is to apply the maximum principle to \cref{jul202}
to conclude that $\EE \tilde h - h$ never exceeds the maximum of its initial data.
However, \cref{jul202} need only hold in the mild sense, and, in addition,
the drift coefficient $\partial_x(\E \tilde{h} + h)$ in \cref{jul202}
is potentially irregular in time and grows polynomially in space.
  We therefore directly verify that \cref{jul202} obeys the maximum principle.
Let
  \begin{equation*}
    z(t,x) = \EE \tilde{h}(t,x) - h(t,x) - \sup_{y\in\R}[\EE \tilde{h}(0, y) - h(0, y)]
  \end{equation*}
  and $b(t,x) = \partial_x (\EE \tilde{h} + h)(t,x)$, so that \cref{jul202} can be written as 
  \begin{equation}
    \label{eq:mild-ineq}
    \partial_t z \leq \frac 1 2 \partial_x^2 z - b \partial_x z,
  \end{equation}
but this only holds  in the mild sense. In addition, we have  $z(0, x) \leq 0$ for all $x\in\R$. 
 As for regularity of $z$ and $b$,  by \cref{eq:exp-u} and \cref{eq:exp-h}, we know that
  $z \in L^\infty([0, T]; L_{\p_{\ell + 1/2}}^\infty(\R))$ is continuous in time and differentiable in space.
  Moreover, by \cref{eq:exp-u}, both $\partial_x z$ and $b$ lie in~$L^2([0, T]; \mathcal{C}_{\p_\ell}^{1/2}(\R))$.
     
  To obtain an honest differential inequality, we mollify \cref{eq:mild-ineq} in space.
  Let $\zeta \in \mathcal{C}_c^\infty(\R)$ be non-negative with $\|\zeta\|_{L^1(\R)} = 1$.
  Define $\zeta_\delta(x) = \delta^{-1} \zeta(\delta^{-1} x)$ for $\delta > 0$, and write $z_\delta = \zeta_\delta \ast z$,
so
  \begin{equation}
    \label{eq:moll-ineq}
    \partial_t z_\delta \leq \frac{1}{2} \partial_x^2 z_\delta - b \partial_x z_\delta + [b (\zeta_\delta \ast \partial_x z) - \zeta_\delta \ast (b \partial_x z)].
  \end{equation}
This inequality is now valid pointwise.
To make use of this  inequality, we must show that the bracketed commutator 
   \begin{equation*}
    E \coloneqq b (\zeta_\delta \ast \partial_x z) - \zeta_\delta \ast (b \partial_x z)
  \end{equation*}
in~\cref{eq:moll-ineq} is sufficiently small.
Note that
  \begin{equation}\label{jul214}
  \begin{aligned}
    |E(t, x)| &\leq \int_\R |b(t, x) - b(t, y)| |\partial_x z(t, y)| \zeta_\delta(x - y) \, \dif y
    \leq C F_1(t) \delta^{1/ 2} {\langle x \rangle}^\ell \int_\R {\langle y \rangle}^\ell \zeta_\delta(x - y) \, \dif y \\
    &\leq C_E \delta^{1/ 2} F_1(t) {\langle x \rangle}^{2\ell},
\end{aligned}
  \end{equation}
where
  \begin{equation*}
    F_1(t) = \|b(t, \cdot)\|_{\mathcal{C}_{\p_\ell}^{1/2}(\R)}  \|\partial_x z(t, \cdot)\|_{\mathcal{C}_{\p_\ell}^{1/2}(\R)} \in L^1([0, T]).
  \end{equation*}
Hence,  the commutator can be made arbitrarily small in $L^1([0, T]; L_{\p_{2\ell}}^\infty(\R))$.

We now construct a super-solution to \cref{eq:moll-ineq}. The function 
  \begin{equation*}
    H(t, x) = \sqrt{\frac{3T}{3T - t}} \exp\left[\frac{x^2}{4(3T - t)}\right].
  \end{equation*}
 satisfies 
 \[
 \partial_t H = \partial_{x}^2 H,
 \] 
 and there exists $C_H \in [1, \infty)$ such that
  \begin{equation*}
    \partial_x^2H \geq C_H^{-1} \langle x \rangle^2H \quad \textrm{and} \quad |\partial_x H| \leq C_H \langle x \rangle H
  \end{equation*}
  on $\left[0, 2T\right] \times \R$.
  Define
  \begin{equation*}
    F_2(t) = \|b(t, \cdot)\|_{\mathcal{C}_{\p_\ell}(\R)} \in L^2([0, T]) \subset L^1([0, T])
  \end{equation*}
  and let
  \begin{equation*}
    \eta(t) = t + c \int_0^t [F_1(s) + F_2(s)]\, \dif s
  \end{equation*}
  with $c > 0$ chosen so that $\eta(T) \leq 2T$.
  We use $\eta$ as a time-change in $H$:
  \begin{equation*}
    \begin{split}
      \partial_t [H(\eta(t), x)] &= \frac{\dot\eta}{2} \partial_{x}^2 H + \frac{\dot\eta}{2} \partial_{x}^2 H
      \geq \frac 1 2 \partial_{x}^2 H - b \partial_x H + \left[\frac{c}{2C_H} \langle x \rangle^2 - C_H \langle x \rangle^{1 + \ell}\right] F_2(t) H.
    \end{split}
  \end{equation*}
  Since $\ell < 1$, the term in brackets is positive when $x$ is large.
  To handle small $x$, we multiply by a time-dependent factor.
 There exists $\kappa > 0$ such that
  \begin{equation*}
    \partial_t [\e^{\kappa \eta(t)} H(\eta(t), \cdot)] > \frac 1 2 \partial_{x}^2 (\e^{\kappa \eta}  H) - b \partial_x (\e^{\kappa \eta} H) + c \kappa F_1(t) \e^{\kappa \eta} H.
  \end{equation*}
  Now let
  \begin{equation*}
    \eps = \frac{C_E \delta^{ 1 /2}}{c\kappa} \sup_{(t, x) \in [0, T] \times \R} \frac{\langle x \rangle^{2\ell}}{H(t, x)},
  \end{equation*}
  so that $\eps c \kappa F_1 H \geq |E|$ by \cref{jul214}.
  Then
  \begin{equation*}
    \tilde{z}(t, x) = z_\delta(t, x) - \eps \e^{\kappa \eta(t)} H(\eta(t), x)
  \end{equation*}
  satisfies $\tilde{z}(0, \cdot) \leq -\eps$ and
  \begin{equation}
    \label{eq:heat-ineq}
    \partial_t \tilde{z} < \frac 1 2 \partial_x^2 \tilde{z} - b \partial_x \tilde{z}.
  \end{equation}
 In addition, by \cref{eq:exp-h}, $H$ grows much faster than $z_\delta$ in space. The standard maximum principle 
 then implies that $\tilde z(t,x)< 0$ for all $t>0$ and $x\in\Rm$.
  Thus, we have
  \begin{equation}
    \label{eq:delta-bound}
    z_\delta(t, x) \leq \eps \e^{\kappa \eta(t)} H(\eta(t), x) \leq C \delta^{1/ 2} \exp\left(\frac{x^2}{4T}\right)
  \end{equation}
  for some $C$ depending on $T, m,$ and $L$, but not on $\delta$.
  
  We now take $\delta \to 0$.
  Since $z$ is continuous, $z_\delta \to z$ pointwise.
  Therefore \cref{eq:delta-bound} implies $z \leq 0$, as desired.
  It follows that
  \[
    \sup_{t\ge 0,x\in\R}(\EE\tilde{h}-h)(t,x)\le \sup_{x\in\R}(\EE\tilde{h}-h)(0,x).
  \]
  Taking expectations with respect to the noise, we obtain \cref{eq:expectationboundofh}.
\end{proof}

\section{Classification of extremal invariant measures\label{sec:classification}}

In this section we identify all elements of $\overline{\mathscr{P}}_{G}^{\mathrm{e}}(\mathcal{X}_{m})$, the extremal invariant measures,
proving \cref{thm:maintheorem-classification}. Here is the key step.
\begin{prop}
  \label{prop:onlycouplingisthesame}
 Let $m\in(0,1)$, $G=L\Zm$ with $L>0$ or $G=\R$, and $\Lambda_G$ be a fundamental
  domain for~$G$. 
 Suppose that $\nu\in\overline{\mathscr{P}}_{G}(\mathcal{X}_{m}^{2N})$, and 
let
  \[
    \mathbf{v}=(\mathbf{v}_{1},\mathbf{v}_{2})=((v_{1,1},\ldots,v_{1,N}),(v_{2,1},\ldots,v_{2,N}))\sim\nu,
  \]
  and $\nu_{i}=\Law(\mathbf{v}_{i})$, so that
$\nu_{1},\nu_{2}\in\overline{\mathscr{P}}_{G}(\mathcal{X}_{m}^{N})$,    and let $\nu_{i,j}$ be the marginal law of~$\nu_{i}$ on the $j$th
  component.
  Suppose that  for each $i=1,2$,  $j=1,\dots,N$, we have~$\nu_{i,j}\in\overline{\mathscr{P}}_{G}^{\mathrm{e}}(\mathcal{X}_{m})$,
  and 
\begin{equation}\label{jul604}
      \E\left[ \int_{\Lambda_G} v_{i, j}(x)^2 \, \dif \lambda_G(x) \right] < \infty.
    \end{equation}
If, in addition, we have
    \begin{equation}\label{jul602}
      \mathbb{E}\left[\int_{\Lambda_G}v_{1,j}(x)\,\dif\lambda_{G}(x)\right]=\mathbb{E}\left[\int_{\Lambda_G}v_{2,j}(x)\,\dif\lambda_{G}(x)\right],~~ 
      \hbox{for each $j=1,\dots,N$, }
 \end{equation}
then $\mathbf{v}_{1}=\mathbf{v}_{2}$ almost surely.
\end{prop}
\begin{proof}
  By \cref{thm:ordering} with hypothesis \cref{enu:hyp-stat},
  the random variables $\chi_{j}\coloneqq\sgn(v_{1,j}(x)-v_{2,j}(x))$
  do not depend on $x$. For given $i,j,$ and $b \in \{-1, 0, 1\}$, 
(i) if  $\Pm\{\chi_{j}=b\}\neq 0$, we
  let $\nu_{i,j,b}$ be the conditional marginal distribution of $v_{i,j}$ on the event $\chi_j =b$,
  and (ii) if   $\mathbb{P}(\chi_{j}=b)=0$, we set $\nu_{i,j,b}=\Law(v_{i,j})$. By \cref{prop:comparisonprinciple}, the
  Burgers evolution preserves ordering, so $\nu_{i,j,b}$ is invariant
  under the evolution. Moreover, $\nu_{i,j,b}$ is $G$-invariant since
  it is the distribution of a $G$-invariant measure conditional on
  a~$G$-invariant event. Therefore, we have $\nu_{i,j,b}\in\overline{\mathscr{P}}_{G}(\mathcal{X}_{m})$.
As we can decompose  $\Law(v_{i,j})$ as
  \[
    \Law(v_{i,j})=\sum_{b\in\{-1,0,1\}}\mathbb{P}(\chi_{j}=b)\nu_{i,j,b},
  \]
  and $\Law(v_{i,j})\in\overline{\mathscr{P}}_{G}^{\mathrm{e}}(\mathcal{X}_{m})$, it follows that
  $\nu_{i,j,b}=\Law(v_{i,j})$ for all $b$.
  Thus if $X \sim \mathrm{Uniform}(\Lambda)$ is independent of all else, we have, by \cref{jul602}, that
  \begin{equation}\label{eq:expectationsagree}
  \begin{aligned}
    \mathbb{E}[v_{1,j}(X)-v_{2,j}(X)\mid\chi_{j}=b] &= \mathbb{E}[v_{1,j}(X)\mid\chi_{j}=b]-\mathbb{E}[v_{2,j}(X)\mid\chi_{j}=b]
                                                    \\
                                                    &=\mathbb{E}v_{1,j}(X)-\mathbb{E}v_{2,j}(X)=0. 
  \end{aligned}
  \end{equation}
On the other hand, if $b\in\{\pm1\}$ and $\mathbb{P}(\chi_{j}=b)>0$,
  then %
  \[
  \mathbb{E}[v_{1,j}(X)-v_{2,j}(X)\mid\chi_{j}=b]\ne 0,
  \]
  contradicting \cref{eq:expectationsagree}. Therefore, $\mathbb{P}(\chi_{j}=0)=1$, and thus
  $\mathbf{v}_{1}=\mathbf{v}_{2}$ almost surely.
\end{proof}
We have the following corollary.
\begin{cor}
  \label{cor:uniqueness}
  Let $m\in(0,1)$, $G=L\Zm$ with $L>0$ or $G=\R$. %
  Suppose that $\nu_{1},\nu_{2}\in\overline{\mathscr{P}}_{G}(\mathcal{X}_{m}^{N})$, with
  some $N\in\Nm$, 
  and let $\nu_{i,j}$ be the marginal law of~$\nu_{i}$ on the~$j$th
  component, and  $v_{i,j}\sim\nu_{i,j}$.
  Suppose that  for each $i=1,2$,  $j=1,\dots,N$, we have~$\nu_{i,j}\in\overline{\mathscr{P}}_{G}^{\mathrm{e}}(\mathcal{X}_{m})$,
  and both \cref{jul604} and \cref{jul602} hold.
 Then $\nu_{1}=\nu_{2}$.
\end{cor}
\begin{proof} %
  By \cref{prop:couple}, we have a coupling $\nu$ $\in\overline{\mathscr{P}}_{G}(\mathcal{X}_{m}^{2N})$
  of $\nu_{1}$ and $\nu_{2}$. Then \cref{prop:onlycouplingisthesame}
  applies, so $\nu_{1}=\Law(\mathbf{v}_{1})=\Law(\mathbf{v}_{2})=\nu_{2}$.
\end{proof}
Using \cref{cor:uniqueness}, we prove \cref{thm:maintheorem-classification}.
\begin{proof}[Proof of \cref{thm:maintheorem-classification}.] We first prove the existence claim. 
  Fix $m \in [1/2, 1)$, $N \in \mathbb{N}$, and $\mathbf{a} \in \R^N$.
  By \cref{prop:boundedtightness} applied with~$N=1$ and $u(0,\cdot)\equiv 0$, 
  Prokhorov's theorem, and \cref{prop:limitsareergodic}, there exists $\nu_{0}\in\overline{\mathscr{P}}_{\mathbb{R}}^{\mathrm{e}}(\mathcal{X}_{m})$
  satisfying the moment condition \cref{enu:hyp-stat}.
  \cref{prop:limitsareergodic} implies also that~$\nu_{0}\in\overline{\mathscr{P}}_{G}^{\mathrm{e}}(\mathcal{X}_{m})$ for $G=L\Zm$ for any~$L>0$.
    Moreover, if $v_{0}\sim\nu_{0}$, then $\mathbb{E}v_{0}(x)=0$ for all~$x \in \R$.

  Next, take some $a \in \R$.
  By the shear-invariance proved in \cref{prop:shearinvariance}, if $v_{0}\sim\nu_{0}$,
  then 
  \[
  \nu_{a}\coloneqq\Law(v_{0}+a)\in\overline{\mathscr{P}}_{G}^{\mathrm{e}}(\mathcal{X}_{m})
  \]
  as well.
  By \cref{prop:couple} (applied inductively), for
  any $N\in\mathbb{N}$ and $\mathbf{a}=a_{1},\ldots,a_{N}\in\mathbb{R}$,
  there is a coupling $\nu_{\mathbf{a}}\in\overline{\mathscr{P}}_{{G}}(\mathcal{X}_{m}^{N})$
  of $\nu_{a_{1}},\ldots,\nu_{a_{N}}$,  both 
  for $G=\Rm$ and~$G=L\Zm$ for all $L>0$. 
 We claim that $\nu_{\mathbf{a}}$ are   in~$\overline{\mathscr{P}}_{G}^{\mathrm{e}}(\mathcal{X}_{m}^{N})$.   
  Suppose that
  \[
  \nu_{\mathbf{a}}=(1-q)\mu_{0} + q\mu_{1},
  \] 
  for some $q\in(0,1)$ and $\mu_{0},\mu_{1}\in\overline{\mathscr{P}}_{G}(\mathcal{X}_{m}^{N})$.
  For each $j=1,\ldots,N$, the extremality of $\nu_{a_{j}}$ means that the $j$th marginals of $\mu_{0}$ and $\mu_{1}$ must both be $\nu_{a_{j}}$.
  Then \cref{cor:uniqueness} implies that $\mu_0 = \mu_1$.
  So indeed $\nu_{\mathbf{a}}\in \overline{\mathscr{P}}_{G}^{\mathrm{e}}(\mathcal{X}_{m}^{N})$.
We have thus shown existence in \cref{thm:maintheorem-classification}.
  
  For uniqueness, fix $G=L\Zm$ or $G=\R$ and a fundamental domain $\Lambda_G$ for $G$.
  Let $X \sim \mathrm{Uniform}(\Lambda_G)$ be independent of all else.
  Suppose that $\nu \in \overline{\mathscr{P}}_{G}^{\mathrm{e}}(\mathcal{X}_{m}^{N})$ satisfies $\E \mathbf{v}(X) = \mathbf{a}$ and $\E |\mathbf{v}(X)|^2 < \infty$ for~$\mathbf{v} \sim \nu$.
  We will show that $\nu=\nu_{\mathbf{a}}$. 
  Using \cref{prop:couple}, let $\tilde{\nu} \in \overline{\mathscr{P}}_{G}(\mathcal{X}_{m}^{N + 1})$ couple $\nu$ and $\nu_{a_1}$, and
  take $\tilde{v} \sim \tilde{\nu}$.
  By \cref{thm:ordering} with hypothesis \cref{enu:hyp-stat}, $\chi = \mathrm{sgn}({\tilde v_{N+1}-\tilde{v}_{1}})$ is independent of $x$ almost surely.
  Suppose for the sake of contradiction that $\mathbb{P}[\chi \neq 0] > 0$.
  Since 
  \[
  \E \tilde{v}_{N+1}(X) = a_1 = \E \tilde{v}_1(X),
  \]
  this would imply that 
  \[
  q: = \mathbb{P}[\chi \geq 0] \in (0, 1).
  \]
Hence, we may define the laws $\tilde{\mu}_1$ and $\tilde{\mu}_2$ as $\tilde{\nu}$ conditioned on the $G$-invariant 
events $\{\chi = -1\}$ and~$\{\chi \geq 0\}$, respectively, so that 
\[
\tilde{\nu} = (1 - q) \tilde{\mu}_1 + q\tilde{\mu}_2,
\]
and if $(w_{i,1}, \ldots,w_{i,N},w_{i,N+1}) \sim  {\tilde\mu_i}$, then $w_{1,1}\ge w_{1,N+1}$ and $w_{2,1}<w_{2,N+1}$ almost surely.
In addition, as~$\nu_{a_1}$ is extremal, we know that both $w_{1,N+1}$ and $w_{2,N+1}$ have the law $\nu_{a_1}$. It follows that
\begin{equation}\label{jul614}
\E w_{1,1}(X) \ge \E w_{1,N+1}=\E w_{2,N+1}> \E w_{2,1}(X).
\end{equation}
 By \cref{prop:comparisonprinciple}, $\tilde{\mu}_i \in \overline{\mathscr{P}}_G(\mathcal{X}_m^{N + 1})$ for  $i =1, 2$.
  Now let $\mu_i$ denote the marginal of $\tilde{\mu}_i$ on the first $N$ components.
  Then $\mu_i \in \overline{\mathscr{P}}_G(\mathcal{X}_m^N)$ and $\nu = (1 - q) \mu_0 + q \mu_1$.
  Since $\nu$ is extremal, we must have $\mu_0 = \mu_1$.
  However, \cref{jul614} shows that the first marginals of $\mu_0$ and $\mu_1$ have different means, a contradiction.
  We conclude that $\chi = 0$ almost surely, so the marginal of $\nu$ on its first component is~$\nu_{a_1}$.
  By an identical argument, all of the one-component marginals of $\nu$ and $\nu_{\mathbf{a}}$ agree.
  Using \cref{cor:uniqueness}, we conclude that $\nu = \nu_{\mathbf{a}}$.
  Thus $\nu_{\mathbf{a}}$ is the unique element of $\overline{\mathscr{P}}_{G}^{\mathrm{e}}(\mathcal{X}_{m}^{N})$ with mean $\mathbf{a}$ and finite variance.

  We now establish properties \cref{enu:ordered}--\cref{enu:decomposesamemean} for $\nu_{\mathbf{a}}$.
The first four properties have been essentially proved already:   \cref{enu:ordered}
  follows from the ordering proved in \cref{thm:ordering},  \cref{enu:shearinvariant} follows from
  the shear-invariance proved in  \cref{prop:shearinvariance}
  and the uniqueness claim, \cref{enu:alsoinvariant} follows from the uniqueness
  statement, and \cref{enu:smoothtempereddistn} follows from \cref{lem:mildisclassical}, \cref{lem:mildregularitypoly}, and
  the smoothness and sub-polynomial growth at infinity of~$\psi$.

  To prove \cref{enu:decomposesamemean},
  suppose that there exists a $q\in(0,1)$ and $\mu_{0},\mu_{1}\in\mathscr{P}_{G}(\mathcal{X}_{m}^{N})$
  so that 
  \begin{equation}\label{jul610}
  \nu_{\mathbf{a}}=(1-q)\mu_{0}+q\mu_{1}.
  \end{equation}
  We claim that for
  $i=0,1$, the family $\{P_{t}^{*}\mu_{i}\}_{t\ge0}$ is tight
  with respect to the topology of $\mathcal{X}_{m}^{N}$. Indeed,
 note that 
 \[
 \{P_{t}^{*}((1-q)\mu_{0}+q\mu_{1})\}_{t\ge0}=\{P_{t}^{*}\nu_{{\mathbf a}}\}_{t\ge0}=\{\nu_{{\mathbf a}}\}
 \]
  is tight with respect to the topology of $\mathcal{X}_{m}^{N}$.
Hence, for each $\eps>0$ there is a compact set $\mathcal{K}\subset\mathcal{X}_{m}^N$
  so that for all $t\ge0$, we have
  \begin{align*}
    (1-q)\eps&>P_{t}^{*}((1-q)\mu_{0}+q\mu_{1})(\mathcal{X}_{m}^N\setminus\mathcal{K})  =(1-q)P_{t}^{*}\mu_{0}(\mathcal{X}_{m}^N\setminus\mathcal{K})+qP_{t}^{*}\mu_{1}(\mathcal{X}_{m}^N\setminus\mathcal{K}) \\
                                                                                        & >(1-q)P_{t}^{*}\mu_{0}(\mathcal{X}_{m}^N\setminus\mathcal{K}),
  \end{align*}
  which means that $P_{t}^{*}\mu_{0}(\mathcal{X}_{m}^N\setminus\mathcal{K})<\eps$
  for all $t\ge0$, so $\{P_{t}^{*}\mu_{0}\}_{t\ge0}$ is tight.
  By an identical argument, so is $\{P_{t}^{*}\mu_{1}\}_{t\ge0}$. Therefore, a subsequential limit
  \[
    \overline{\mu}_{i}\coloneqq\lim_{T_k\to\infty}\frac{1}{T_k}\int_{0}^{T_k}P_{t}^{*}\mu_{i}\,\dif t
  \]
  exists (in the sense of weak convergence of measures on $\mathcal{X}_{m}^N$)
  for each $i$, and by \cref{prop:krylovbogoliubov}, we know that~$\overline{\mu}_{i}\in\overline{\mathscr{P}}_{G}(\mathcal{X}_{m}^N)$.
  In addition, it follows from \cref{jul610} that
  \[
  (1-q)\overline{\mu}_{0}+q\overline{\mu}_{1}=\nu_{\mathbf{a}}.
  \]
As  $\nu_{\mathbf{a}}\in\overline{\mathscr{P}}_{G}^{\mathrm{e}}(\mathcal{X}_{1/2}^{N})$, we deduce that
$\overline{\mu}_{0}=\overline{\mu}_{1}=\nu_{\mathbf{a}}$.
Recall that by   \cref{prop:L1omegaconservation} the Burgers dynamics \cref{eq:uPDE-many} preserves the expectation
  of a space-stationary solution. We conclude that if~$\mathbf{v}_{i}\sim\mu_{i}$
  and $\overline{\mathbf{v}}_{i}\sim\overline{\mu}_{i}=\nu_{\mathbf{a}}$,
  then $\mathbb{E}\mathbf{v}_{i}(X)=\mathbb{E}\overline{\mathbf{v}}_{i}(X)=\mathbf{a}$.
\end{proof}
 
\section{Stability of the extremal invariant measures \label{sec:stability}}

In this section we prove \cref{thm:maintheorem-stability}, establishing
convergence to the {extremal} invariant measures~{$\nu_{\mathbf a}$} constructed in \cref{thm:maintheorem-classification}.
As a first step, we upgrade the convergence along subsequences of time-averaged laws of solutions starting at deterministic initial data (implied by \cref{prop:boundedtightness}) to true convergence by identifying the limits using \cref{prop:limitsareergodic,thm:maintheorem-classification}.
\begin{prop}
  \label{prop:convergencerandomstarttime}
  Let $G=L\Zm$, with $L>0$ or $G=\Rm$,   and $\mathbf{v}\in L^{\infty}(\mathbb{R})^{N}$ be
  a deterministic~$G$-invariant function. Let $\Lambda_G$ be a fundamental
  domain for $G$ and let
  \begin{equation}
    \mathbf{a}=\int_{\Lambda_G}\mathbf{v}(x)\,\dif\lambda_{G}(x).\label{eq:adef-1}
  \end{equation}
  Let $\delta_{\mathbf{v}}$ be the measure on $\mathcal{X}_{1/2}^N$
  with a single atom at $\mathbf{v}$. Then
  \begin{equation}
    \lim_{T\to\infty}\frac{1}{T}\int_{1}^{T+1}P_{t}^{*}\delta_{\mathbf{v}}\,\dif t=\nu_{\mathbf{a}}\label{eq:identifylimits}
  \end{equation}
  in the sense of weak convergence of probability measures on $\mathcal{X}_{1/2}^N$.
\end{prop}

\begin{proof}
  By \cref{prop:boundedtightness,prop:krylovbogoliubov}, for any sequence $T_{k}\uparrow\infty$
there is a subsequence $T_{k_{n}}$ and~$\nu\in\overline{\mathscr{P}}_{G}(\mathcal{X}_{1/2}^N)$
  so that
  \[
    \lim_{k\to\infty}\frac{1}{T_{k}}\int_{1}^{T_{k}+1}P_{t}^{*}\delta_{\mathbf{v}}\,\dif t=\nu,
  \]
  in the sense of weak convergence of probability measures on $\mathcal{X}_{1/2}^N$.
  By \cref{prop:L1omegaconservation}, \cref{eq:supu2bdd} and~\cref{eq:adef-1}, if $\mathbf{w}\sim\nu$
  we have
  \[
    \mathbb{E}\left[\int_{\Lambda}\mathbf{w}(x)\,\dif\lambda_{G}(x)\right]=\mathbf{a}.
  \]
  Then, \cref{prop:limitsareergodic} and \cref{thm:maintheorem-classification}
  imply that $\nu=\nu_{\mathbf{a}}$. Since the topology of weak convergence of probability measures on $\mathcal{X}^N_{1/2}$ is metrizable, the uniqueness of 
  the subsequential limit implies~\cref{eq:identifylimits}.
\end{proof}

The next step is to use the $L^{1}$-contraction in \cref{prop:L1omega} to eliminate the time averaging in the statement of \cref{prop:convergencerandomstarttime}.
Let is first introduce some definitions. 
For $G=L\Zm$ or $G=\R$, 
recall the definition of the measure $\lambda_{G}$ %
from the beginning of \cref{subsec:L1contraction}, and
set $\tilde{\lambda}_{G}=\mathbf{1}_{\Lambda_{G}}\lambda_{G}$.
We also define the weight
\[
  \p_{G}(x)=\frac{1}{(\frac{1}{\p_{2}}* \tilde\lambda_{G})(x)},
\]
and the Banach space $\mathcal{Y}_{G}^{N}=L_{\p_{G}}^{1}(\mathbb{R})^{N}$,
equipped with the norm
\[
  \|(f_{1},\ldots,f_{N})\|_{\mathcal{Y}_{G}^{N}}=\sum_{i=1}^{N}\|f_{i}\|_{L_{\p_{G}}^{1}(\mathbb{R})}.
\]
We note that $\mathcal{X}_{1/2}^{N}$ is continuously included in
$\mathcal{Y}_{G}^{N}$ for any $G$, since (as it is sufficient to
check $N=1$), if we fix $\ell\in(1/2, 1)$, then
\[
  \|f\|_{\mathcal{Y}_{G}^{1}}=\int_{\mathbb{R}}|f(x)|\left(\frac{1}{\p_{2}}*\lambda_{G}\right)(x)\,\dif x\le\|f\|_{\mathcal{C}_{\p_{\ell}}(\mathbb{R})}\int_{\mathbb{R}}\p_{\ell}(x)\left(\frac{1}{\p_{2}}*\tilde{\lambda}_{G}\right)(x)\,\dif x
\]
and the integral is finite, so $\mathcal{C}_{\p_{\ell}}(\mathbb{R})$
embeds continuously into $\mathcal{Y}_{G}^{1}$ and of course $\mathcal{X}_{1/2}$
embeds continuously into $\mathcal{C}_{\p_{\ell}}(\mathbb{R})$ for
any $\ell>1/2$.

Let $d_{G,N}$ denote the Wasserstein-$1$ (or Kantorovich--Rubinstein)
metric between measures on~$\mathcal{Y}_{G}^{N}$.
That is, if $\nu_{1},\nu_{2}$ are probability measures on $\mathcal{Y}_{G}^{N}$,
then
\begin{equation}
  d_{G,N}(\nu_{1},\nu_{2})=\inf_{\mu}\int_{\mathcal{Y}_{G}^{2N}}\|f_{1}-f_{2}\|_{\mathcal{Y}_{G}^{N}}\,\dif\mu(f_{1},f_{2})=\sup_{\|f\|_{\Lip(\mathcal{Y}_{G}^{N})\le1}}\left(\int f\dif\nu_1-\int f\dif\nu_2\right),\label{eq:Wassersteindef}
\end{equation}
with the supremum taken over all couplings $\mu \in \mathscr{P}(\mathcal{Y}_{G}^{2N})$ of $\nu_1$ and $\nu_2$, and $\Lip(\mathcal{Y}_{G}^{N})$ is the Lipschitz seminorm.
The second equality in \cref{eq:Wassersteindef} is the  
 Kantorovich--Rubinstein duality, see e.g.~\cite[(6.3)]{Vil09}.
We will use the fact that for probability measures $\{\mu_{t}\}_{t\in[0,T]}$
and $\nu$, we have
\begin{align}
  d_{G,N}\left(\frac{1}{T}\int_{0}^{T}\mu_{t}\,\dif t,\nu\right) & =\sup_{\|f\|_{\Lip(\mathcal{Y}_{G}^{N})=1}}\left|\frac{1}{T}\int_{0}^{T}\int f\,\dif\mu_{t}\,\dif t-\int f\,\dif\nu\right|\nonumber                                                        \\
                                                                 & \le\sup_{\|f\|_{\Lip(\mathcal{Y}_{G}^{N})=1}}\frac{1}{T}\int_{0}^{T}\left|\int f\,\dif\mu_{t}-\int f\,\dif\nu\right|\,\dif t\le\frac{1}{T}\int_{0}^{T}d_{G,N}(\mu_{t},\nu)\,\dif t.\label{eq:W1mixture}
\end{align}
Indeed, the metric $d_{G,N}$ comes from a norm on the space of all
finite signed measures \cite{KR58}, from which \cref{eq:W1mixture} follows immediately.

It is well-known (see e.g. \cite[Theorem~4.1]{Vil09})  that if $d_{G,N}(\nu_{1},\nu_{2})$ is finite, then
the infimum in the second expression in \cref{eq:Wassersteindef} is
actually achieved by some coupling $\mu$. In the case of translation-invariant
measures, we will need the stronger statement that the coupling can
be chosen to also be translation-invariant. We record this in the
following lemma, which is an application of the results in \cite{Moa16}.
\begin{lem}[\cite{Moa16}]
  \label{lem:translationinvariantcouplingachieved}Suppose $m\in(0,1)$,
  $G=L\Zm$ with some $L>0$ or $G=\R$, and $\mu_{1},\mu_{2}\in\mathscr{P}_{G}(\mathcal{X}_{m}^{N})$.
  Then there is a coupling $\mu\in\mathscr{P}_G(\mathcal{X}_{m}^{2N})$
  of $\mu_{1}$ and $\mu_{2}$ so that
  \begin{equation}
    \int_{\mathcal{Y}_{G}^{2N}}\|f_{1}-f_{2}\|_{\mathcal{Y}_{G}^{N}}\,\dif\mu(f_{1},f_{2})=d_{G,N}(\mu_{1},\mu_{2}).\label{eq:couplingachieved}
  \end{equation}
\end{lem}

\begin{proof}
  If $G=L\mathbb{Z}$ for some $L>0$, then this is immediate from \cite[Corollary 1]{Moa16}.
  If $G=\mathbb{R}$, then we note that by \cite[Theorem 1.6]{Moa16},
  there exists a measure $\mu$ satisfying \cref{eq:couplingachieved}
  that is invariant under spatial translations by elements of $\mathbb{Q}$.
  We can extend this to elements of $\mathbb{R}$ by continuity as in
  the proof of \cite[Corollary 1]{Moa16}.
\end{proof}
The next observation is that the Burgers
dynamics \cref{eq:uPDE-many} is a contraction in the $d_{G,N}$
metric, which is a simple consequence of the $L^{1}$-contraction
in the probability space  in \cref{prop:L1omega}. %
\begin{prop}
  \label{prop:Ptdecreasesnorm}If $m\in(0,1)$, $N\in\mathbb{N}$, $G=L\Zm$ with $L>0$ or $G=\R$, 
  and $\mu_{1},\mu_{2}\in\mathscr{P}_G(\mathcal{X}_{m}^{N})$,
  then for all $t\ge0$,
  \[
    d_{G,N}(P_{t}^{*}\mu_{1},P_{t}^{*}\mu_{2})\le d_{G,N}(\mu_{1},\mu_{2}).
  \]
\end{prop}

\begin{proof}
  Let $\mu\in\mathscr{P}_G(\mathcal{X}_{m}^{2N})$ be a coupling of
  $\mu_{1}$ and $\mu_{2}$ so that
  \[
    \int_{\mathcal{Y}_G^{2N}}\|\mathbf{v}_{1}-\mathbf{v}_{2}\|_{\mathcal{Y}_G^N}\,\dif\mu(\mathbf{v}_{1},\mathbf{v}_{2})=d_{G,N}(\mu_{1},\mu_{2}),
  \]
that exists by \cref{lem:translationinvariantcouplingachieved}.
  \cref{prop:L1omega} %
  then implies that if
  \[
    \mathbf{u}=(\mathbf{u}_{1},\mathbf{u}_{2})=((u_{1,1},\ldots,u_{1,N}),(u_{2,1},\ldots,u_{2,N}))
  \]
  solves \cref{eq:uPDE-many} with initial condition $\mathbf{u}(0,\cdot)\sim\mu$
  (independent of the noise), then for any $t \ge 0$, $y \in \R$, and $j=1,\ldots,N$, we have
  \begin{equation}
    \mathbb{E}\left(\int_{y+\Lambda_{G}}|u_{1,j}(t,x)-u_{2,j}(t,x)|\,\dif\lambda_{G}(x-y)\right)\le\mathbb{E}\left(\int_{y+\Lambda_{G}}|u_{1,j}(0,x)-u_{2,j}(0,x)|\,\dif\lambda_{G}(x-y)\right).\label{eq:L1contraction-apply}
  \end{equation}
  We note that for any $f:\mathbb{R}\to\mathbb{R}$, we have
  \begin{equation}
    \label{eq:Y_G-formula}
    \|f\|_{\mathcal{Y}_{G}}  =\int_{\mathbb{R}}\frac{|f(y)|}{\p_{G}(y)}\,\dif y  =\int_{\mathbb{R}}\left(\frac{1}{\p_{2}}*\tilde{\lambda}_{G}\right)(y)|f(y)|\,\dif y  =\int_{\mathbb{R}}\frac{1}{\p_{2}(y)}\int_{y+\Lambda_{G}}|f(x)|\,\dif\lambda_{G}(x-y)\,\dif y.
  \end{equation}
  Therefore, integrating \cref{eq:L1contraction-apply} in $y$ against
  $1/\p_{2}$ yields

  \[
    \mathbb{E}\|u_{1,j}(t,\cdot)-u_{2,j}(t,\cdot)\|_{\mathcal{Y}_{G}}\le\mathbb{E}\|u_{1,j}(0,\cdot)-u_{2,j}(0,\cdot)\|_{\mathcal{Y}_{G}},
  \]
  and adding up the components we obtain
  \[
    d_{G,N}(P_{t}^{*}\mu_{1},P_{t}^{*}\mu_{2})\le\mathbb{E}\|\mathbf{u}_{1}(t,\cdot)-\mathbf{u}_{2}(t,\cdot)\|_{\mathcal{Y}_{G}^{N}}\le\mathbb{E}\|\mathbf{u}_{1}(0,\cdot)-\mathbf{u}_{2}(0,\cdot)\|_{\mathcal{Y}_{G}^{N}}=d_{G,N}(\mu_{1},\mu_{2}),
  \]
  as claimed.
\end{proof}

We now show that if two solutions with the same periodic initial conditions are
started at different times, then 
their laws become close in the Wasserstein distance as they evolve.
\begin{prop}
  \label{prop:fixedtimediffs}
  Let $N\in\mathbb{N}$ and $G=L\Zm$ with $L>0$ or $G=\Rm$.     
  Given a deterministic $G$-invariant function~$\mathbf{v}\in L^{\infty}(\mathbb{R})^{N}$,
  let $\delta_{\mathbf{v}}$ be a delta measure on $\mathbf{v}$. Then
  for any fixed $s\ge0$, we have
  \begin{equation}
    \lim_{t\to\infty}d_{G,N}(P_{t}^{*}\delta_{\mathbf{v}},P_{t+s}^{*}\delta_{\mathbf{v}})=0.\label{eq:constantdifferencetozero}
  \end{equation}
\end{prop}

\begin{proof}
  Let $\mathbf{u}=(\mathbf{u}_{1},\mathbf{u}_{2})$ solve \cref{eq:uPDE-many}
  with initial condition
  \[
    \mathbf{u}(0,\cdot)\sim(P_{s}^{*}\delta_{\mathbf{v}})\otimes\delta_{\mathbf{v}}.
  \]
  For each $T\in(0,\infty)$, let $S_{T}\sim\Uniform([0,T])$ be independent
  of everything else. As in the proof of \cref{prop:boundedtightness},
  there is a sequence $T_{k}\uparrow\infty$ and a measure $\omega_{*}\in\overline{\mathscr{P}}_{G}(\mathcal{X}_{1/2}^{2N})$
  so that
  \[
  \Law(\mathbf{u}(S_{T},\cdot))\to\omega_{*}
  \]
  weakly. By \cref{prop:convergencerandomstarttime},
  each of the $2N$ marginals of $\omega_{*}$ is extremal, i.e. an
  element of $\overline{\mathscr{P}}_{G}^{\mathrm{e}}(\mathcal{X}_{1/2})$.
  By \cref{prop:onlycouplingisthesame}, this means that if $\mathbf{v}=(\mathbf{v}_{1},\mathbf{v}_{2})\sim\omega_{*}$
  then $\mathbf{v}_{1}=\mathbf{v}_{2}$ almost surely since
  \[
    \mathbb{E}\left[\int_{\Lambda_{G}}\mathbf{v}_{1}(x)\,\dif\lambda_{G}(x)\right]=\mathbb{E}\left[\int_{\Lambda_{G}}\mathbf{v}_{2}(x)\,\dif\lambda_{G}(x)\right] = \mathbf{a},
  \]
  as %
  the expectation is preserved in the limit). In particular,
 the difference $\mathbf{u}_{1}(S_{T_{k}},\cdot)-\mathbf{u}_{2}(S_{T_{k}},\cdot)$
  converges to zero in distribution with respect to the topology of $\mathcal{X}_{1/2}^{N}$,
  hence with respect to the topology of $L_{\p_{G}}^{1}(\mathbb{R})$) as $k\to\infty$.
  Since the limit is constant, it also converges in probability.
  Furthermore, \cref{eq:supu2bdd} provides uniform integrability, so in fact $\mathbf{u}_{1}(S_{T_{k}},\cdot)-\mathbf{u}_{2}(S_{T_{k}},\cdot) \to 0$ in mean.
  Therefore, we have
  \begin{align*}
    0 & =\lim_{k\to\infty}\mathbb{E}\|\mathbf{u}_{1}(S_{T_{k}},\cdot)-\mathbf{u}_{2}(S_{T_{k}},\cdot)\|_{L_{\p_{1,G}}^{1}(\mathbb{R})^{N}}=\lim_{k\to\infty}\frac{1}{T_{k}}\int_{0}^{T_{k}}\mathbb{E}\|\mathbf{u}_{1}(t,\cdot)-\mathbf{u}_{2}(t,\cdot)\|_{L_{\p_{1,G}}^{1}(\mathbb{R})^{N}}\,\dif t \\
      & \ge\lim_{k\to\infty}\frac{1}{T_{k}}\int_{0}^{T_{k}}d_{G,N}(\Law(\mathbf{u}_{1}(t,\cdot)),\Law(\mathbf{u}_{2}(t,\cdot)))\,\dif t=\lim_{k\to\infty}\frac{1}{T_{k}}\int_{0}^{T_{k}}d_{G,N}(P_{s+t}^{*}\delta_{\mathbf{v}},P_{t}^{*}\delta_{\mathbf{v}})\,\dif t                                \\
      & \ge\lim_{k\to\infty}d_{G,N}(P_{T_{k}+s}^{*}\delta_{\mathbf{v}},P_{T_{k}}^{*}\delta_{\mathbf{v}}),
  \end{align*}
  with the last inequality by \cref{prop:Ptdecreasesnorm}. Using \cref{prop:Ptdecreasesnorm}
  again proves \cref{eq:constantdifferencetozero}, since if a subsequence
  of a nonnegative decreasing sequence converges to zero, then so does
  the whole sequence.
\end{proof}
The following proposition establishes \cref{thm:maintheorem-stability} for periodic
initial conditions, or, in the notation of \cref{prop:stability-criterion},
when $\mathbf{v}_\mathrm{int} =\mathbf{v}_\mathrm{z}\equiv 0$,
\begin{prop}
  \label{prop:periodicconverge}Let $N\in\mathbb{N}$ and $G=L\Zm$ with $L>0$ or $G=\R$. 
  Given a deterministic~$G$-invariant~$\mathbf{v}\in L^{\infty}(\mathbb{R})^{N}$, let $\delta_{\mathbf{v}}$
  be the delta measure on $\mathbf{v}$, and set
  \[
  \mathbf{a}=\int_{\Lambda_{G}}\mathbf{v}(x)\,\dif\lambda_{G}(x).
  \]
  With $\nu_{\mathbf{a}}$ defined as in \cref{thm:maintheorem-classification},
  we have
  \begin{equation}
    \lim_{t\to\infty}d_{G,N}(P_{t}^{*}\delta_{\mathbf{v}},\nu_{\mathbf{a}})=0\label{eq:convergencefromconstantinitialconditions}
  \end{equation}
  and
  \begin{equation}
    \lim_{t\to\infty}P_{t}^{*}\delta_{\mathbf{v}}=\nu_{\mathbf{a}},\label{eq:weakconvergenceperiodic}
  \end{equation}
  weakly with respect to the topology of $\mathcal{X}_{1/2}^{N}$.
\end{prop}

\begin{proof}
  Let
  \[
    \mu_{t,T}=\frac{1}{T}\int_{t}^{t+T}P_{s}^{*}\delta_{\mathbf{v}}\,\dif s.
  \]
  By \cref{prop:convergencerandomstarttime}, $\mu_{0,T}$ converges to
  $\nu_{\mathbf{a}}$ weakly with respect to the topology of $\mathcal{X}_{1/2}^{N}$, hence with respect to the topology of 
  $\mathcal{Y}^N_G$, as $T \to \infty$. By \cref{lem:L2bound}, if $\mathbf{v}\sim\mu_{0,T}$,
  then $\E\|\mathbf{v}\|_{\mathcal{Y}^N_G}^2$ is bounded uniformly
  in $T$, so in fact we have
  \begin{equation}
    \lim_{T\to\infty}d_{G,N}(\mu_{0,T},\nu_{\mathbf{a}})=0.\label{eq:mu0TconvindN}
  \end{equation}
  Now $P_{t}^{*}\mu_{0,T}=\mu_{t,T}$ and $P_{t}^{*}\nu_{\mathbf{a}}=\nu_{\mathbf{a}}$, so \cref{prop:Ptdecreasesnorm} implies that $d_{G,N}(\mu_{t,T},\nu_{\mathbf{a}})$ is decreasing in $t$.
  Hence we can upgrade \cref{eq:mu0TconvindN} to
  \begin{equation}
    \adjustlimits\lim_{T\to\infty}\sup_{t\ge0}d_{G,N}(\mu_{t,T},\nu_{\mathbf{a}})=0.\label{eq:liminT}
  \end{equation}
  Now fix $T > 0$ and consider
  \begin{equation*}
    \lim_{t\to\infty}\int_{t}^{t+T}d_{G,N}(P_{t}^{*}\delta_{\mathbf{v}},P_{s}^{*}\delta_{\mathbf{v}})\,\dif s=\lim_{t\to\infty}\int_{0}^{T} d_{G,N}(P_{t}^{*}\delta_{\mathbf{v}},P_{t+s}^{*}\delta_{\mathbf{v}})\,\dif s.
  \end{equation*}
  By \cref{prop:Ptdecreasesnorm}, the integrand is at most $d_{G,N}(\delta_{\mathbf{v}},P_{s}^{*}\delta_{\mathbf{v}})$, which is locally integrable in $s$ by \cref{prop:comparisonprinciple} and \cref{prop:H1bound}.
  Furthermore, \cref{prop:fixedtimediffs} shows that the integrand converges to $0$ pointwise in $s$.
  Thus by dominated convergence, we have
  \begin{equation}
    \label{eq:fixedTaveragelimit}
    \lim_{t\to\infty}\int_{t}^{t+T}d_{G,N}(P_{t}^{*}\delta_{\mathbf{v}},P_{s}^{*}\delta_{\mathbf{v}})\,\dif s = 0.
  \end{equation}
  Now the triangle inequality and \cref{eq:W1mixture} yield
  \begin{align*}
    d_{G,N}(P_{t}^{*}\delta_{\mathbf{v}},\nu_{\mathbf{a}}) & \le d_{G,N}(P_{t}^{*}\delta_{\mathbf{v}},\mu_{t,T})+d_{G,N}(\mu_{t,T},\nu_{\mathbf{a}})                                                     \\
                                                           & \le\frac{1}{T}\int_{t}^{t+T}d_{G,N}(P_{t}^{*}\delta_{\mathbf{v}},P_{s}^{*}\delta_{\mathbf{v}})\,\dif s+d_{G,N}(\mu_{t,T},\nu_{\mathbf{a}}).
  \end{align*}
  Now for any $\eps>0$, \cref{eq:liminT} let us choose $T$ so large
  that the second term is less than $\eps/2$ for any $t$, and then
  \cref{eq:fixedTaveragelimit} lets us choose $t$ so large that the
  first term is less than $\eps/2$. This proves \cref{eq:convergencefromconstantinitialconditions}.

  To prove \cref{eq:weakconvergenceperiodic}, note that by \cref{prop:boundedtightness},
  for every sequence $t_{k}\to\infty$ we have  a subsequence~$t_{k_{n}}\to\infty$
  and a $\nu\in\mathscr{P}(\mathcal{X}_{1/2}^{N})$ so that $P_{t_{k_{n}}}^{*}\delta_{\mathbf{v}}$
  converges weakly in $\mathscr{P}(\mathcal{X}_{1/2}^{N})$ to $\nu$
  as $n\to\infty$. Since the topology of weak convergence is metrizable,
  to complete the proof it suffices to show that~$\nu=\nu_{\mathbf{a}}$.
  Since $P_{t_{k_{n}}}^{*}\delta_{\mathbf{v}}$ converges weakly to
  $\nu$ in $\mathscr{P}(\mathcal{X}_{1/2}^{N})$, it also converges
  weakly to $\nu$ with respect to the topology of $L_{\p_{G}}^{1}(\mathbb{R})$
  due to the continuous inclusion $\mathcal{X}_{1/2}^{N}\subset L_{\p_{G}}^{1}(\mathbb{R})$.
  On the other hand, \cref{eq:convergencefromconstantinitialconditions}
  implies that $P_{t_{k_{n}}}^{*}\delta_{\mathbf{v}}$ converges to
  $\nu_{\mathbf{a}}$ weakly with respect to the topology of $L_{\p_{G}}^{1}(\mathbb{R})$.
  This means that $\nu=\nu_{\mathbf{a}}$, so the proof is complete.
\end{proof}

We now consider initial conditions satisfying \cref{def:basinsofattraction}
to prove \cref{thm:maintheorem-stability}. Initial conditions satisfying \cref{def:basinsofattraction}
can be sandwiched between periodic initial conditions whose means are very close to each other. Solutions started with such periodic initial conditions converge to the stationary solutions with the corresponding means. Then the solutions started from the aperiodic initial conditions are stuck in the middle by the comparison principle, and must therefore converge to the stationary solution with the appropriate mean.
\begin{proof}[Proof of \cref{thm:maintheorem-stability}.]
\cref{prop:shearinvariance} implies that may assume  
  that~$\mathbf{a}=(0,\ldots,0)$.
  Consider the solution~$\underline{\mathbf{u}}=(\mathbf{u}_{0},\mathbf{u}) \in \mathcal{Z}_{1/2}^{2N}$
 to \cref{eq:uPDE-many} with initial condition
  \[
    \underline{\mathbf{u}}(0,\cdot)=((0,\ldots,0),\mathbf{v}).
  \]
  By \cref{prop:boundedtightness}, for every sequence $t_{k}\uparrow\infty$,
  there is a subsequence $t_{k_{n}}\uparrow\infty$ and a limit 
  measure~$\nu\in\mathscr{P}(\mathcal{X}_{1/2}^{2N})$
  so that
  \begin{equation}
    \Law(\underline{\mathbf{u}}(t_{k_{n}},\cdot))\to\nu\label{eq:convinlaw}
  \end{equation}
  weakly in $\mathscr{P}(\mathcal{X}_{1/2}^{2N})$ as $k\to\infty$.
  Since the topology of weak convergence is metrizable, 
  it suffices to prove that the law of the marginal
  of $\nu$ on the last $N$ components is $\nu_{\mathbf{0}}$.

  Let $\eps>0$ and pick $L,\mathbf{v}_{-},\mathbf{v}_{+}$ as in \cref{def:basinsofattraction}.
  Let $\overline{\mathbf{u}}=(\mathbf{u}_{-},\mathbf{u}_{+},\mathbf{u}_{0},\mathbf{u})\in \mathcal{Z}_{1/2}^{4N}$
  solve \cref{eq:uPDE-many} with initial condition $\overline{\mathbf{u}}(0,\cdot)=(\mathbf{v}_{-},\mathbf{v}_{+},\mathbf{0},\mathbf{v})$.
  By the comparison principle \cref{prop:comparisonprinciple}, we have
  \begin{align}
    \mathbf{u}_{-}(t,\cdot) & \preceq\mathbf{u}_{0}(t,\cdot)\preceq\mathbf{u}_{+}(t,\cdot)&&\text{and}&
      \mathbf{u}_{-}(t,\cdot) & \preceq\mathbf{u}(t,\cdot)\preceq\mathbf{u}_{+}(t,\cdot)
    \label{eq:sandwichperiodic}
  \end{align}
  almost surely for all $t\ge0$. Here, $\preceq$ refers to the partial
  order defined in \cref{eq:partialorder}. Define
  \[
    \mathbf{a}_{\pm}=(a_{\pm,1},\ldots,a_{\pm,N})=\frac{1}{L}\int_{0}^{L}\mathbf{v}_{\pm}(x)\,\dif x,
  \]
  so that
  \begin{equation}
    \sum_{j=1}^{N}(a_{+,j}-a_{-,j})<2N\eps\label{eq:asclose}
  \end{equation}
  by \cref{eq:basinofattractioncond}. By \cref{prop:boundedtightness},
  there is a further subsequence $t_{k_{n_{i}}}$ of $t_{k_{n}}$
  and a measure $\overline{\nu}\in\mathscr{P}(\mathcal{X}_{1/2}^{4N})$,
  so that
  \begin{equation}
    \lim_{i\to\infty}\Law(\overline{\mathbf{u}}(t_{k_{n_{i}}},\cdot))=\overline{\nu}\label{eq:ubarconvinlaw}
  \end{equation}
  weakly in the topology of $\mathscr{P}(\mathcal{X}_{1/2}^{4N})$.
  By \cref{eq:convinlaw}, the marginal of $\overline{\nu}$ on the last
  $2N$ coordinates is $\nu$. By \cref{prop:periodicconverge}, the marginal
  of $\overline{\nu}$ on the first $3N$ components is $\nu_{\mathbf{a}_{-},\mathbf{a}_{+},\mathbf{a}}$.
  On the other hand, we have by \cref{eq:sandwichperiodic} that for each
  $t\ge0$,
  \begin{equation}
    \|\mathbf{u}(t,\cdot)-\mathbf{u}_{0}(t,\cdot)\|_{\mathcal{Y}_{G}^{N}}\le\|\mathbf{u}_+(t,\cdot)-\mathbf{u}_-(t,\cdot)\|_{\mathcal{Y}^N_G}
    \label{eq:comparetosandwich}
  \end{equation}
  almost surely. By \cref{eq:supu2bdd}, the family $\{\|\mathbf{u}_+(t,\cdot)-\mathbf{u}_-(t,\cdot)\|_{\mathcal{Y}^N_G}\} _{t\ge0}$
  is uniformly integrable, so by the Skorokhod representation theorem
  and \cref{eq:ubarconvinlaw}, if
  \[
    (\mathbf{w}_{-},\mathbf{w}_{+})=((w_{-,1},\ldots,w_{-,N}),(w_{+,1},\ldots,w_{+,N}))\sim\nu_{\mathbf{a}_{-},\mathbf{a}_{+}},
  \]
  then \cref{enu:ordered} of \cref{thm:maintheorem-classification}, \cref{eq:Y_G-formula}, and \cref{eq:asclose} imply
  \begin{equation}
    \label{eq:limofexps}
    \begin{split}
      \lim_{t\to\infty}  \mathbb{E}\|\mathbf{u}_+(t,\cdot)-\mathbf{u}_-(t,\cdot)\|_{\mathcal{Y}^N_G} &=\mathbb{E}\|\mathbf{w}_+-\mathbf{w}_-\|_{\mathcal{Y}^N_G}\\
      &=\left(\int_{\mathbb{R}}\frac{\dif x}{\p_2(x)}\right)\sum_{j=1}^{N}(a_{+,j}-a_{-,j}) \le2N\eps\int_{\mathbb{R}}\frac{\dif x}{\p_2(x)}.
    \end{split}
  \end{equation}
  Using Fatou's lemma,
  \cref{eq:comparetosandwich}, and \cref{eq:limofexps}, we have that if
  $(\mathbf{w}_{0},\mathbf{w})\sim\nu$, then
  \begin{equation*}
    \mathbb{E}\|\mathbf{w}_{0}-\mathbf{w}\|_{\mathcal{Y}_{G}^{N}}\le\liminf_{t\to\infty}\mathbb{E}\|\mathbf{u}(t,\cdot)-\mathbf{u}_{0}(t,\cdot)\|_{\mathcal{Y}_{G}^{N}}  \le\lim_{t\to\infty}\mathbb{E}\|\mathbf{u}_{+}(t,\cdot)-\mathbf{u}_{-}(t,\cdot)\|_{\mathcal{Y}_{G}^{N}}\le2N\eps\int_{\mathbb{R}}\frac{\dif x}{\p_{G}(x)}.
  \end{equation*}
  But this holds for all $\eps>0$, so in fact $\mathbb{E}\|\mathbf{w}_{0}-\mathbf{w}\|_{\mathcal{Y}_{G}^{N}}=0$,
  so $\mathbf{w}_{0}=\mathbf{w}$ almost surely. This means that $\Law(\mathbf{w})=\Law(\mathbf{w}_{0})=\nu_{\mathbf{0}}$,
  which is what we needed to show.
\end{proof}

\appendix
\section{Proof of \texorpdfstring{\cref{prop:stability-criterion}}{Proposition~\ref{prop:stability-criterion}}}
\label{appendix:stability-criterion}

We now prove \cref{prop:stability-criterion}. 
The proof is elementary and independent of the rest of the paper.
  Let $\eps>0$. For each $j=1,\ldots,N$, define $w_{j}\in L^{\infty}(\mathbb{R})$
  by
  \begin{align*}
    w_{j}(x) & =\sup_{|y|\ge|x|}|v_{\mathrm{z},j}(y)|,
  \end{align*}
  so $|v_{\mathrm{z},j}(x)|\le w_{j}(x)$ for all $x\in\mathbb{R}$. Then, 
  $w_{j}(x)$ is decreasing in $|x|$, and
  \[
    \lim_{|x|\to\infty}w_{j}(x)=0.
  \]
  Therefore, we can find a $K\in\mathbb{N}$ so large that
  \begin{align*}
    \frac{1}{KL}\max_{j=1}^{N}\|v_{\mathrm{int},j}\|_{L^{1}(\mathbb{R})} & <\eps/2, & \frac{1}{KL}\max_{j=1}^{N}\int_{0}^{KL}w_{j}(x)\,\dif x & <\eps/2.
  \end{align*}
Let us define
  \begin{align}
    v_{\pm,j}(x) & =v_{\mathrm{per},j}(x)\pm w_{j}(x)\pm\sup_{m\in\mathbb{Z}}|v_{\mathrm{int},j}(x+mKL)|\label{eq:vminusplusdefs}
  \end{align}
and $\mathbf{v}_{-}=(v_{-,1},\ldots,v_{-,N})$ and $\mathbf{v}_{+}=(v_{+,1},\ldots,v_{+,N})$,
  so   $\mathbf{v}_{-},\mathbf{v}_{+}\in L^{\infty}(\mathbb{R})^{N}$
  and $\mathbf{v}_{-}\preceq\mathbf{v}\preceq\mathbf{v}_{+}$. 
Then we have
  \begin{align*}
    \frac{1}{KL}\int_{0}^{KL}\mathbf{v}_{-}(x)\,\dif x-\mathbf{a} & =\frac{1}{KL}\int_{0}^{KL}[\mathbf{v}_{-}(x)-\mathbf{v}_{\mathrm{per}}(x)]\,\dif x                                                            \\
                                                                  & =\frac{1}{KL}\int_{0}^{KL}\left[-w_{j}(x)+\inf_{i\in\mathbb{Z}}v_{\mathrm{int},j}(x+iKL)\right]\,\dif x                                       \\
                                                                  & \succeq-\frac{\eps}{2}(1,\ldots,1)-\frac{1}{KL}\max_{j=1}^{N}\int_{0}^{KL}\sum_{i\in\mathbb{Z}}\left|v_{\mathrm{int},j}(x+iKL)\right|\,\dif x \\
                                                                  & =-\frac{\eps}{2}(1,\ldots,1)-\frac{1}{KL}\max_{j=1}^{N}\|v_{\mathrm{int},j}\|_{L^{1}(\mathbb{R})}\succeq-(\eps,\ldots,\eps).
  \end{align*}
  A similar argument shows that 
  \[
  \frac{1}{KL}\int_{0}^{KL}\mathbf{v}_{+}(x)\,\dif x-\mathbf{a}\preceq(\eps,\ldots,\eps).
  \]
  Since this is possible for every $\eps>0$, we have $\mathbf{v}\in\mathscr{B}_{\mathbf{a}}$.
\hfill $\qed$

\section{Weighted spaces}\label{appendix:app-weight}

We now record several useful results on weighted spaces. We begin with some basic lemmas that are used throughout the paper. Then we will prove bounds on the heat kernel on weighted spaces.
\subsection{Basic properties of weighted spaces}\label{sec:append-weighted}

Our first lemma allows us to upgrade convergence from one weight to another.
\begin{prop}\label{prop:boosttheweight}
Let $\mathcal{T}$ be a metric space 
and $w_{1},w_{2},w_{3}$ be weights on $\R$ 
so that 
\[
\lim\limits _{|x|\to\infty}\frac{w_{1}(x)}{w_{2}(x)}=0.
\]
Suppose that $u_n\in\mathcal{C}_{\mathrm{b}}(\mathcal{T};\mathcal{C}_{w_{1}}(\mathbb{R}))$, and $u\in\mathcal{C}_{\mathrm{b}}(\mathcal{T};\mathcal{C}_{w_{3}}(\mathbb{R}))$ satisfy
\begin{equation*}
\lim\limits _{n\to\infty}
\|u_{n}-u\|_{\mathcal{C}_{\mathrm{b}}(\mathcal{T};\mathcal{C}_{w_{3}}(\mathbb{R}))}=0 \quad \textrm{and} \quad \sup\limits _{n\in\mathbb{N}}\|u_{n}\|_{\mathcal{C}_{\mathrm{b}}(\mathcal{T};\mathcal{C}_{w_{1}}(\mathbb{R}))}<\infty.
\end{equation*}
Then $u\in\mathcal{C}_{\mathrm{b}}(\mathcal{T};\mathcal{C}_{w_{2}}(\mathbb{R}))$, and
 $\lim\limits _{n\to\infty}\|u_{n}-u\|_{\mathcal{C}_{\mathrm{b}}(\mathcal{T};\mathcal{C}_{w_{2}}(\mathbb{R}))}=0$ as well.
\end{prop}

\begin{proof} Fix $\eps>0$ and define
\begin{equation*}
K = \sup\limits _{n\in\mathbb{N}}
\|u_{n}\|_{\mathcal{C}_{\mathrm{b}}(\mathcal{T};\mathcal{C}_{w_{1}}(\mathbb{R}))}<\infty.
\end{equation*}
Then choose $M$ so that
\[
\left|\frac{w_{1}(x)}{w_{2}(x)}\right|\le \frac{\eps}{2K}~~~\hbox{ if $|x|\ge M$,}
\]
and  $N$ so large that if $n\ge N$ then
\begin{equation*}
\|u_{n}-u\|_{\mathcal{C}_{\mathrm{b}}(\mathcal{T};\mathcal{C}_{w_{3}}(\mathbb{R}))}
\le\eps\inf\limits _{|x|\le M}\frac{w_{2}(x)}{w_{3}(x)}.
\end{equation*}
Now, for any $n\ge N$, if $|y|\le M$, then for all $t\in\mathcal{T}$ we
  have
  \begin{equation*}
    |(u_{n}-u)(t,y)|\le\eps w_{3}(y)\inf\limits _{x\in[-M,M]}\frac{w_{2}(x)}{w_{3}(x)}\le\eps w_{2}(y)
  \end{equation*}
  while if $|y|\ge M$ we have
  \begin{equation*}
    |(u_{n}-u)(t,y)|\le\eps w_{3}(y)\inf\limits _{x\in[-M,M]}\frac{w_{2}(x)}{w_{3}(x)}\le\eps w_{2}(y).
  \end{equation*}
  Therefore, $\|u_{n}-u\|_{\mathcal{C}_{\mathrm{b}}(\mathcal{T};\mathcal{C}_{w_{2}}(\mathbb{R}))}<\eps$.
  This proves that $\lim\limits _{n\to\infty}\|u_{n}-u\|_{\mathcal{C}_{\mathrm{b}}(\mathcal{T};\mathcal{C}_{w_{2}}(\mathbb{R}))}=0$, as claimed.
\end{proof}

We next establish a form of the Arzel\`{a}--Ascoli theorem in weighted spaces.
\begin{prop}
  \label{prop:arzelaascoli}Suppose that $w_{1},w_{2},w_{3}$ are weights
  so that $\lim\limits _{|x|\to\infty}\frac{w_{1}(x)}{w_{2}(x)}=0$ and fix $\alpha > 0$.
  Then the embedding
  \[
    \mathcal{C}_{w_{1}}(\mathbb{R})\cap\mathcal{C}_{w_{3}}^{\alpha}(\mathbb{R})\hookrightarrow\mathcal{C}_{w_{2}}(\mathbb{R})
  \]
  is compact, where $\mathcal{C}_{w_{1}}(\mathbb{R})\cap\mathcal{C}_{w_{3}}^{\alpha}(\mathbb{R})$
  is equipped with the norm $\|u\|_{\mathcal{C}_{w_{1}}(\mathbb{R})\cap\mathcal{C}_{w_{3}}^{\alpha}(\mathbb{R})}=\|u\|_{\mathcal{C}_{w_{1}}(\mathbb{R})}+\|u\|_{\mathcal{C}_{\p_{w_{3}}}^{\alpha}(\mathbb{R})}$.
\end{prop}

\begin{proof}
  It suffices to show that the unit ball of $\mathcal{C}_{w_{1}}(\mathbb{R})\cap\mathcal{C}_{w_{3}}^{\alpha}(\mathbb{R})$
  is precompact in $\mathcal{C}_{w_{2}}(\mathbb{R})$. Fix a sequence
  $(v_{n})_{n}$ of elements of this unit ball. On any compact subset
  of $\mathbb{R}$, $(v_{n})$ is uniformly bounded and Hölder. Thus
  by Arzelà--Ascoli and diagonalization, there exists a subsequence
  $(v_{n_{k}})_{k}$ which converges locally uniformly to some $v\in\mathcal{C}(\mathbb{R})$.
  As noted in the proof of \cref{lem:weightedhkbound}, this is equivalent
  to convergence in \emph{some} weighted space. Since $(v_{n})$ is
  uniformly bounded in $\mathcal{C}_{w_{1}}(\mathbb{R})$, \cref{prop:boosttheweight}
  implies that $v\in\mathcal{C}_{w_{2}}(\mathbb{R})$ and $\lim\limits _{k\to\infty}v_{n_{k}}=v$
  in $\mathcal{C}_{w_{2}}(\mathbb{R})$.
\end{proof}

Finally, we record a compactness criterion in $\mathcal{X}_m$.
\begin{lem}
  \label{lem:Xmcompactness}If $K\subset\mathcal{X}_{m}$ is such that
  $K$ is compact in the topology of $\mathcal{C}_{\p_{\ell}}(\mathbb{R})$
  for each $\ell>m$, then~$K$ is compact in the topology of $\mathcal{X}_{m}$
  as well.
\end{lem}

\begin{proof}
  Let $(v_{n})_{n}$ be a sequence of elements in $K$.
  By a diagonal argument, there is a subsequence $(v_{n_k})_k$ which converges in the topology of $\mathcal{C}_{\p_{\ell}}(\mathbb{R})$ for all $\ell>m$, and hence in the topology of $\mathcal{X}_{m}$.
\end{proof}

\subsection{Heat kernel bounds in weighted spaces\label{subappendix:hkbounds}}

Here, we prove some weighted estimates for the heat kernel. 
\begin{lem}
  \label{lem:weightedhkbound}Fix a weight $w\in\{(\log\langle\cdot\rangle)^{3/4}\}\cup\{\p_{\ell}\st\ell\in\mathbb{R}\}$,
  $\beta\ge\alpha\ge0$, and $T<\infty$. There is a constant $C=C(w,\alpha,\beta,T)<\infty$
  so that for all $t\in(0,T]$ and $f\in\mathcal{C}_{w}^{\alpha}(\mathbb{R})$
  we have
  \begin{equation}
    \|G_{t}*f\|_{\mathcal{C}_{w}^{\beta}(\mathbb{R})}\le Ct^{-\frac{\beta-\alpha}{2}}\|f\|_{\mathcal{C}_{w}^{\alpha}(\mathbb{R})}\label{eq:Gtf}
  \end{equation}
  In particular,
  \begin{equation}
    \|\partial_{x}G_{t}*f\|_{\mathcal{C}_{w}^{\beta}(\mathbb{R})}\le Ct^{-\frac{\beta-\alpha+1}{2}}\|f\|_{\mathcal{C}_{w}^{\alpha}(\mathbb{R})}.\label{eq:dxGf}
  \end{equation}
  In the case $\alpha=0$, it is only necessary to assume that $f\in L_{w}^{\infty}(\mathbb{R})$
  and the norm $\|f\|_{\mathcal{C}_{w}^{\alpha}(\mathbb{R})}$ can be
  replaced by $\|f\|_{L_{w}^{\infty}(\mathbb{R})}$ on the right-hand
  sides of \cref{eq:Gtf} and \cref{eq:dxGf}.
\end{lem}

The proof of this lemma is word-for-word the same as that of \cite[Lemma 2.8]{HL15}.
There, only exponential weights (since a uniformity statement in the
weight is needed) and continuous functions are considered, but there
is no difference in the treatment given the Gaussian decay of the
heat kernel. The essence of the argument is that only singularity in the heat kernel is at $t=0$, $x=0$, so the part of the heat kernel that is exposed to the growth of $f$ at infinity is smooth, and moreover decays quickly enough not to pose any difficulty for these estimates.
\begin{lem}
  \label{lem:weightedhkcts}
  Fix $m \in \mathbb{R}$.
  If $p \in [1, \infty)$ and $f \in L_{\p_m}^p(\mathbb{R})$, then $G_t \ast f \to f$ in $L_{\p_m}^p(\mathbb{R})$ as $t \downarrow 0$.
  If $f \in L_{\p_m}^\infty(\mathbb{R})$, then $G_t \ast f \overset{\mathrm{w}^*}{\longrightarrow} f$ in $L_{\p_m}^\infty(\mathbb{R})$ as $t \downarrow 0$.
  Finally, if $f \in \mathcal{C}_{\p_{m}}(\mathbb{R})$, then $G_t \ast f \to f$ in $\mathcal{X}_m$ as $t \downarrow 0$.
\end{lem}

\begin{proof}
  First fix $p \in [1, \infty)$ and $f \in L_{\p_m}^p(\mathbb{R})$.
  We provide a simple variant of a standard ``approximation of the identity'' argument to deal with the weighted spaces.
  Using the scaling symmetry of $G$, we can write
  \begin{equation*}
    G_{t}*f(x)-f(x)=\int_{\mathbb{R}}[f(x-\sqrt{t}y)-f(x)]G_{1}(y)\,\dif y,
  \end{equation*}
  so by the triangle inequality,
  \begin{equation}
    \label{eq:Gtphiphi}
    \|G_{t}*f-f\|_{L_{\p_m}^{p}(\mathbb{R})}\le\int_{\mathbb{R}}\|\tau_{\sqrt{t}y}f-f\|_{L_{\p_m}^{p}(\mathbb{R})}G_{1}(y)\,\dif y.
  \end{equation}

  We will use the dominated convergence theorem, so we first establish an integrable majorant.
  Assume $t\in(0,1]$.
  We can easily verify that there exists $C = C(m) <\infty$ such that
  \begin{equation}
    \label{eq:binomial}
    \p_{-m}(a + b) \le C \p_{-m}(a) \p_{|m|}(b)
  \end{equation}
  for all $a, b \in \R$.
  In the sequel, we permit $C$ to change from line to line.
  Then
  \begin{equation*}
    \|\tau_{\sqrt{t}y}f\|_{L_{\p_m}^{p}(\mathbb{R})} = \left(\int_{\mathbb{R}} |f(x)|^p \p_{m}(x+\sqrt{t}y)^{-p}\,\dif x\right)^{\frac 1 p} \le C\p_{|m|}(y)\|f\|_{L_{\p_{m}}^{p}(\mathbb{R})}.
  \end{equation*}
  Since $\p_{|m|}$ is integrable against the Gaussian $G_{1}$, this is a suitable majorant.

  By the dominated convergence theorem, it now suffices to prove pointwise (in $y$) convergence to $0$ as $t\downarrow0$ in \cref{eq:Gtphiphi}.
  We can therefore fix $y \in \R$ and consider $t > 0$ such that $\sqrt{t} y \leq 1$.
  For fixed $\eps>0$, we can find a compactly-supported continuous function $\zeta$ on $\R$ so that $\|\zeta-f\|_{L_{\p_{m}}^{p}(\mathbb{R})} < \eps$.
  Then we have
  \[
    \|\tau_{\sqrt{t}y}f-f\|_{L_{\p_{m}}^{p}(\mathbb{R})}\le\|\tau_{\sqrt{t}y}f-\tau_{\sqrt{t}y}\zeta\|_{L_{\p_{m}}^{p}(\mathbb{R})}+\|\tau_{\sqrt{t}y}\zeta-\zeta\|_{L_{\p_{m}}^{p}(\mathbb{R})}+\|\zeta-f\|_{L_{\p_{m}}^{p}(\mathbb{R})}.
  \]
  By \cref{eq:binomial} and $\sqrt{t} y \leq 1$, the first and third terms are each less than a constant times $\eps$
  and the second term goes to $0$ as $t \to 0$.
  Therefore $G_t \ast f \to f$ in $L_{\p_m}^p(\mathbb{R})$ as $t \downarrow 0$.

  Now suppose $f \in L_{\p_m}^\infty(\mathbb{R})$, and fix $\phi$ in the dual space $L_{\p_{-m}}^{1}(\mathbb{R})$.
  We must show that
  \begin{equation}
    \label{eq:weak-limit}
    \langle\phi,G_{t}*f\rangle\to\langle\phi,f\rangle \quad \textrm{as } t \downarrow 0.
  \end{equation}
  Since $G_{t}$ is symmetric, we have $\langle\phi,G_{t}*f\rangle=\langle G_{t}*\phi,f\rangle$.
  But we have just shown that $G_{t}*\phi \to \phi$ in $L_{\p_{-m}}^1(\mathbb{R})$, so \cref{eq:weak-limit} follows.

  Finally, suppose that $f \in \mathcal{C}_{\p_{m}}(\mathbb{R})$.
  Fix $\eps > 0$ and $x \in \R$.
  We write
  \[
    G_t \ast f(x) - f(x) = \int_{\mathbb{R}}[f(y)-f(x)]G_t(x-y)\,\dif y.
  \]
  Now $|f(y)|\le\|f\|_{\mathcal{C}_{\p_{m}}(\mathbb{R})}\p_{m}(y)$
  and $\p_{m} G_t(x - \cdot)\in L^{1}(\mathbb{R})$. When $t \in (0, 1]$, $G_t$
  is decreasing in $t$ outside a compact set.
  Thus there exists a compact set $K\subset\mathbb{R}$
  containing $x$ so that
  \[
    \int_{\mathbb{R}\setminus K}|f(y)-f(x)|G_t(x-y)\,\dif y<\eps
  \]
  for all $n\ge0$. On $K$, $f$ is uniformly continuous and bounded.
  Thus there exists a $\delta>0$ such that $|f(y)-f(x)|<\eps$ when
  $|y-x|<\delta$. Since $G_t$ has unit mass, this implies
  \[
    \int_{B_{\delta}(x)}|f(y)-f(x)|G_t(x-y)\,\dif y<\eps
  \]
  for all $t \in (0, 1]$. Finally, $G_t(x-y)\to0$ uniformly on $K\setminus B_{\delta}(x)$ as $t \downarrow 0$,
  so there exists $\delta > 0$ such that
  \[
    \int_{K\setminus B_{\delta}(x)}|f(y)-f(x)|G_t(x-y)\,\dif y<\eps
  \]
  for all $t < \delta$. Together, these bounds show that $|G_t \ast f(x) - f(x)|\to0$
  as $n\to0$.

  In fact, the convergence is \emph{locally uniform} in
  $x$. But locally uniform convergence is equivalent to the existence
  of a weight $w$ such that
  \[
    \lim_{n\to\infty}\|G_t \ast f - f\|_{\mathcal{C}_{w}(\mathbb{R})}=0.
  \]
  Combining this with the uniform bound \cref{eq:Gtf} with $\alpha = \beta = 0$ and $w = \p_m$, \cref{prop:boosttheweight}
  implies
  \[
    \lim_{n\to\infty}\|G_t \ast f - f\|_{\mathcal{C}_{\p_{\ell}}(\mathbb{R})}=0
  \]
  for any $\ell>m$.
  That is, $G_t \ast f \to f$ in $\mathcal{X}_m$.
\end{proof}
Next, we show an estimate with super-exponential weights. The restriction
$\beta<2$ is needed in the following lemma simply because the heat
equation is not well-posed for initial conditions growing like $\exp(cx^{2})$
with $c>0$.
\begin{lem}
  \label{lem:hk-superexp}Fix $\beta\in[3/2,2)$ and define, for $\lambda\ge0$,
  $\q_{\lambda}(x)=\e^{\lambda\langle x\rangle^{\beta}}$. For any $\Lambda>0$
  and $T>0$, there exists a $C<\infty$ so that for all $\lambda\in[0,\Lambda]$,
  $t\in(0,T]$, $f\in\mathcal{C}_{q_{\lambda}}(\mathbb{R})$, and $x\in\mathbb{R}$,
  we have
  \begin{equation}
    |\partial_{x}G_{t}*f(x)|\le Ct^{-\frac{1}{2}}\e^{Ct\langle x\rangle^{2(\beta-1)}}q_{\lambda}(x)\|f\|_{\mathcal{C}_{\q_{\lambda}}(\mathbb{R})}.\label{eq:crazyweightbound}
  \end{equation}
\end{lem}

\begin{rem}
  This also holds for $\beta\in(1,3/2)$. The argument is similar but
  not identical, and is not needed in this paper, so we omit it.
\end{rem}

\begin{proof}
  Throughout the proof, $C$ denotes a positive constant that depends
  only on $\Lambda$ and $T$. It may change from line to line. We may
  assume without loss of generality that $\|f\|_{\mathcal{C}_{\q_{\lambda}}(\mathbb{R})}=1$.
  We begin by noting that $|y|\le\exp (y^{2}/4)$ for all $y\in\mathbb{R}$,
  so
  \[
    |\partial_{x}G_{t}|\le\sqrt{2}t^{-\frac{1}{2}}G_{2t}.
  \]
  Therefore, we have
  \begin{equation}
    |\partial_{x}G_{t}*f(x)|\le\sqrt{2}t^{-\frac{1}{2}}G_{2t}*\q_{\lambda}(x),\label{eq:dxGbd}
  \end{equation}
  so it remains to bound $|G_{2t}*\q_{\lambda}(x)|$. The function $w(s,\cdot) = G_{s}*q_{\lambda}$
  solves the heat equation with initial condition $\q_{\lambda}$, so
  we can bound it from above by constructing a supersolution $v$ with
  the same initial condition. Set
  \begin{equation}
    v(s,x)=\exp\left(As\langle x\rangle^{2(\beta-1)}+Bs^{\frac{\beta}{2-\beta}}\right) q_{\lambda}(x)\label{eq:vdef}
  \end{equation}
  for constants $A,B>0$ to be determined. Then we have
  \[
    \partial_{s}v\ge\left(A\langle x\rangle^{2(\beta-1)}+Bs^{\frac{2(\beta-1)}{2-\beta}}\right)v
  \]
  and
  \[
    \partial_{xx}v\le8\left(\lambda^{2}+\lambda+A(\lambda+1)s\langle x\rangle^{\beta-2}+A^{2}s^{2}\langle x\rangle^{2(\beta-2)}\right)\langle x\rangle^{2(\beta-1)}v.
  \]
  Comparing these, in order for $v$ to be a supersolution for the heat
  equation, i.e. to satisfy
  \begin{equation}
    \partial_{s}v\ge\frac{1}{2}\partial_{xx}v,\label{eq:vsupersoln}
  \end{equation}
  it suffices to choose $A$ and $B$ so that
  \begin{equation}
    4\left(\lambda^{2}+\lambda+A(\lambda+1)s\langle x\rangle^{\beta-2}+A^{2}s^{2}\langle x\rangle^{2(\beta-2)}\right)\le A+Bs^{\frac{2(\beta-1)}{2-\beta}}\langle x\rangle^{-2(\beta-1)}.\label{eq:whatyouneedaboutAandB}
  \end{equation}
  To accomplish this, let $A=4(\lambda^{2}+\lambda)+2$ and $\xi=s\langle x\rangle^{\beta-2}$.
  Then for \cref{eq:whatyouneedaboutAandB} to hold, it suffices to choose
  $B$ so that
  \begin{equation}
    4(\lambda+1)A\xi+4A^{2}\xi^{2}\le2+B\xi^{\frac{2(\beta-1)}{2-\beta}}\label{eq:whatyouneedaboutB}
  \end{equation}
  for all $\xi\ge0$. When $\xi\le1/(4A(\lambda+1))$, the left side
  of \cref{eq:whatyouneedaboutB} is bounded by $2$, and the inequality
  holds regardless of $B$. Moreover, $2(\beta-1)/(2-\beta)\ge2$
  since $\beta\ge3/2$, so the right side of \cref{eq:whatyouneedaboutB}
  grows at least as fast as the left as $\xi\to+\infty$. Thus, there
  exists $B=B(\Lambda)$ sufficiently large that \cref{eq:whatyouneedaboutB}
  holds also when $\xi\ge1/(4A(\lambda+1))$. With these values of $A$
  and $B$, $v$ satisfies \cref{eq:vsupersoln} and $v(0,\cdot)\equiv\q_{\lambda}\equiv w(0,\cdot)$.
  Thus $w(s,x)\le v(s,x)$ for all $s\ge0$ and $x\in\mathbb{R}$ by
  the comparison principle for the heat equation. The bound \cref{eq:crazyweightbound}
  then follows from \cref{eq:dxGbd} and \cref{eq:vdef}.
\end{proof}

\section{Classical solutions are mild\label{appendix:classicalismild}}

In this appendix we prove a converse to \cref{lem:mildisclassical}.
\begin{lem}[Classical solutions are mild]\label{lem:classicalismild}
Suppose that $m\in(0,1)$, $L\in(0,\infty]$,
  $T>0$ and $\theta^{[L]}\in\tilde{\mathcal{Z}}_{m,T}$ is a classical
  solution to \cref{eq:thetaLproblem}. Then $\theta^{[L]}$ satisfies
  \cref{eq:mildformulation}.
\end{lem}

\begin{proof}
  Let $f=-\partial_{x}(\theta^{[L]}+\psi^{[L]})^{2}$ denote the nonlinear
  forcing in \cref{eq:thetaLproblem}. Fix $t\in(0,T]$ and define, for
  $s\in[0,t]$,
  \begin{align*}
    \Theta(s,x) & =G_{s}*\theta^{[L]}(t-s,x), & F(s,x) & =G_{s}*f(t-s,x).
  \end{align*}
  It is clear that $\Theta$ is twice-differentiable in space. We claim
  that it is also continuous in $s$, pointwise in $x$.
  Fix $\ell>m$, $(s,x)\in [0,t]\times\mathbb{R}$, and a sequence $(s_{n})_{n \in \mathbb{N}}\subset(0,t)$ converging to $s$.
  If $s > 0$, we use
  \begin{equation}
    |\Theta(s_{n},x)-\Theta(s,x)|\le|(G_{s_{n}}-G_{s})*\theta^{[L]}(t-s_{n},\cdot)(x)|+|G_{s}*[\theta^{[L]}(t-s_{n},\cdot)-\theta^{[L]}(t-s,\cdot)](x)|.\label{eq:ThessnTheta}
  \end{equation}
  Then $\|G_{s_{n}}-G_{s}\|_{L_{\p_{-\ell}}^{1}(\mathbb{R})}\to0$.
  Since $\theta^{[L]}$ is uniformly bounded in $L_{\p_{\ell}}^{\infty}(\mathbb{R})$,
  we have
  \[
    |(G_{s_{n}}-G_{S})*\theta^{[L]}(t-s_{n},\cdot)(x)|\le\|G_{s_{n}}-G_{s}\|_{L_{\p_{-\ell}}^{1}(\mathbb{R})}\|\theta^{[L]}(t-s_{n},\cdot)\|_{L_{\p_{\ell}}^{\infty}(\mathbb{R})}\to0.
  \]
  On the other hand,
  \[
    G_{s}*[\theta^{[L]}(t-s_{n},\cdot)-\theta^{[L]}(t-s,\cdot)](x)=\int_{\mathbb{R}}G_{s}(x-y)[\theta^{[L]}(t-s_{n},y)-\theta^{[L]}(t-s,y)]\,\dif y\to0
  \]
  by the weak-$*$ continuity of $\theta^{[L]}$, since $G_{s}\in L_{\p_{-\ell}}^{1}(\mathbb{R})$.
  By \cref{eq:ThessnTheta}, $|\Theta(s_{n},x)-\Theta(s,x)|\to0$ as $n\to\infty$.

  Now suppose $s=0$.
  Then $G_{0}=\delta_{0}$ is singular,
  so we must argue differently. In this case, we are considering $\theta^{[L]}$
  near a time $t>0$, where it is continuous. Fix $\eps>0$, and consider
  the opposite decomposition
  \begin{equation}
    |\Theta(s_{n},x)-\Theta(s,x)|\le|G_{s_{n}}*[\theta^{[L]}(t-s_{n},\cdot)-\theta^{[L]}(t,\cdot)](x)|+|(G_{s_{n}}-\delta_{0})*\theta^{[L]}(t,\cdot)(x)|.\label{eq:oppositedecomp}
  \end{equation}
  Since $\theta^{[L]}\in\mathcal{C}_{\mathrm{b}}((0,S];\mathcal{C}_{\p_{\ell}}(\mathbb{R}))$,
  there exists a $\delta>0$ such that
  \[
    \|\theta^{[L]}(t-s_{n},\cdot)-\theta^{[L]}(t,\cdot)\|_{\mathcal{C}_{\p_{\ell}}(\mathbb{R})}<\eps
  \]
  when $s_{n}<\delta$. Now $G_{s_{n}}*\p_{\ell}\le C_\ell \p_{\ell}$, so
  \[
    |G_{s_{n}}*[\theta^{[L]}(t-s_{n},\cdot)-\theta^{[L]}(t,\cdot)](x)|\le C_\ell \eps\p_{\ell}(x).
  \]
  For the second term of \cref{eq:oppositedecomp}, we use the fact that
  $\theta^{[L]}(t,\cdot)\in\mathcal{C}_{\p_{\ell}}(\mathbb{R})$, so
  by \cref{lem:weightedhkcts}, we have pointwise convergence and $|(G_{s_{n}}-\delta_{0})*\theta^{[L]}(t,\cdot)(x)|\to0$.
  Thus by \cref{eq:oppositedecomp}, $|\Theta(s_{n},x)-\Theta(0,x)|\to0$
  as $n\to\infty$, as desired.

  Next, \cref{eq:time-deriv-control} implies that $\Theta$ is differentiable in $s$ on $(0, t)$ and that
  \begin{align*}
    \partial_s \Theta(s, x) &= (\partial_s G_s) \ast \theta^{[L]}(t - s, x) - G_s \ast \partial_t \theta^{[L]}(t - s, x)\\
                            &= \frac 1 2 \partial_{xx} G_s \ast \theta^{[L]}(t - s, x) - \frac 1 2 G_s \ast \partial_{xx}\theta^{[L]}(t - s, x) - F(s, x).
  \end{align*}
  Now $\theta^{[L]}(t - s, \cdot)$ is a tempered distribution, so we may exchange differentiation and convolution to find
  \begin{equation*}
    \partial_s \Theta(s, x) = -F.
  \end{equation*}
  The continuity of $\Theta$ in time ensures that
  \[
    \Theta(t,x)-\Theta(0,x)=\int_{0}^{t}\partial_{s}\Theta(s,x)\,\dif s = - \int_0^t F(s, x) \,\dif s
  \]
  for all $x\in\mathbb{R}$.
  After a change of variables in the integral, this is simply \cref{eq:mildformulation}.
\end{proof}

\section{Elementary probabilistic and analytic lemmas\label{appendix:symmetrylemma}}

In this appendix we prove some elementary technical lemmas that were deferred until this point to avoid disrupting the flow of the paper.

Several symmetry arguments in the paper relied on the following lemma.
\begin{lem}
  \label{lem:symmetrylemma}
  Let $X_{1}$ and $X_{2}$ be random variables
  such that $X_{1}\overset{\mathrm{law}}{=}X_{2}$ and $\mathbb{E}(X_{2}-X_{1})^{-}>-\infty$.
  Then $\mathbb{E}|X_{2}-X_{1}|<\infty$ and $\mathbb{E}(X_{2}-X_{1})=0$.
\end{lem}

\begin{proof}
  Of course the claim is obvious if $\mathbb{E}|X_{i}|<\infty$, but
  for our applications in the paper it will be convenient to not have
  to assume this. Define $f_{M}(x)=\max\{\min\{x,M\},-M\}$ and note
  that $f_{M}$ is bounded and $1$-Lipschitz. The function $g_{M}(x)=f_{M}(x)-x$
  is also $1$-Lipschitz. Fix $R\in(0,\infty]$ and compute
  \begin{align}
    \mathbb{E}(|g_{M}(X_{2})-g_{M}(X_{1})| & \cdot\mathbf{1}\{X_{2}-X_{1}\le R\})\nonumber                                                                             \\
                                           & =\mathbb{E}(|g_{M}(X_{2})-g_{M}(X_{1})|\mathbf{1}\{\max\{|X_{1}|,|X_{2}|\}\ge M\}\mathbf{1}\{X_{2}-X_{1}\le R\})\nonumber \\
                                           & \le\mathbb{E}(|X_{2}-X_{1}|\mathbf{1}\{\max\{|X_{1}|,|X_{2}|\}\ge M\}\mathbf{1}\{X_{2}-X_{1}\le R\}).\label{eq:EYX1X}
  \end{align}
  For each $0<R<\infty$ fixed, the last expression in \cref{eq:EYX1X}
  goes to $0$ as $M\to\infty$ by the dominated convergence theorem,
  since $\mathbb{E}(|X_{2}-X_{1}|\mathbf{1}\{X_{2}-X_{1}\le R\})$ is
  finite since $(X_{2}-X_{1})^{+}\mathbf{1}\{X_{2}-X_{1}\le R\}$ is
  bounded and $\mathbb{E}(X_{2}-X_{1})^{-}$ is finite by assumption.
  Therefore, we have
  \begin{align}
    0 & =\lim_{M\to\infty}\mathbb{E}(|g_{M}(X_{2})-g_{M}(X_{1})|\cdot\mathbf{1}\{X_{2}-X_{1}\le R\})\nonumber                                                                                \\
      & =\lim_{M\to\infty}\mathbb{E}(|f_{M}(X_{2})-f_{M}(X_{1})-(X_{2}-X_{1})|\cdot\mathbf{1}\{X_{2}-X_{1}\le R\})\nonumber                                                                  \\
      & \ge\limsup_{M\to\infty}\left|\mathbb{E}((f_{M}(X_{2})-f_{M}(X_{1}))\mathbf{1}\{X_{2}-X_{1}\le R\})-\mathbb{E}((X_{2}-X_{1})\mathbf{1}\{X_{2}-X_{1}\le R\})\right|.\label{eq:limis0}
  \end{align}
  Since $X_{1}\overset{\mathrm{law}}{=}X_{2}$, we have $\mathbb{E}[f_{M}(X_{2})-f_{M}(X_{1})]=0$.
  Also, $f_{M}(x)$ is monotone in $x$, so
  \[
    \mathbb{E}[(f_{M}(X_{2})-f_{M}(X_{1}))\mathbf{1}\{X_{2}-X_{1}\le R\}]=-\mathbb{E}[(f_{M}(X_{2})-f_{M}(X_{1}))\mathbf{1}\{X_{2}-X_{1}>R\}]\le0
  \]
  for all $M,R\in(0,\infty)$. Using this in \cref{eq:limis0} gives
  \[
    \mathbb{E}((X_{2}-X_{1})\mathbf{1}\{X_{2}-X_{1}\le R\})\le0
  \]
  for all $R\in(0,\infty)$. Taking $R\to+\infty$ and using the monotone
  convergence theorem and the assumption $\mathbb{E}(X_{2}-X_{1})^{-}>-\infty$
  yields $\mathbb{E}(X_{1}-X_{2})\le0$, so in particular $\mathbb{E}|X_{1}-X_{2}|<\infty$.
  Thus we can take $R=+\infty$ in \cref{eq:EYX1X} and again use the
  dominated convergence theorem to obtain \cref{eq:limis0} with $R=+\infty$,
  namely
  \[
    0\ge\limsup_{M\to\infty}\left|\mathbb{E}[f_{M}(X_{2})-f_{M}(X_{1})]-\mathbb{E}(X_{2}-X_{1})\right|=|\mathbb{E}(X_{2}-X_{1})|,
  \]
  and so $\mathbb{E}(X_{2}-X_{1})=0$ as claimed.
\end{proof}

The following lemma will be used in the proof of \cref{lem:gcts} below.
\begin{lem}
  \label{lem:howmanyderivs}Let $\eta\in\mathcal{C}^{3}(I)$ for some  closed interval $I \subset \R$ and define
  $A=\|\eta\|_{\mathcal{C}^{3}(I)}$. If $x_{1}\ne x_{2}\in I$
  satisfy $\eta(x_{1})=\eta(x_{2})$ and $\eta'(x_{1}),\eta'(x_{2})>\eps$
  or $\eta'(x_{1}),\eta'(x_{2})<-\eps$, then $|x_{1}-x_{2}|\ge\sqrt{2\eps/A}.$
\end{lem}

\begin{proof}
  Without loss of generality, we may assume that $x_2 > x_1$ and $\eta'(x_{1}),\eta'(x_{2})>\eps$.
  Let $\delta=x_{2}-x_{1}$. By Rolle's theorem, there
  is a $y\in(x_{1},x_{2})$ so that $\eta'(y)=0$. By the mean value
  theorem, there exist $z_{1}\in(x_{1},y)$ and $z_{2}\in(y,x_{2})$
  so that
  \begin{align*}
    \eta''(z_{1}) & \le-\delta^{-1}\eps, & \eta''(z_{2}) & \ge\delta^{-1}\eps.
  \end{align*}
  By another application of the mean value theorem, there exists a $w\in(z_{1},z_{2})$
  so that
  \[
    \eta'''(w)\ge2\delta^{-2}\eps.
  \]
  This means that $2\delta^{-2}\eps\le A$, so $\delta\ge\sqrt{2\eps/A}$,
  as claimed.
\end{proof}

The following lemma was used in the proof of \cref{lem:Fdecreasing}.
\begin{lem}
  Suppose that $\zeta \geq 0$ is in the Schwartz class and $\eta$ is a smooth
  function all of whose derivatives have at most polynomial growth at
  infinity. Define
  \[
    g(\lambda)=\sum_{\substack{y\in\eta^{-1}(\lambda)\\
        \eta'(y)\ne0
      }
    }|\eta'(y)|\zeta(y).
  \]
  Then $g$ is a continuous function of $\lambda$.\label{lem:gcts}
\end{lem}

\begin{proof}
  Let $\zeta_{k} \geq 0$ be smooth such that $\supp\zeta_{k}\subset[k-1,k+2]$
  and $\sum\limits_{k\in\mathbb{Z}}\zeta_{k}=\zeta$. Let
  \begin{equation}
    g_{k}(\lambda)=\sum_{\substack{y\in\eta^{-1}(\lambda)\\
        \eta'(y)\ne0
      }
    }|\eta'(y)|\zeta_{k}(y).\label{eq:gkdef}
  \end{equation}
  We first show that $g_{k}$ is continuous. Let $A_{k}=\|\eta\|_{\mathcal{C}^{3}([k-1,k+2])}+\|\zeta_{k}\|_{\mathcal{C}^{0}(\R)}+\|\eta'\zeta_k\|_{\mathcal{C}^1(\R)}$ + 1.
  Define
  \begin{equation*}
    S_{k,\ell}(\lambda)=\{y\in(k-1,k+2) \mid \eta(y)=\lambda,|\eta'(y)|\in [2^{-\ell},2^{-\ell+1})\}.
  \end{equation*}
  \cref{lem:howmanyderivs} implies
  \begin{equation}
    \label{eq:S-size}
    |S_{k,\ell}(\lambda)|\le 2^3 A_{k}^{1/2}2^{\ell/2},
  \end{equation}
  so
  \begin{equation}
    \label{eq:smallterms}
    \sum_{y\in S_{k,\ell}(\lambda)}|\eta'(y)|\zeta_{k}(y)\le 2^{-\ell + 1} \|\zeta_k\|_{\mathcal{C}^0(\R)} |S_{k, \ell}(\lambda)| \leq 2^4 A_{k}^{3/2} 2^{-\ell/2}
  \end{equation}
  for all $\lambda$.

  Now fix $\eps>0$ and choose $\ell$ so large that $2^9 A_k^{3/2} 2^{-\ell/2} < \eps.$
  Define
  \begin{align*}
    T_{k, \ell}^{+}(\lambda) & =\{y\in (k - 1, k + 2) \mid \eta(y)=\lambda, \eta'(y)\ge2^{-\ell}\},  \\
    T_{k, \ell}^{-}(\lambda) & =\{y\in (k - 1, k + 2) \mid \eta(y)=\lambda, \eta'(y)\le-2^{-\ell}\}.
  \end{align*}
  Suppose that $\lambda_{1}<\lambda_{2}$ satisfy $\lambda_{2}-\lambda_{1} < 2^{-2\ell - 2}A_k^{-1}$.
  Take $x\in T_{k, \ell}^{+}(\lambda_{1})$.
  On the interval ${[x - 2^{-\ell-1} A_k^{-1}, x + 2^{-\ell-1} A_k^{-1}]}$, we must have
  \begin{equation*}
    \eta' \geq \eta'(x) - \|\eta''\|_{\mathcal{C}^0([k-1, k+2])} 2^{-\ell-1} A_k^{-1} \geq 2^{-\ell - 1}.
  \end{equation*}
  Thus
  \begin{equation*}
    \eta(x + 2^{-\ell-1} A_k^{-1}) \geq \lambda_{1} + 2^{-\ell - 1} \cdot 2^{-\ell-1} A_k^{-1} > \lambda_2.
  \end{equation*}
  Since $\eta$ is continuous, there exists $y \in T_{k, \ell+1}^{+}(\lambda_{2})\cap(x,x+ 2^{-\ell-1} A_k^{-1}]$.
  We say that $y$ is ``paired to $x$.''
  Notice that $y - x < 2^{-\ell-1} A_k^{-1} \geq 2^{-\ell - 1}$ while $\eta \leq \lambda_1$ on $[x - 2^{-\ell-1} A_k^{-1}, x]$, so $y$ is not paired to any other $x \in T_{k, \ell}^{+}(\lambda_{1})$.
  Thus to each $x \in T_{k, \ell}^{+}(\lambda_{1})$ we have paired a \emph{unique} $y \in T_{k, \ell+1}^{+}(\lambda_{2})$.
  Also,
  \begin{equation*}
    |\eta'(x)|\zeta_{k}(x)-|\eta'(y)|\zeta_{k}(y) \leq \|\eta' \zeta_k\|_{\mathcal{C}^1(\R)} 2^{-\ell-1} A_k^{-1} \leq 2^{-\ell - 1}.
  \end{equation*}

  Now consider the difference
  \begin{equation}
    \label{eq:T-ell-diff-prelim}
    \sum_{x\in T_{k, \ell}^{+}(\lambda_{1})}|\eta'(x)|\zeta_{k}(x)-\sum_{y\in T_{k, \ell}^{+}(\lambda_{2})}|\eta'(y)|\zeta_{k}(y).
  \end{equation}
  Each $x \in T_{k, \ell - 1}^+(\lambda_1)$ is paired to a unique $y \in T_{k, \ell}^{+}(\lambda_{2})$, and the corresponding terms' difference is at most $2^{-\ell - 1}$.
  On the other hand, if $x \in T_{k, \ell}^+(\lambda_1) \setminus T_{k, \ell - 1}^+(\lambda_1)$, we have
  \begin{equation*}
    |\eta'(x)|\zeta_{k}(x) \leq 2^{-\ell + 1} A_k.
  \end{equation*}
  Decomposing \cref{eq:T-ell-diff-prelim} into these two cases, we obtain
  \begin{equation*}
    \label{eq:T-ell-diff}
    \sum_{x\in T_{k, \ell}^{+}(\lambda_{1})}|\eta'(x)|\zeta_{k}(x)-\sum_{y\in T_{k, \ell}^{+}(\lambda_{2})}|\eta'(y)|\zeta_{k}(y) \leq \sum_{x \in T_{k, \ell}^+(\lambda_1)} (2^{-\ell - 1} + 2^{\ell + 1}A_k) \leq 2^2 A_k 2^{-\ell} |T_{k, \ell}^+(\lambda_1)|.
  \end{equation*}
  Now \cref{eq:S-size} implies
  \begin{equation*}
    |T_{k, \ell}^+(\lambda_1)| \leq \sum_{\ell' \leq \ell} |S_{k, \ell'}(\lambda_1)| \leq 2^3 A_k^{1/2} 2^{\ell/2} \sum_{m \geq 0} 2^{-m/2} \leq 2^5 A_k^{1/2} 2^{\ell/2}.
  \end{equation*}
  Therefore,
  \begin{equation*}
    \sum_{x\in T_{k, \ell}^{+}(\lambda_{1})}|\eta'(x)|\zeta_{k}(x)-\sum_{y\in T_{k, \ell}^{+}(\lambda_{2})}|\eta'(y)|\zeta_{k}(y)\le 2^2 A_k 2^{-\ell} |T_{k, \ell}^+(\lambda_1)| \leq 2^7 A_k^{3/2} 2^{-\ell/2}.
  \end{equation*}
  Symmetric arguments show that in fact
  \begin{equation}
    \label{eq:T-control}
    \Bigg|\sum_{x\in T_{k, \ell}^{\pm}(\lambda_{1})}|\eta'(x)|\zeta_{k}(x)-\sum_{y\in T_{k, \ell}^{\pm}(\lambda_{2})}|\eta'(y)|\zeta_{k}(y)\Bigg|\le 2^7 A_{k}^{3/2} 2^{-\ell/2}.
  \end{equation}
  Now consider $g_k(\lambda_1) - g_k(\lambda_2)$.
  We divide \cref{eq:gkdef} into terms in $T_{k, \ell}^\pm(\lambda_i)$ or $S_{k, \ell'}(\lambda_i)$ for $\ell' > \ell$ and $i \in \{1, 2\}$.
  We pair terms in $T_{k, \ell}^\pm(\lambda_1)$ with those in $T_{k, \ell}^\pm(\lambda_2)$ to take advantage of the cancellation in \cref{eq:T-control}.
  We treat the terms in $S_{k, \ell'}(\lambda_i)$ as error and use \cref{eq:smallterms}.
  Then
  \begin{equation*}
    |g_k(\lambda_1) - g_k(\lambda_2)| \leq 2^8 A_k^{3/2} 2^{-\ell/2} + 2^5 A_k^{3/2} \sum_{\ell' > \ell} 2^{-\ell'/2} \leq 2^9 A_k^{3/2} 2^{-\ell/2} < \eps.
  \end{equation*}
  It follows that $g_k$ is uniformly continuous.

  We now wish to sum over $k$ to conclude the same for $g$.
  To do so, we bound $g_k$.
  Let ${\ell_k = -\log_2 A_k + 1}$.
  Then $S_{k, \ell}(\lambda) = \emptyset$ for all $\lambda \in \R$ and $\ell < \ell_k$.
  Hence \cref{eq:S-size} implies
  \begin{equation*}
    g_k(\lambda) = \sum_{\ell \geq \ell_k} \sum_{y \in S_{k, \ell}(\lambda)} |\eta'(y)| \zeta_k(y) \leq 2\|\zeta_k\|_{\mathcal{C}^0(\R)} \sum_{\ell \geq \ell_k} 2^{-\ell} |S_{k, \ell}(\lambda)| \leq 2^4 \|\zeta_k\|_{\mathcal{C}^0(\R)} A_k^{1/2} 2^{-\ell_k/2} \sum_{m \geq 0} 2^{-m/2}.
  \end{equation*}
  Using the definition of $\ell_k$, this yields
  \begin{equation*}
    g_k(\lambda) \leq 2^5 \|\zeta_k\|_{\mathcal{C}^0(\R)} A_k.
  \end{equation*}
  By hypothesis, there exists $C \geq 1$ independent of $k$ such that $A_k \leq C \langle k \rangle^C$ for all $k \in \Z$.
  But $\zeta_k \leq \zeta$ decays super-polynomially as $|k| \to \infty$.
  It follows that $\|g_k\|_{\mathcal{C}^0(\R)}$ itself decays super-polynomially in $k$.
  Therefore $g=\sum\limits_{k\in\mathbb{Z}}g_{k}$ is an absolutely and uniformly convergent sum of uniformly continuous functions and hence is (uniformly) continuous.
\end{proof}

\bibliography{burgers}
\bibliographystyle{hplain-ajd}

\end{document}